\documentclass[12pt]{amsart}
\usepackage{amssymb}
\usepackage[noadjust]{cite}
\usepackage{array}
\usepackage{booktabs}
\usepackage{mdwtab}
\usepackage{mathtools}
\usepackage{lmodern}
\usepackage{microtype}
\usepackage{ifpdf}
\usepackage[T1]{fontenc}
\usepackage[utf8]{inputenc}
\usepackage{hyphenat}
\usepackage{enumitem}
\usepackage{mathabx}
\usepackage{color}
\usepackage[pdftex]{graphicx}
\usepackage{float}
\usepackage[pdftex,margin=1in]{geometry}
\usepackage{xr-hyper}
\usepackage[bookmarks=true, bookmarksopen=true,%
    bookmarksdepth=3,bookmarksopenlevel=2,%
    colorlinks=true,%
    linkcolor=blue,%
    citecolor=blue,%
    filecolor=blue,%
    menucolor=blue,%
    urlcolor=blue]{hyperref}
\hypersetup{pdftitle={Ahlfors-regular conformal dimension and energies of graph maps}}
\hypersetup{pdfauthor={Kevin M. Pilgrim and Dylan P. Thurston}}
\usepackage{url}
\usepackage[normalem]{ulem} 
\usepackage{tikz}
\usetikzlibrary{arrows.meta}
\usetikzlibrary{decorations,decorations.pathmorphing,decorations.markings}
\usetikzlibrary{hobby}
\usetikzlibrary{cd}
\tikzstyle{every picture}=[> = To]
\tikzset{cdlabel/.style={execute at begin node=$\scriptstyle,execute at end node=$}}
\tikzset{implication/.style={double equal sign distance, -implies}}
\tikzset{biimplication/.style={double equal sign distance, implies-implies}}
\tikzset{markdir/.style={postaction={decorate,decoration={markings,mark=at position 0.55 with {\arrow{>[scale=1.4]}}}}}}

\makeatletter
\newcommand\mi@kern[1]{%
  \settowidth\@tempdima{$\mi@obj^{#1}$}
  \kern-\@tempdima
  #1
  \settowidth\@tempdima{$\mi@obj$}
  \kern\@tempdima
}

\newtoks\mi@toksp
\newtoks\mi@toksb
\DeclareRobustCommand{\manyindices}[5]{
  \def\mi@obj{#5}
  \mi@toksp\expandafter{\mi@kern{#2}}
  \mi@toksb\expandafter{\mi@kern{#1}}
  \@mathmeasure4\textstyle{#5_{#1}^{#2}}
  \@mathmeasure6\textstyle{#5_{#3}^{#4}}
  \dimen0-\wd6 \advance\dimen0\wd4
  \@mathmeasure8\textstyle{\hphantom{{}_{#1}^{#2}}#5^{\the\mi@toksp#4}_{\the\mi@toksb#3}}
  \hbox to \dimen0{}{\kern-\dimen0\box8}
}
\makeatother 

\newread\testin

\def\mathcenter#1{\vcenter{\hbox{$#1$}}}
\def\grapha#1{\includegraphics{#1}}

\def\mfig#1{\mathcenter{\grapha{#1}}}

\newcommand*{\doi}[1]{\href{https://doi.org/#1}{doi:\allowbreak#1}}

\renewcommand{\colon}{\nobreak\mskip2mu\mathpunct{}\nonscript
  \mkern-\thinmuskip{:}\allowbreak\mskip6muplus1mu\relax}

\ifpdf
  \let\textalt\texorpdfstring
\else
  \newcommand{\textalt}[2]{#1}
\fi

\newcommand{\RR}{\mathbb R}
\newcommand{\CC}{\mathbb C}

\newcommand{\ZZ}{\mathbb Z}

\newcommand{\cS}{\mathcal{S}}
\newcommand{\cT}{\mathcal{T}}
\newcommand{\cH}{\mathcal{H}}

\newcommand{\cE}{\mathcal{E}}
\newcommand{\cG}{\mathcal{G}}
\newcommand{\cJ}{\mathcal{J}}

\newcommand{\cU}{\mathcal{U}}
\newcommand{\cV}{\mathcal{V}}

\newcommand{\whe}{\widehat{e}}
\newcommand{\wte}{{\widetilde{e}}}

\newcommand{\wtU}{\widetilde{U}}
\newcommand{\wtV}{\widetilde{V}}

\newcommand{\wtY}{\widetilde{Y}}

\newcommand{\co}{\colon}

\renewcommand{\epsilon}{\varepsilon}
\newcommand{\abs}[1]{\lvert #1 \rvert}
\newcommand{\norm}[1]{\lVert #1 \rVert}

\DeclareMathOperator*{\esssup}{ess\,sup}

\theoremstyle{plain}
\numberwithin{equation}{section}
\newtheorem{proposition}[equation]{Proposition}
\newtheorem{lemma}[equation]{Lemma}
\newtheorem{corollary}[equation]{Corollary}
\newtheorem{conjecture}[equation]{Conjecture}

\newtheorem{theorem}{Theorem}

\newtheorem{citethm}[equation]{Theorem}

\theoremstyle{definition}
\newtheorem{definition}[equation]{Definition}
\newtheorem{convention}[equation]{Convention}

\newtheorem{question}[equation]{Question}

\newtheorem{claim}[equation]{Claim}

\theoremstyle{remark}

\newtheorem{remark}[equation]{Remark}

\theoremstyle{plain}
\newtheorem{taggedthmx}{Theorem}
\newenvironment{taggedthm}[1]
 {\renewcommand\thetaggedthmx{#1}\taggedthmx}
 {\endtaggedthmx}

\newcommand{\C}{\CC}

\newcommand{\rs}{\widehat{\C}}

\renewcommand{\th}{^\mathrm{th}}

\DeclareMathOperator{\modulus}{mod}
\DeclareMathOperator{\EL}{EL} 
\DeclareMathOperator{\SF}{SF} 
\DeclareMathOperator{\Fill}{Fill} 

\DeclareMathOperator{\Edge}{Edge}
\newcommand{\Edges}{\Edge}

\newcommand{\CCa}{\widehat{\CC}}

\newcommand{\id}{\mathrm{id}}
\DeclareMathOperator{\ARCdim}{ARCdim}
\DeclareMathOperator{\hdim}{hdim}

\DeclareMathOperator{\diam}{diam}
\DeclareMathOperator{\size}{size}
\DeclareMathOperator{\tracesize}{tracesize}
\DeclareMathOperator{\IMG}{IMG}
\DeclareMathOperator{\systole}{systole}

\newcommand{\SFto}{\overleftarrow{\SF}}

\newcommand{\shortseq}[5]{#1 \overset{#2}{\longrightarrow} #3 \overset{#4}{\longrightarrow} #5}

\newcommand{\wt}[1]{\widetilde{#1}}

\newcommand{\oE}{\overline{E}{}}
\newcommand{\oN}{\overline{N}{}}

\newcommand{\qdual}{q^\vee}

\colorlet{mygreen}{green!90!blue}
\colorlet{dark-green}{black!40!mygreen}
\colorlet{dark-red}{black!30!red}
\colorlet{dark-blue}{black!20!blue}

\graphicspath{{draws/}{figs/}}

\begin{document}
\title{Ahlfors-regular conformal dimension and energies of graph maps}

\author[Pilgrim]{Kevin M.~Pilgrim}
\address{Indiana University\\
         831 E. Third St.,
         Bloomington, Indiana 47405\\
         USA}
\email{pilgrim@iu.edu}

\author[Thurston]{Dylan P.~Thurston}
\address{Indiana University\\
         831 E. Third St.,
         Bloomington, Indiana 47405\\
         USA}
\email{dpthurst@iu.edu}
\date{June 1, 2025}

\begin{abstract}
  For a hyperbolic rational map~$f$ with connected Julia set, we give
  upper and lower bounds on 
  the Ahlfors-regular conformal dimension of its Julia set~$J_f$ from
  a family of energies of associated graph maps. Concretely, the
  dynamics of $f$ is faithfully encoded by a pair of maps
  $\pi, \phi\co G_1 \rightrightarrows G_0$ between finite graphs that satisfies a
  natural expanding condition. Associated to this combinatorial data,
  for each $q \geq 1$, is a numerical invariant $\oE^q[\pi,\phi]$, its
  asymptotic $q$-conformal energy. We
  show that the Ahlfors-regular conformal dimension of $J_f$ is
  contained in the interval where $\oE^q=1$.

  Among other applications, we give two families of quartic rational
  maps with Ahlfors-regular conformal dimension approaching~$1$
  and~$2$, respectively.
\end{abstract}

\subjclass[2020]{Primary 37F10; Secondary 30L10, 20E08}

\maketitle

\setcounter{tocdepth}{1}

\tableofcontents

\section{Introduction}
\label{sec:intro}

\subsection{Motivation}
\label{subsec:motivation}
The iterates of a rational function~$f$ define a holomorphic dynamical
system on the Riemann sphere $\widehat{\mathbb{C}}$.
Its Julia set, typically fractal, may be defined as the smallest set
$J_f$ satisfying $J_f=f^{-1}(J_f)=f(J_f)$ and $\#J_f \geq 3$.

For instance, figure~\ref{fig:barycentric_carpet} shows the Julia set of the rational function $f(z)=\frac{4}{27}\frac{(z^2-z+1)^3}{(z(z-1))^2}$.
It turns out that as a topological space, this $J_f$ is a \emph{Sierpiński
  carpet}---the complement in the sphere of a countable collection of
Jordan domains whose closures are disjoint and whose diameters tend to
zero.
\begin{figure}
\centerline{\includegraphics[width=2.75in]{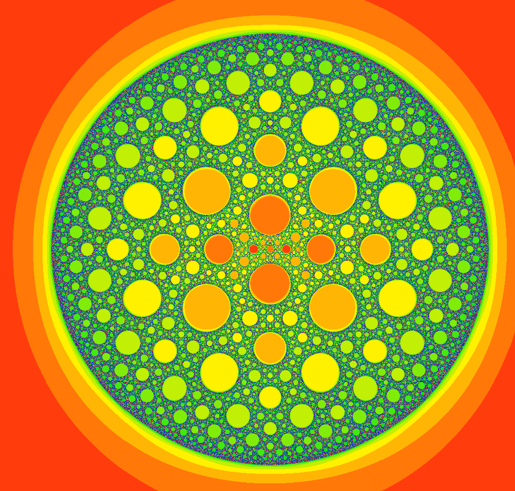}}
\caption{The Julia set of $f$.}
\label{fig:barycentric_carpet}
\end{figure}
With the spherical metric inherited from the round metric on $\widehat{\mathbb{C}}$, the Julia set $J_f$ becomes a metric space. As a dynamical system, the map $f$ is \emph{hyperbolic}---each critical point converges to an
attracting cycle---and \emph{critically finite}---the orbits of the critical points are finite.  Hyperbolicity is equivalent to the
condition that the restriction $f\co J_f \to J_f$ is an expanding self-covering map.  In this setting, hyperbolicity is a dynamical regularity condition that leads to strong metric consequences. As is visually evident, $J_f$ is
 \emph{approximately-self-similar} (see Definition \ref{defn:approx_self_similar}). 
A key invariant of $J_f$ is its Hausdorff dimension $\hdim(J_f)$.

D. Sullivan \cite[Thm. 4]{MR730296} showed that for any hyperbolic rational map $f$, upon setting
$q\coloneqq\hdim(J_f)$ and
$\cH^q$ the corresponding Hausdorff measure on $J_f$, one has 
$0 < q < 2$ and $0<\cH^q(J_f)<\infty$; see also \cite[Theorem 9.1.6., Corollary
9.1.7]{MR2656475} for more general statements. In particular, for any ball $B(x,r)$ with $x \in
J_f$, and any $r\leq \diam(J_f)$ we have $\cH^q(B(x,r)) \asymp r^q$,
with implicit constants independent of $x$ and $r$. This latter
condition is known as \emph{Ahlfors $q$-regularity}. 
If in addition $J(f)$ is a carpet, we have $1 <  \hdim(J_f) < 2$; see \cite{MR1005606, MR1879234} for these lower and upper bounds,
 respectively.

A homeomorphism between metric spaces is \emph{quasi-symmetric} if it
does not distort the roundness of balls too much; see
\S\ref{sec:ahlfors-regular}.  The \emph{Ahlfors-regular
  conformal gauge} $\mathcal{G}$ of $J_f$ is
the set of all metric measure spaces $(X, d, \mu)$ such that there
exists a quasi-symmetric homeomorphism $(J_f, d_{Spherical}) \to (X,d)$ and, for some $q>0$,
the measure $\mu$ is $q$-Ahlfors regular with respect to $d$; see
\cite{heinonen:analysis,kmp:ph:cxci}.   The regularity
assumption on $\mu$ implies that $\mu$ is comparable to the $q$-dimensional
Hausdorff measure $\mathcal{H}^q$ on $X$. The \emph{Ahlfors-regular conformal
  dimension} of $J_f$ is the infimum over such exponents $q$, i.e.,
\[
\ARCdim(J_f)\coloneqq\inf\{ \hdim(X) \mid (X,d,\mu) \in \mathcal{G}(J_f)\};
\]
see \cite{MT10:ConfDim} for an introduction.
For approximately self-similar carpets $X$, such as the Julia
set~$J_f$ of Figure~\ref{fig:barycentric_carpet}, we know
$1 < \ARCdim(X)<2$; see \cite{MR2667133} and \cite[Corollary 9.17]{MR2656475}.
Interest in conformal dimension stems in part from the
following. Limit sets of Kleinian groups acting on the Riemann sphere
and, more generally, boundaries of hyperbolic groups are another
source of approximately self-similar spaces. In that setting, the
conformal dimension, analogously defined, carries significant
information about the group; see \cite{Kleiner06:ICM}.

Hyperbolicity and the critically finite
property imply, by a rigidity result of W.~Thurston~\cite{DH1}, that
the geometry and dynamics on $J_f$ is
determined by topological data: the
conjugacy-up-to-isotopy class of~$f$, relative to its post-critical
set.  More precisely: if $g$ is another rational map, and if there are
orientation-preserving homeomorphisms $\phi_0, \phi_1\co
(\widehat{\mathbb{C}},P_f) \to (\widehat{\mathbb{C}},P_g)$ such that
$\phi_0\circ f=g\circ \phi_1$ on $P_f$ and $\phi_0$ is isotopic to
$\phi_1$ through homeomorphisms agreeing on $P_f$, then we may take
$\phi_0=\phi_1$ to be a Möbius transformation.  Hence the invariant
$\ARCdim(J_f)$ is determined from discrete data. 

For general hyperbolic critically finite rational maps with connected
Julia set, our main result, Theorem \ref{thm:crit-sandwich}, implies
an estimate for $\ARCdim(J_f)$ in terms of combinatorial data. In
concrete cases, by-hand computations with this data can yield
nontrivial rigorous upper and lower bounds. For the carpet Julia set
of Figure \ref{fig:barycentric_carpet}, such computations yield
$1.6309 \approx \frac{1}{1-\log_6 2} \le \ARCdim(J_f) \le
\frac{2}{1-\log_6 (10/13)} \approx 1.7445$.  See
\S\ref{sec:applications} for details.

\subsection{Combinatorial encoding} 
Our methods rely on a particular method of combinatorial encoding of
rational maps~\cite{Thurston20:Characterize}.  We choose a finite graph
$G_0$, called a \emph{spine}, onto which $\rs - P_f$
deformation retracts.  The homotopy type of~$G_0$ depends only on $\#P_f$.
Letting $G_1=f^{-1}(G_0) \subset \rs - P_f$, we obtain two graph maps $\pi,
\phi\co G_1 \rightrightarrows G_0$, where $\pi$ and $\phi$ are
respectively the
restrictions of $f$ and of the deformation retraction. The data $(\pi,
\phi)$ is a \emph{virtual endomorphism of
  graphs} (see Definition \ref{def:ve}) and is well-defined up to a notion of
homotopy equivalence; see \cite[Definition 2.2]{Thurston20:Characterize}.  We denote by $[\pi, \phi]$
the homotopy class
of $(\pi, \phi)$.  Figure~\ref{fig:barycentric_ve} illustrates the
data for the above map~$f$.
\begin{figure}
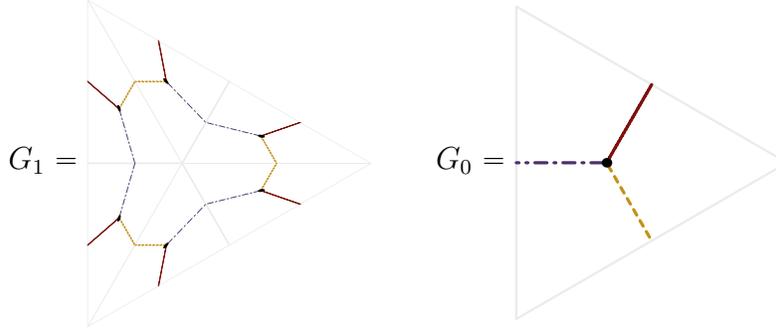

\[
G_1 = \mathcenter{\includegraphics[width=1.5in]{barycenter-tri-11}}
\qquad
G_0 = \mathcenter{\includegraphics[width=1.45in]{barycenter-tri-10}}
\]
\caption{Spines for the map~$f$, up to conjugacy-up-to-isotopy.
  Doubling the large Euclidean
  equilateral triangles over their boundaries gives two Riemann
  surfaces, each isomorphic to $\rs$.  The map $f$ sends each small
  triangle at left conformally to one of the two large triangles at right and implements
  barycentric subdivision. The set of vertices of the large triangles
  is~$P_f$.  Half of a spine~$G_0$ of $\rs - P_f$ is shown at right, and half its
  preimage $G_1$ under the piecewise-affine map homotopic to $f$
  is shown at left. The map $\pi$ is the covering map
  preserving colors, while the map $\phi$ (defined up to homotopy) is
  induced by the deformation retraction.}
\label{fig:barycentric_ve}
\end{figure}

For any iterate $n \in \mathbb{N}$, the Julia set of $f$ is the same
as that of $f^n$.  It follows from the  expanding nature of the
dynamics of $f$ that upon replacing $f$ by a suitable iterate, we may
assume the virtual endomorphism
$(\pi, \phi)$ constructed in the previous paragraph is
\emph{forward expanding} or, synonymously, \emph{backward contracting};
see Definition~\ref{def:contracting}. The critically finite
property implies that $G_1$ and $G_0$ are connected and $\phi$ is
surjective on fundamental group. This is a property we call
\emph{recurrence}; see Definition~\ref{def:recurrent}.
To summarize: to the dynamics of a critically finite hyperbolic rational map, we associate a forward-expanding recurrent virtual graph endomorphism $(\pi, \phi)$.  

It turns out (see \S \ref{sec:ve_and_dynamics}) that any
forward-expanding recurrent virtual graph endomorphism $(\pi,\phi)$
determines, via now-standard constructions, a dynamical system given by an expanding topological self-cover on a compact metrizable locally connected space $f\co \cJ \to \cJ$ (Theorem \ref{thm:lc}).  The topological conjugacy class of $f$ 
depends only on the homotopy class $[\pi,\phi]$; see \cite[Theorem 4.2]{ishii:smillie:homotopy-shadowing}.

Via other now-standard constructions,
there is an associated conformal gauge of Ahlfors-regular metrics
$\cG(\cJ[\pi,\phi])$ associated to $[\pi,\phi]$; see \S
\ref{sec:gauge}. If $[\pi,\phi]$ arises from a hyperbolic rational map
$f$, then the spherical metric on $J_f$ belongs to
$\cG(\cJ[\pi,\phi])$ (Proposition~\ref{prop:elevator}).

\subsection{Asymptotic \textalt{$q$}{q}-conformal energies}  
A virtual  endomorphism of graphs $(\pi,\phi)$ has, for each $q \in
[1,\infty]$, an associated \emph{asymptotic $q$-conformal graph energy} $\oE^q(\pi,
\phi)$, introduced by the second author
\cite{Thurston19:Elastic}.
We summarize some key points; see \S\ref{sec:energies} or the
references for more.  First,
$\oE^q(\pi, \phi)$ depends only on the homotopy class $[\pi, \phi]$,
so that these analytic quantities depend only on
combinatorial data. If $(\pi, \phi)$ arises from a rational map $f$,
different choices of spine lead to homotopic graph endomorphisms, so
that we may write unambiguously $\oE^q(f)$. In the general case, we
will also write
$\oE^q[\pi,\phi]$ to indicate the asymptotic energy depends only on
the homotopy class.
The inequality $\oE^\infty[\pi,\phi]<1$ holds if and only if some
iterate of $(\pi, \phi)$ is homotopic to a backward-contracting
virtual graph endomorphism. If $(\pi,\phi)$ arises from a hyperbolic
critically-finite rational map $f$, then $\oE^2[\pi,\phi]<1$, and this
property characterizes such rational maps among the wider class of their
topological counterparts, specifically, critically finite
self-branched-coverings of the sphere for which each cycle in the
post-critical set contains a critical point \cite{Thurston20:Characterize}.
As a function of~$q$, the asymptotic energy $\oE^q[\pi,\phi]$ is
continuous and non-increasing, so that the level set
$\{q\mid \oE^q[\pi,\phi]=1\}$ is an interval
$[q_*[\pi,\phi], q^*[\pi, \phi]]$.

\subsection{Main result}
Our main result relates Ahlfors-regular conformal dimension to these
critical exponents, and implies that the invariant $\oE^q[\pi,\phi]$
contains useful information for other values of $q$.

\begin{theorem}\label{thm:crit-sandwich}
  For any recurrent, forward-expanding virtual graph endomorphism $(\pi,\phi)$,
  \[
    q_*[\pi,\phi] \le \ARCdim(\cJ[\pi,\phi]) \le q^*[\pi,\phi].
  \]
  Equivalently, for $q=\ARCdim(\cJ[\pi,\phi])$, we have
  $\oE^q[\pi,\phi]=1$.
  \end{theorem}

  In fact we expect that $q_*=q^*$: 
  \begin{conjecture}\label{conj:crit-equality}
  For any recurrent virtual graph endomorphism $[\pi,\phi]$, the
  function $q \mapsto \oE^q[\pi,\phi]$ is
  either constant or strictly
  decreasing.
\end{conjecture}
 In particular if $[\pi,\phi]$ is forward-expanding then (since
 $\oE^1[\pi,\phi] \ge 1$ and $\oE^\infty[\pi,\phi] < 1$), the
 conjecture would imply that
 \[
   q_*[\pi,\phi] = q^*[\pi,\phi]
 \]
 and Theorem~\ref{thm:crit-sandwich} characterizes $\ARCdim$.

 One might expect more to be true in
Conjecture~\ref{conj:crit-equality}, in particular some type of
convexity of the function $q \mapsto \oE^q[\pi,\phi]$ (after some
reparametrization of the domain and/or range).
More generally, it would
be interesting to know the relationship between our constructions and
more classical constructions in thermodynamic formalism. In this
vein, we remark that  Das, Przytycki, Tiozzo, and Urbański
\cite{DPTZ19:CoarseExpanding} have developed the thermodynamic
formalism in a setting which includes the topologically coarse
expanding conformal maps $f\co \cJ \to \cJ$ considered here.
 
\subsection{Outline} The bulk of the paper is devoted to developing
the technology to prove Theorem~\ref{thm:crit-sandwich}.

\subsubsection*{Topological dynamics} (\S \ref{sec:ve_and_dynamics})
We begin with an arbitrary forward-expanding recurrent virtual
endomorphism of graphs $(\pi, \phi\co G_1 \rightrightarrows G_0)$.
Iteration, suitably defined, gives rise to
\begin{enumerate}
\item a sequence of virtual endomorphisms $(\pi^n_{n-1}, \phi^n_{n-1}\co G_n \rightrightarrows G_{n-1}), n \in \mathbb{N}$;
\item a connected, locally connected, compact space $\cJ$, the inverse
  limit of
\[
 \cdots \overset{\phi^{n+1}_n}{\longrightarrow}
  G_n \overset{\phi^n_{n-1}}{\longrightarrow} G_{n-1}
  \overset{\phi^{n-1}_{n-2}}{\longrightarrow} \cdots
  \overset{\phi^2_1}{\longrightarrow} G_1
  \overset{\phi^1_0}{\longrightarrow} G_0;
\]
\item a positively expansive self-covering map $f\co \mathcal{J} \to
  \mathcal{J}$ of degree $d\coloneqq\deg(\pi)$;
\item the topological conjugacy class of $f\co \cJ \to \cJ$ depends only on the homotopy class $[\pi,\phi]$.  
\end{enumerate}
Our development in this section is quite general. We consider pairs of
maps $\pi, \phi\co X_1 \rightrightarrows X_0$ between finite CW
complexes equipped with length metrics and satisfying natural
expansion conditions, and establish properties of the dynamics on the
limit space. The main result, Theorem~\ref{thm:lc}, shows that under
these conditions, the construction of the conformal gauges given in
the next section applies.  
We also give a general result, Theorem
\ref{thm:Ashadowing}, which promotes a family of homotopy
pseudo-orbits to a map to the limit space in a natural way. This
allows us to relate features of the pair $(\pi, \phi)$ to features of
the limit space.

\subsubsection*{The conformal gauge} (\S \ref{sec:gauge}) We apply a
construction in \cite{kmp:ph:cxci} to put a nice metric $d_\epsilon$
on $\mathcal{J}$, called a \emph{visual metric}. It depends on a
suitably small but arbitrary parameter~$\epsilon$, and on the data of
a finite cover $\mathcal{U}_0$ of $\mathcal{J}$ by open, connected
sets. Changing this data changes the metric by a special type of
quasi-symmetric map called a \emph{snowflake map}. Equipped with a
visual metric, $f$ is positively expansive. Even better, any ball of
sufficiently small radius is sent homeomorphically and homothetically onto
its image, with expansion constant~$e^\epsilon$. The metric space
$(\mathcal{J}, d_\epsilon, \mathcal{H}^q)$ is Ahlfors regular of
exponent $q\coloneqq \log(d)/\epsilon=\hdim(\mathcal{J}, d_\epsilon)$.
Therefore, the invariant $\ARCdim(\cJ[\pi, \phi])$ is well-defined. See
Theorem~\ref{thm:visual}.

\subsubsection*{Energies of graph maps} (\S \ref{sec:energies}) When
equipped with natural length metrics, the maps $\phi^n\coloneqq
\phi^n_0\co G_n \to G_0$ have, for each $q \in [1,\infty]$, a
$q$-conformal energy $E^q_q[\phi^n]$ in the sense introduced by the
second author \cite{Thurston19:Elastic}.  The growth rate of this
energy as $n$ tends to infinity, namely
$\oE^q[\pi, \phi]\coloneqq  \lim E^q_q[\phi^n]^{1/n}$,
is called the \emph{asymptotic $q$-conformal energy}.  It 
depends only on the homotopy class $[\pi, \phi]$, and is continuous and 
non-increasing in~$q$ (Proposition~\ref{prop:energy}).  The expansion
hypothesis implies
$\oE^\infty[\pi,\phi]<1$, giving the interval in the statement of
Theorem~\ref{thm:crit-sandwich}.

\subsubsection*{Combinatorial modulus} (\S \ref{sec:modulus}) In this
section we recall (and extend slightly) results on a combinatorial
version of modulus in a fairly general setting, and how it is related
to Ahlfors-regular conformal dimension. Although the limit space~$\cJ$
hardly appears in this section, the ultimate motivation is of course
to estimate its conformal dimension. In more detail, fix
$n \in \mathbb{N}$. There is a natural projection
$\phi^\infty_n\co \cJ \to G_n$. The collection $\cV_n$ of fibers
$(\phi^\infty_n)^{-1}(e)$ of closed edges $e \in E(G_n)$ gives a
covering of $\cJ$.%
\footnote{For technical reasons, we actually work with a slightly
  different cover $\cV_n$ given by inverse images of a slightly larger
  set~$\widehat{e}$ called the \emph{star} of~$e$; see Definition
  \ref{def:star}.}
Given a family of curves $\Gamma$ in $\cJ$ and
an exponent $q \geq 1$, we get a numerical invariant,
$\modulus_q(\Gamma, \cV_n)$, the combinatorial modulus of this
family.

\S\S \ref{subsec:comb-mod}--\ref{sec:regularity} develop general
properties of combinatorial modulus. We need to
consider a mild generalization: we define combinatorial modulus for
families of \emph{weighted} curves.

\S \ref{sec:comb-mod-graph} continues by relating combinatorial
modulus to energies of graph maps.  For this, it is technically
convenient to work with the reciprocal of modulus, namely extremal
length. The relation arises via the characterization of graph map
energy in terms of maximum distortion of extremal length.  See
Theorem~\ref{thm:sf}, which requires a formulation of extremal length
in terms of weighted curves.

\S \ref{subsecn:conf-dim-comb-mod}  recalls a result of Carrasco
\cite[Theorem~1.3]{C13:Gauge}, which was independently proved by
Keith-Kleiner (unpublished). This
result states that for a suitably self-similar space $Z$ and a
suitable family of coverings $(\cS_n)_n$ indexed by $\mathbb{N}$, there is a
\emph{critical exponent} $q$. For a reasonably natural curve
family~$\Gamma$ in the space, this critical exponent distinguishes between
$\modulus_p(\Gamma, \cS_n) \to \infty$ if $p<q$ and
$\modulus_p(\Gamma, \cS_n) \to 0$ if $p>q$; see
Theorem~\ref{thm:keith-kleiner}.

\subsubsection*{Sandwiching the dimension} (\S \ref{sec:sandwich})
The proof of Theorem \ref{thm:crit-sandwich} applies the developments
in \S \ref{sec:modulus} and \S \ref{sec:energies}. We relate
combinatorial modulus of curve families in the limit space to
combinatorial modulus of curve families on the graphs $G_n$, and
then to energies of graph maps.

Here is a brief summary of the proof. The collection $\{\cV_n\}_n$
above forms a family of \emph{snapshots} of the limit
space, equipped with a visual metric;  see \S \ref{subsec:snap-stars}.   We
show a curve $\gamma\co C \to G_n$ can be approximately lifted under
$\phi^\infty_n$ to a curve $\gamma'\co C \to \cJ$ such that the composition
$\gamma''\coloneqq \phi^\infty_n \circ \gamma'\co C \to G_n$ is homotopic to
$\gamma$, with traces of size uniformly bounded independent of $n$. This
implies that combinatorial moduli for $\gamma$ and $\gamma''$ on $G_n$ are
comparable to that of $\gamma'$ on $\cJ$ (Lemma \ref{lemma:eg_vs_ej}).

With this setup, the upper bound on conformal dimension is
straightforward to verify; see \S \ref{subsec:leq}. The lower bound is
more involved and uses in an essential way the existence of a curve
$\gamma\co C \to G_n$ which is extremal for the distortion of extremal
length in the homotopy class of $\phi^n_0\co G_n \to G_0$. When
projected to $G_0$, the strands of $\phi^n_0\circ \gamma$ are very long and
cross edges of $G_0$ many times. We decompose $\gamma$ into a family of
subcurves $\zeta$, each of which projects to an edge of $G_0$, and make
the needed estimates. See \S \ref{subsec:geq}.

\subsection{Applications} In \S \ref{sec:applications}, we give
several applications of our methods.

\subsubsection{Techniques for estimates}
Theorem~\ref{thm:crit-sandwich} yields practical methods for
estimating the Ahlfors-regular conformal dimension.  If specific
$\phi$ and $q$ are given with
$E^q_q(\phi)<1$, the submultiplicativity of energy under composition
yields $\oE^q[\pi, \phi]<1$ and thus, by
Theorem~\ref{thm:crit-sandwich}, $\ARCdim(\cJ[\pi,\phi]) < q$.
Furthermore, there are bounds on how quickly $\oE^q[\pi,\phi]$ can
decrease as a function of~$q$ \cite[Proposition
6.11]{Thurston20:Characterize}, so if we
know $\oE^q[\pi,\phi]< 1$, we get an upper bound on~$q^*$ that is
smaller than~$q$.  See Proposition~\ref{prop:Eq-lower-bound}.  

For lower bounds, we have the following. Set
$\overline{N}[\pi, \phi]=\oE^1[\pi,\phi]$.
One way to view this quantity is as the
asymptotic growth rate of the maximal number of
edge-disjoint curves
in~$G_n$, each representing a non-trivial loop in~$G_0$.
The bounds mentioned above on how fast $\oE^q$ decreases as a function
of~$q$ yield the following.
\begin{theorem}
\label{thm:Nbar}
For any recurrent forward-expanding virtual graph automorphism $[\pi,\phi]$ where
$\deg(\pi) = d$, we have
\[ \ARCdim(\cJ[\pi, \phi]) \geq \frac{1}{1-\log_d \overline{N}[\pi, \phi]}.\]
\end{theorem}
For the virtual endomorphism of $f$ from
Figures~\ref{fig:barycentric_carpet} and~\ref{fig:barycentric_ve}, we
find by hand that $\overline{N}(f)=2$ and
$\oE^2[\pi,\phi] < \sqrt{10/13}$, so
$1.6309 \approx \frac{1}{1-\log_6 2} < \ARCdim(J_f) <
\frac{2}{1-\log_6 (10/13)} \approx 1.7445$; see
\S\ref{sec:barycentric_carpet} for details.

If $f$ is a hyperbolic rational map and $[\pi, \phi]$ an associated
virtual graph endomorphism, the quantity $\overline{N}(f)$
seems to be closely related to the topological and metric structure of
the Julia set. For example, we show the following.
\begin{theorem}
\label{thm:Nbar_ge_1}
Suppose $f$ is a critically finite hyperbolic rational map. If $J_f$ is a Sierpiński carpet, then $\oN(f)>1$. 
\end{theorem}
Examples show that the converse need not hold; see
\S\ref{sec:Nbar_gtr_1}.

\subsubsection{When the conformal dimension equals \textalt{$1$}{1}}  M. Carrasco \cite[Theorem 1.2]{C14:ConfDim} gives
a metric condition, \emph{uniformly well-spread cut points} (UWSCP; Definition \ref{def:uwscp} below),
on a compact doubling metric space $X$ which guarantees
$\ARCdim(X)=1$.  All hyperbolic polynomials and rational maps with
``gasket-type'' Julia sets satisfy the UWSCP  condition.  Combining
his observation with our Theorem~\ref{thm:Nbar}, we obtain the following.
\begin{theorem}
\label{thm:uwscp_implies_Nbar1}
Suppose $f$ is a hyperbolic rational map. If $J_f$ satisfies the UWSCP condition, then $\overline{N}(f)=1$.
\end{theorem}

We also show that Carrasco's criterion
 for $\ARCdim = 1$ is not necessary.
 
\begin{proposition}
  \label{prop:rmb}
  Let $R$ be the rational map obtained by mating the Douady rabbit quadratic polynomial with the basilica polynomial $z^2-1$. Then $\ARCdim(J_R)=1$ but $J_R$ does not satisfy UWSCP.
\end{proposition}

This example is shown in Figure~\ref{fig:mate-rabbit-basilica}. More
generally, applying our methods, InSung Park
\cite{Park21:Obstructions} has proved the following generalization:
\emph{A hyperbolic critically finite rational map $f$ satisfies
  $\ARCdim(J_f)=1$ if and only if $f$ is a crochet map, if and only
  if $\oN(f) = 1$.} Here a \emph{crochet
map} is one in which any pair of points in the Fatou set is joined by a
path which meets the Julia set in a countable set of points.

\begin{figure}
  \centerline{\includegraphics[height=3in]{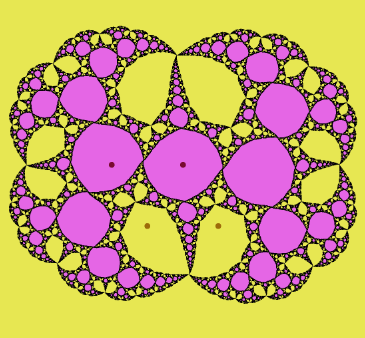}}
  \caption{The mating of the rabbit (yellow) and the basilica (pink)
    polynomials.}
  \label{fig:mate-rabbit-basilica}
\end{figure}

\subsubsection{Variation in a family} R. Devaney \emph{et al.} studied
the family $f_\lambda(z)=z^2+\lambda/z^2$ for $\lambda \in \C-\{0\}$
\cite{devaney:look:uminsky:IUMJ05}. Figure \ref{fig:devaney-2-2-param}
shows the bifurcation locus in the parameter plane for this
family. (The four critical points at fourth roots of~$\lambda$ end up
in the same orbit for this
family, making it a critical orbit variety \cite{BdM13:SpecialCurves}
and simplifying the
analysis; see Eq.~\eqref{eq:devaney-orbits}.)
\begin{figure}
\includegraphics[width=3in]{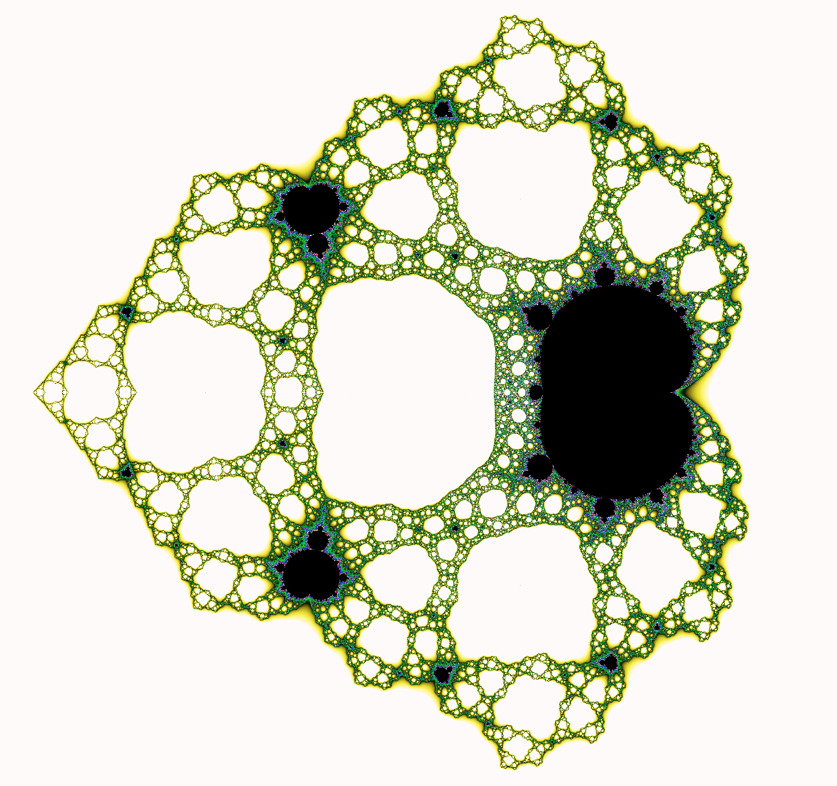}
\caption{Parameter plane for $f_\lambda(z)=z^2+\lambda/z^2$. The white
regions are where the orbits of the critical points $\lambda^{1/4} \mapsto \pm 2\sqrt{\lambda} \mapsto 4\lambda + 1/4 \mapsto \ldots$ escape to the
attracting fixed point at~$\infty$.  The small
black Mandelbrot sets correspond to cases where that
orbit is attracted to a cycle other than the fixed-point at infinity.}
\label{fig:devaney-2-2-param}
\end{figure}
Parameters taken from the prominent ``holes'' along the real axis have carpet Julia sets. 
In \S\S \ref{sec:fat-devaney} and \ref{sec:skinny-devaney}, we present
two sequences $\lambda_n^{\mathrm{skinny}},
\lambda_n^{\mathrm{fat}}$, $n \in \mathbb{N}$, whose values converge, respectively, to
the left-most real parameter (at
$\lambda=-\frac{3}{16}-\frac{\sqrt{2}}{8}$) and to the origin. Figure
\ref{fig:two_devaney_exs} illustrates two
examples.
 \begin{figure}
\centerline{\includegraphics[width=2.5in]{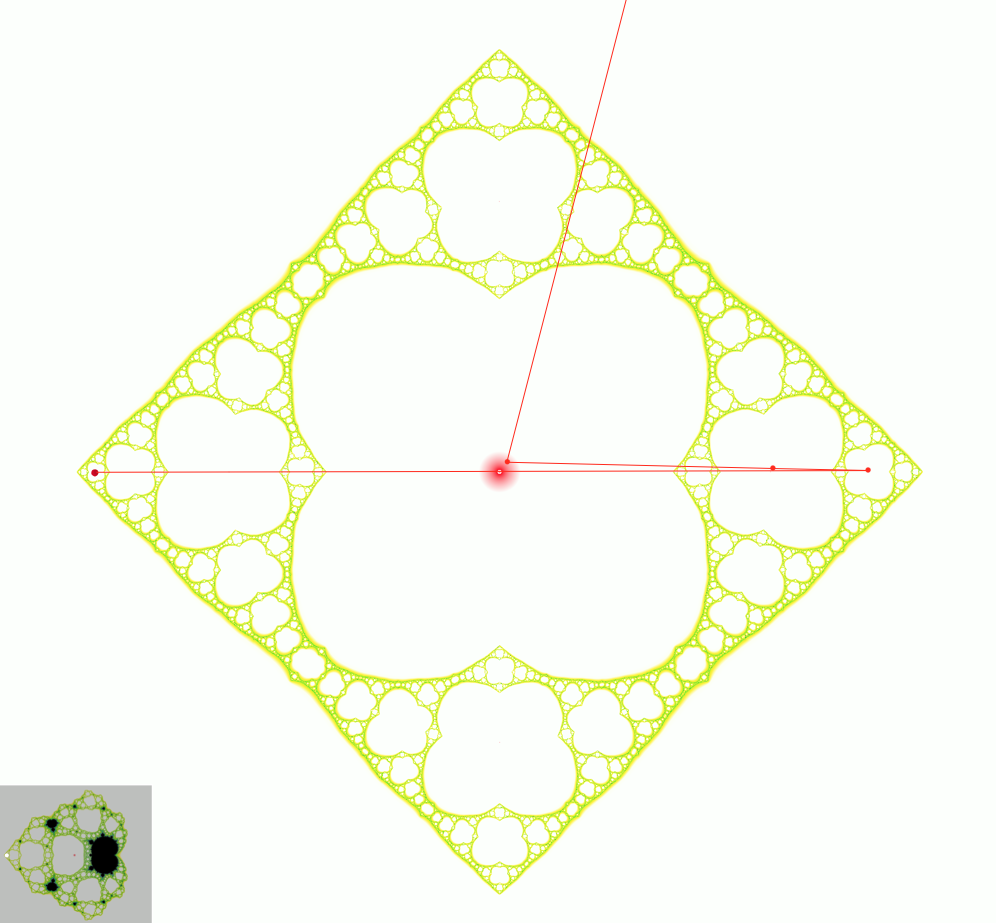}\quad
\includegraphics[width=2.5in]{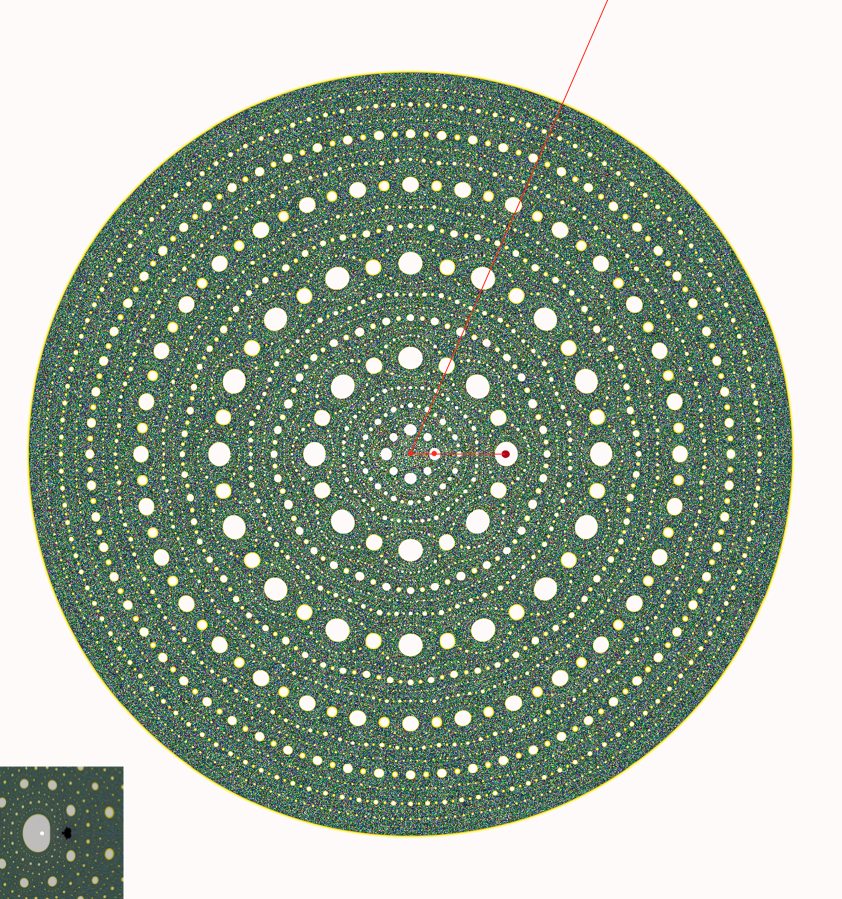}}
 \caption{Two Sierpiński carpet Julia sets in the Devaney family
   corresponding to parameters $\lambda_2^{\mathrm{skinny}} \approx
   -.35891$ (left) and  $\lambda_2^{\mathrm{fat}} \approx -1.4996769
   \times 10^{-5}$ (right). The red points are the common second iterates $x_\lambda$ of the finite nonzero critical points; the white points show the orbit of $x_\lambda$. The insets at lower left
   are  local pictures in the parameter plane. Figures created with
   FractalStream \cite{Noonan:FractalStream}.}
  \label{fig:two_devaney_exs}
 \end{figure}
 Note the difference in the apparent ``thickness'' of the carpets. Though each Julia set is a Sierpiński carpet, that of $\lambda_n^{\mathrm{skinny}}$ visually becomes ``skinnier'' as $n \to \infty$, while that of $\lambda_n^{\mathrm{fat}}$ visually becomes ``fatter'' as $n \to \infty$. The former rate seems to be rather gradual, while the latter rate seems to be very fast. 
 
 M. Bonk, M. Lyubich, and S. Merenkov showed that a
 quasi-symmetric map between hyperbolic carpet Julia sets extends to a
 Möbius transformation \cite{bonk:lyubich:merenkov:2016}. This
 easily implies that the three carpets shown in Figures
 \ref{fig:barycentric_carpet} and \ref{fig:two_devaney_exs} are
 pairwise quasi-symmetrically inequivalent. Our techniques allow us to
 quantify this distinction.
 \begin{theorem}
 \label{thm:1_and_2}
 We have
 \[ 
   1 < \ARCdim(J_{\lambda_n^{\mathrm{skinny}}}) <  1+ \frac{1}{\log_2(2n+3)},\]
 and
 \[
   \frac{2}{1+2^{-n}} \le \ARCdim(J_{\lambda_n^{\mathrm{fat}}}) <  2.
 \]
 \end{theorem}
 In particular, within this family of fixed degree, there are hyperbolic carpet maps with conformal dimension tending to 1 and to 2.  This latter result answers a question of the first author and P. Haïssinsky \cite{kmp:ph:ex}.  
 
\subsection{Note on notions of conformal dimension}

The \emph{conformal Hausdorff dimension} of an Ahlfors-regular metric
space $X$ is defined as the infimum of the Hausdorff dimensions of
metric spaces quasi-symmetrically homeomorphic to $X$.  \emph{A
  priori} it could be strictly smaller than $\ARCdim(X)$.

After submission of this article, we learned of a result of S. Eriksson-Bique \cite[Theorem 1.6]{erikssonbique:equality}, which implies the following:

\begin{citethm}[Eriksson-Bique]
    Let $X$ be a compact, connected, locally connected, quasi-self-similar metric space. Then the conformal Hausdorff dimension and Ahlfors-regular conformal dimension coincide.
\end{citethm}

The property of being quasi-self-similar means arbitrarily small balls
in $X$ are quasi-symmetrically equivalent to open sets in $X$ of
definite size, where the distortion function in the definition of
quasi-symmetry is independent of the location and radius of the chosen
ball. This property is invariant under quasi-symmetric maps.  The
visual metrics we construct in \S \ref{sec:gauge} below are
quasi-self-similar.  Therefore, all of our main results hold with
``Ahlfors-regular conformal dimension'' in the conclusion replaced
with ``conformal Hausdorff dimension''.

\subsection{Related work}
Kwapisz \cite{Kwapisz20:ConfDimResist} uses a similar approach to
estimating the conformal dimension for Sierpiński carpets. In his
paper, he considers only the standard square carpet, but uses very
similar notions, with the slight variation that his $q$-resistance
$r(e)$ is related to our
$q$-length $\alpha(e)$ by
\begin{equation}
  \label{eq:q-resist-length}
  r(e) = \alpha(e)^{q-1}.
\end{equation}
(In particular, as in this paper, there are quantities associated to
the edges rather than the vertices, as contrasted with the more common
notions of combinatorial $q$-modulus in the literature.) With this
correspondence, Kwapisz' formulas match with ours; for instance, his
$P(\mathcal{J})$ \cite[Eq.~(1.2)]{Kwapisz20:ConfDimResist} agrees with
our $\bigl(E^1_q(\phi)\bigr)^q$ in Eq.~\eqref{eq:E1q}. One difference between our
approaches is that he deals with signed flows as in electrical networks,
while we deal with  more general unsigned tensions related to
elasticity; see the discussion in
\cite[Appendix~B]{Thurston19:Elastic}.

A more substantive difference is that our $q$-conformal energies enjoy exact
sub\hyp multiplicativity \cite[Prop.~A.12]{Thurston19:Elastic}, while
Kwapisz only proves weak sub\hyp multiplicativity, up to a constant
\cite[Theorem~1.3]{Kwapisz20:ConfDimResist}. On the other hand, by
only asking for weak sub-multiplicativity, he is able to get bounds
on $q$-resistance for both the graph analogous to (but more general than)
the ones we consider
and its dual. Correspondingly, he gets numerical
lower bounds, in addition to numerical upper bounds analogous to the
ones we get. (Our numerical lower bounds rely on
Theorem~\ref{thm:Nbar}, which is unlikely to be sharp in general.)

In another direction, in the non-dynamical setting of Gromov
hyperbolic complexes, Bourdon and Kleiner \cite{MR3356973} consider
families of analytic invariants such as $\ell^p$ cohomology and
separation properties of certain associated function spaces. It would
be interesting to know if there is a connection to this work in our
setting. 

\subsection{Notation}\label{subsecn:notation}
We denote the unit interval by $I\coloneqq[0,1]\subset \mathbb{R}$.

Non-dynamical functions, i.e., where the domain and codomain are not
identified, are typically denoted using lower-case Greek letters.

If $K$ is a constant, or set of constants, the notation $A \lesssim_K
B$ means that $A/B \leq C(K)$ for some constant $C$ depending only on
$K$. The notation $A \asymp_K B$ means $A \lesssim_K B$ and $B
\lesssim_K A$.

The study of dynamical systems includes ``multivalued'' ones, which
are formally pairs of maps between different spaces. We denote these
spaces with
subscripts, e.g., $X_1 \to X_0$ in the case of general spaces, or
$G_1 \to G_0$ when the spaces are graphs. Such systems generate
sequences of spaces and graphs which we denote by $X_n$ and $G_n$
respectively, for $n=0, 1, 2, \ldots$ In this context, we decorate
maps with a superscript corresponds to the domain and the subscript to
the codomain, e.g., we have structure maps $\phi^n_m \colon G_n \to
G_m$ for $n > m$ (\S\ref{subsecn:multivalued}). Similarly when
looking at graph energies $E^p_q(\phi)$, the
domain has a $p$-conformal structure and the codomain has a
$q$-conformal structure
(\S\ref{sec:energies}).

We also often work with maps which are not part of the data of a dynamical system. 
For such maps, we typically distinguish domain and codomain by using different
letters as opposed to subscripts.

\subsection*{Acknowledgements}

We thank
Caroline Davis, 
InSung Park, 
and
Giulio Tiozzo 
for useful conversations, and the referee for detailed comments on the
first submission. The first author was supported by the Simons
Foundation under Grant Numbers 245269 and 615022.   The second author
was supported by the National Science Foundation under Grant Numbers
DMS-1507244 and DMS-2110143.


\section{Topological dynamics}
\label{sec:ve_and_dynamics}

The main result of this section is the following theorem,
giving good
dynamical properties of the action on a limit space.

\begin{theorem}
  \label{thm:lc}
Suppose $X_0, X_1$ are finite CW complexes equipped with complete length metrics.
Suppose $\pi, \phi \co X_1 \rightrightarrows X_0$ is a
$\lambda$-backward-contracting and recurrent virtual endomorphism, and
let $f\colon \cJ \to \cJ$ be the induced dynamics on its limit space.
Then  
\begin{enumerate}
\item\label{item:J-conn} The space $\cJ$ is connected and locally connected. 
\item\label{item:J-expansive} The map $f\colon\cJ \to \cJ$ is a positively expansive covering map of degree $\deg(\pi)$.
\item\label{item:J-cxc} The dynamical system $f\colon \cJ \to \cJ$ is topologically
  coarse expanding conformal (cxc), in the sense of Haïssinsky and Pilgrim
  \cite{kmp:ph:cxci}.
\end{enumerate}
\end{theorem}
The terminology is defined in \S\ref{subsecn:cxc} and
\S\ref{subsecn:multivalued}. One main result of \cite{kmp:ph:cxci} is
that
topologically cxc systems have a canonically associated nonempty set
of special Ahlfors-regular metrics, called \emph{visual metrics}. In
\S\ref{sec:gauge}, we will describe the geometry of~$\cJ$ when
equipped with such metrics.

Theorem \ref{thm:lc} could be deduced as follows. Nekrashevych
\cite[Theorem 5.10]{Nekrashevych14:CombModel} associates to such a virtual
endomorphism a self-similar recurrent contracting group
action. Such an object has also an associated limit space homeomorphic
to $\cJ$ that is connected and locally connected \cite[Theorem
3.5.1]{nekrashevych:book:selfsimilar}, establishing~\eqref{item:J-conn}. Moreover,
there is an induced dynamics on this limit space naturally conjugate
to $f\colon \cJ \to \cJ$. Conclusions \eqref{item:J-expansive} and~\eqref{item:J-cxc} then follow from
\cite[Theorem 6.15]{ph:kmp:GGD2011}. To keep this work self-contained,
and because we have later need of certain other related technical
facts, we give more direct arguments.

To prove Theorem \ref{thm:lc}, we present the limit space as a
subspace of the infinite product $(X_1)^\infty$, following Ishii and
Smillie \cite{ishii:smillie:homotopy-shadowing}. We generalize the
theory of homotopy pseudo-orbits developed there to \emph{families} of
homotopy pseudo-orbits parameterized by an auxiliary space~$A$, and show, by
generalizing their results on homotopy shadowing, that such a family
determines a map $A \to \cJ$ (Theorem \ref{thm:Ashadowing}). We prove
that the limit space is locally path-connected by applying this
generalization to the case when $A=I$ is an interval (\S
\ref{subsecn:lc}). To obtain the needed family of homotopy
pseudo-orbits in this setting, we develop in \S \ref{subsecn:apl} a
general notion of approximate path-lifting. 

For technical reasons, we need to control the
geometry of non-rectifiable paths. To this end, we introduce the
\emph{size} of a path, defined to be the diameter of the
lift to a universal cover; see~\S \ref{sec:size-length}.

We also present the associated limit space~$\cJ$ in an
equivalent, but more convenient, way as an inverse limit of
spaces with increasingly large diameter; see
Theorem~\ref{thm:Ashadowing-cover}. To find path families, we also
need to find ``homotopy
sections'' of projections from~$\cJ$.
We therefore need results on homotopy
shadowing and homotopy sections in those settings; see~\S\ref{subsecn:sections}.

\subsection{Topologically cxc systems} 
\label{subsecn:cxc}
We first recall the notion of topologically cxc systems.  To
streamline our presentation,  we specialize the setup of
\cite{kmp:ph:cxci}  to the case of self-covers.  We next define
positively expansive systems. Finally, we show that conclusions~\eqref{item:J-conn}
and~\eqref{item:J-expansive} of Theorem~\ref{thm:lc} imply
conclusion~(\ref{item:J-cxc}).

Suppose $\cJ$ is a compact, connected, and locally-connected
topological space, and $f\colon \cJ \to \cJ$ is a self-covering. 
A finite open cover $\cU_0$ of $\cJ$ by connected sets inductively
generates a sequence of coverings $\cU_n, n \in \mathbb{N}$, via the
recipe
\begin{equation}\label{eq:inductive-cover}
  \cU_{n+1} \coloneqq \{ \wtU \mid U \in \cU_n, \wtU\text{ is a
    component of }f^{-1}(U)\}.
\end{equation}
Also set
\[
  \mathbf{U} \coloneqq \bigcup_n \cU_n.
\]
The \emph{mesh} of a covering of a metric space is the
supremum of the diameters of its elements.

\begin{definition}
  The pair of the dynamical system $f \co \cJ \to \cJ$ and cover
  $\cU_0$ is said to satisfy Axiom [Expansion] if, for some
  (equivalently, any) metric on~$\cJ$ compatible with the topology,
  the mesh of $\mathcal{U}_n$ from Eq.~\eqref{eq:inductive-cover}
  tends to zero as $n \to \infty$. The map $f$ satisfies [Expansion]
  if there is some $\cU_0$ satisfying this condition.
\end{definition}

\begin{lemma}
  \label{lemma:expansion_enough}
  Suppose $\cJ$ is compact, connected, and locally connected,
$f\colon \cJ \to \cJ$ is a covering map,
and $\mathcal{U}_0$ is a covering so $(f,\cU_0)$ satisfies Axiom [Expansion].  Then this dynamical system satisfies the following additional two axioms:
\begin{itemize}
\item \textnormal{[Degree]} For fixed $n$ there is a uniform upper bound on
  the cardinality of fibers of restrictions $f^{\circ n}\colon
  \widetilde{U}_n \to U_0$, for all $U_0 \in \mathcal{U}_0$ and
  $\widetilde{U}_n \in \mathcal{U}_n$.
\item \textnormal{[Irreducibility]} For any nonempty open set
  $W \subset \cJ$, there is an integer $N$ for which $f^{\circ N}(W)=\cJ$. 
\end{itemize}
\end{lemma}

On the conditions of this lemma, $f$ is \emph{topologically cxc}.
The general definition of topologically cxc is that  there exists a finite open cover $\mathcal{U}_0$ such that all three axioms [Expansion], [Degree], and [Irreducibility] hold. 

\begin{proof}
For a self-cover satisfying axiom [Expansion], axiom
[Degree] clearly holds, since for large enough $n$ the elements of
$\mathcal{U}_n$ will be evenly covered.

To show [Irreducibility], fix arbitrarily a metric on $\cJ$ compatible
with its topology. For $x \in \cJ$, define
\begin{align*}
  Y(x) &\coloneqq \{\,y\in X\mid \text{the backward orbit of $y$ accumulates at
         $x$}\,\}\\
       &= \{\, y \in X \mid \text{there is a sequence }\tilde{y}_{n_k}
         \in f^{-n_k}(y) \text{ s.t.\ }\lim_{k \to \infty} \tilde{y}_{n_k} = x\,\}.
\end{align*}
We first claim $Y(x)$ is nonempty for each $x$. To see this, fix
$x \in \cJ$, and let $y$ be any accumulation point of the forward
orbit of~$x$, i.e., $\lim_{k \to \infty} f^{\circ n_k}(x) = y$. Let
$U \in \mathcal{U}_0$ contain $y$. For each
sufficiently large~$k$, let $\widetilde{U}_{n_k}$ be the unique component of
$f^{-n_k}(U)$ containing $x$, and pick
$\tilde{y}_{n_k} \in \widetilde{U}_{n_k} \cap f^{-n_k}(y)$. Then
$\tilde{y}_{n_k} \to x$ since $\diam \widetilde{U}_{n_k} \to 0$. Hence $y \in Y(x)$.  

Now suppose $y \in Y(x)$ is arbitrary and choose $U \in \mathcal{U}_0$
with $y \in U$. The same reasoning with $\diam \widetilde{U}_{n_k}$
shows that $U \subset Y(x)$. Since $X$ is connected and
covered by $\mathcal{U}_0$, we conclude $Y(x)=X$. Since $x$ is
arbitrary, we conclude the set of backward orbits under $f$ of each point is
dense in $X$.

Now fix an arbitrary nonempty open set~$W$ as in the statement of Axiom [Irreducible].  
By [Expansion], there exists $U \in \mathbf{U}$ with $U \subset W$. Pick $y \in U$.
By the previous paragraph, there exists a backward orbit of $y$ accumulating at $y$.
Pulling back $U$ along this backward orbit, [Expansion] implies that
there exists some $\wtU \in \mathbf{U}$ so that
$\wtU \subset U$ and $f^{\circ m}\co \wt U \to U$ is a covering map for some $m>0$.
The previous paragraph implies $\bigcup_{n=0}^\infty f^{nm}(\wtU)=X$.
By our choice of $m$ and $\wtU$, this is an increasing union.  Since
$X$ is compact, $f^{nm}(\wtU)=X$ for some $n$. Then $N=nm$ suffices
for the statement, since $\wtU \subset W$.
\end{proof} 

\begin{definition}[Positively expansive]
\label{def:pos-expansive}
Let $(\cJ,d)$ be a metric space. A continuous surjection
$f\colon \cJ \to \cJ$ is \emph{positively expansive} if there is a
constant $e>0$ so that for any $x \neq y \in \cJ$ there is an $n\ge 0$ so that
$d(f^n(x),f^n(y))>e$.
\end{definition}

Note that a positively expansive map is locally injective. 

In the case $\cJ$ is compact, positively expansive is equivalent to
the following condition: there exists a neighborhood $N \supset
\{(x,x)\co x \in \cJ\}$ of the diagonal such that if $x, y \in \cJ$
satisfy $(f^ix, f^iy) \in N$ for all $i \geq 0$, then $x=y$.
(In particular, positively expansive is independent of the metric.)
In addition, by a
theorem of Reddy \cite[Theorem 2.2.10]{aoki:hiraide:topological} there
exists a compatible metric $D$ on~$\cJ$, called an \emph{adapted
  metric}, and \emph{expansion constants} $\delta>0$ and $0<
\lambda<1$
such that for any $x, y \in \cJ$,
\[ D(x,y) \leq \delta \implies D(f(x),f(y))\geq \lambda^{-1} D(x,y).\]
Note that this implies $f$ is a homeomorphism on $\delta$-balls in the $D$-metric. 

\begin{proposition}[Positively expansive implies {[Expansion]}]
\label{prop:shrink}
Suppose $\cJ$ is compact and locally connected, and $f\colon \cJ \to
\cJ$ is a positively expansive covering map.  Then $f$ satisfies Axiom
[Expansion].
\end{proposition}

To prove this, we will need the following result.

\begin{citethm}[{Eilenberg constants \cite[Theorem 2.1.1]{aoki:hiraide:topological}}]
\label{thm:eilenberg}
Let $X$ be compact and $f\co X \to Y$ be a continuous surjective local
homeomorphism. Then there exist two positive numbers $\tau$ and $\mu$
such that for each subset $U$ of $Y$ with diameter less than $\tau$
there is a decomposition of the set $f^{-1}(U)=U_1 \cup \dots \cup U_d$ with the following properties:
\begin{enumerate}
\item $f\co U_i \to U$ is a homeomorphism;
\item for $i \neq j$ no point of $U_i$ is closer than $2\mu$ to a
  point of $U_j$; and
\item for each $\eta>0$ there exists $0<\epsilon<\tau$ such that if
  $\diam U < \epsilon$ then  for all~$j$, $\diam U_j < \eta$.
\end{enumerate}
\end{citethm}

\begin{proof}[Proof of Proposition~\ref{prop:shrink}]
  Equip $\cJ$ with an adapted metric. We apply Theorem \ref{thm:eilenberg} with $X=Y=\cJ$ and obtain the constant $\tau$.  Let $\delta$ be as in the definition of adapted metric, and take the constant $\eta$ in Theorem \ref{thm:eilenberg} so that $\eta<\delta$; we obtain a constant $\epsilon$. In summary: any open connected set $U$ of diameter at most $\epsilon$ is evenly covered by $f$ and has preimages $\wtU_1, \ldots, \wtU_d$ of diameter at most $\delta$. The definition of adapted metric then implies that $f\co \wtU_j \to U$ expands distances by at least the factor $\lambda^{-1}$ and so the inverse branches $f_j^{-1}\co U \to \wtU_j$ contract distances by at least the factor $\lambda$.  

Since $\cJ$ is locally connected and compact, there is a finite open
cover $\cU_0$ by connected sets of mesh~$\epsilon$. The previous
paragraph implies that the covering $\cU_1$ has mesh at most
$\lambda\epsilon<\epsilon$. Induction shows $\mathrm{mesh}(\cU_n) <
\lambda^n\epsilon \to 0$ as required in Axiom [Expansion].
\end{proof} 

\subsection{Length spaces}
\label{subsecn:lscm} 
In this subsection, we prepare for the proof of local connectivity by collecting some technical results related to covering maps and length spaces. 
Here, $X$ denotes a finite, hence compact, connected CW complex, equipped with a compatible length metric.  The Hopf-Rinow theorem \cite[Proposition I.3.7]{MR1744486} implies that $X$ is a geodesic metric space.  While balls might not be simply-connected, they are path-connected. The \emph{systole} is 
\[ \systole(X)\coloneqq\inf\{ \ell(\gamma) \mid \gamma \; \text{is an essential loop in $X$}\};\]
if there are no essential loops in $X$, we set $\systole(X)\coloneqq+\infty$. 
The systole is positive, since $X$ is compact. Since $X$ is a length space,
any cover $\wt X$ inherits a lifted metric by lengths of paths. A ball
of radius less than $\tfrac{1}{2}\systole(X)$ in~$X$ no essential
loops, and
thus is evenly covered.

\begin{lemma}
\label{lemma:evenly_covered}
Let $X$ be a finite connected CW complex with a length metric. Suppose
$p\colon \widetilde{X} \to X$ is a covering map, $\widetilde{X}$ is equipped with the lifted metric,
and $r < \tfrac{1}{4}\systole(X)$. Then for any $\tilde{x} \in \widetilde{X}$ with
$p(\tilde{x})=x$, the restriction $p\colon B(\tilde{x},r) \to B(x,r)$
is an isometry. 
\end{lemma}

This is standard, but we provide a proof for completeness.

\begin{proof} The definition of the lifted metric says
  that $p$ preserves the length of paths, and is therefore
  1-Lipschitz. Thus $p(B(\tilde{x},r)) \subset B(x,r)$.
  We have $p(B(\tilde{x},r))=B(x,r)$ since a geodesic joining $x$ to
  $y \in B(x,r)$ lifts to a path of the same length joining
  $\tilde{x}$ to some point $\tilde{y}$ which therefore lies in
  $B(\tilde{x},r)$. We now claim that
  $p\colon B(\tilde{x},r) \to B(x,r)$ is an isometry. Suppose
  $\tilde{a}, \tilde{b} \in B(\tilde{x}, r)$, and put $a=p(\tilde{a})$
  and $b=p(\tilde{b})$. Then $a,b \in B(x,r)$. Consider the piecewise
  geodesic path~$\gamma$ comprised of 3 length-minimizing segments
  which runs from $x$ to $a$, then from $a$ to $b$, then from $b$ to
  $x$. This loop may not lie in $B(x,r)$. However, $\ell(\gamma) < 4r$
  and both endpoints are at~$x$,
  so $\gamma \subset B(x,2r)$, which is evenly covered by the
  choice of $r$. It follows that the middle segment lifts to a segment
  joining $\tilde{a}$ to $\tilde{b}$ of length equal to $d(a,b)$.
  Hence $d(\tilde{a}, \tilde{b}) \leq d(a,b)$ and the result is
  proved.
\end{proof}

If $\phi\colon X \to Y$ is a map between metric spaces, we say a non-decreasing
function $\omega_\phi\colon [0, \diam(X)] \to [0, \diam(Y)]$ is a
\emph{modulus of continuity} if, for all $E \subset X$,
$\diam \phi(E) \leq \omega_\phi(\diam E)$.

\begin{lemma} 
\label{lemma:uc} 
Let $X, Y$ be compact, connected, CW complexes equipped with
length metrics, and let $\phi\colon X \to Y$ be a continuous map, with modulus
of continuity~$\omega_\phi$. Let $p_Y\colon \widetilde{Y} \to Y$ be any
covering map, and equip $\widetilde{Y}$ with the lifted metric.
Let $p_X\colon \widetilde{X} \to X$ and
$\widetilde{\phi}\colon \widetilde{X} \to \widetilde{Y}$ be the maps
induced by pullback, and equip $\widetilde{X}$ with the lifted
metric from $X$. Then $\widetilde{\phi}$ is uniformly continuous, with
modulus of continuity $\omega_{\wt\phi}$ independent of the
cover $p_Y$.
Indeed, there exists $r_0>0$ depending only on
$\phi\colon X \to Y$ so that we can take
\begin{equation*}
  \omega_{\wt\phi}(\delta) =
  \begin{cases}
    \omega_\phi(\delta) & \delta \le r_0 \\
    \omega_\phi(r_0)\left\lceil\frac{d(a,b)}{r_0}\right\rceil & \delta > r_0.
  \end{cases}
\end{equation*}
\end{lemma}
So $\widetilde{\phi}$ behaves just like $\phi$ at small scales, and is
Lipschitz at large scales.

\begin{proof} Put $s_0\coloneqq\frac{1}{4}\systole(Y)$.  By the
  uniform continuity of $\phi$, there exists
  $0<r_0<\frac{1}{4}\systole(X)$ such that for each $x \in X$
  and $y =\phi(x)$, we have $\phi(B(x,r_0))\subset B(y, s_0)$. Fix now
  $\tilde{x} \in \widetilde{X}$ and put $x=p_X(\tilde{x})$,
  $\tilde{y}=\widetilde{\phi}(\tilde{x})$, and
  $y=p_Y(\tilde{y})=\phi(x)$. By Lemma \ref{lemma:evenly_covered},
\[ p_Y^{-1}\co B(y,s_0) \to B(\tilde{y},s_0), \quad p_X\co B(\tilde{x}, r_0) \to B(x,r_0)\]
are isometric homeomorphisms. Now  fix $0<r\leq r_0$.  Then 
\[ \widetilde{\phi}(B(\tilde{x},r)) = (p_Y^{-1}\circ
  \phi \circ p_X)(B(\tilde{x},r)) = p_Y^{-1}(\phi(B(x,r))\subset
  p_Y^{-1}(B(y, \omega_\phi(r)))\subset B(\tilde{y}, \omega_\phi(r)),
\]
establishing the estimate in the case $\delta \le r_0$.
For the other case, suppose $a, b \in \widetilde{X}$ are at distance $R\geq r_0$,
and let $\gamma$ be a geodesic joining $a$ to $b$. Divide $\gamma$
into sub-segments $\gamma=\gamma_2 * \gamma_1 * \ldots * \gamma_n$ with
$\ell(\gamma_i) \leq r_0$ for $i=1, \ldots, n$, so
that $n = \lceil \ell(\gamma)/r_0\rceil$. Then 
\begin{align*}
 d(\widetilde{\phi}(a), \widetilde{\phi}(b)) &  \leq  \sum_{i=0}^n \diam \phi(\gamma_i)
 \le \omega_\phi(r_0)\lceil d(a,b)/r_0\rceil.\qedhere
 \end{align*}
\end{proof}

\subsection{Sizes of paths and traces of homotopies}
\label{sec:size-length}
It would be nice to always work with the length of paths, but it
turns out that not all the paths we consider are rectifiable. (In
particular, we consider paths in the Julia set~$\cJ$ and their
projections to the finite approximations~$G_n$; these projections
are usually not rectifiable.) We could consider the diameter of paths,
but we also need to lift paths to covers. We work instead with a
hybrid.

\begin{convention}
  For paths $\gamma \co I \to X$, the path $\overline{\gamma}$ is the
  reversed path, and $\gamma_1 * \gamma_2$ denotes composition of
  paths, defined when $\gamma_1(1) = \gamma_2(0)$.
  For homotopies $H \colon I \times A \to X$, we will more generally
  use the same notations $\overline{H}$ and $H_1 * H_2$, always
  operating on the first input (which is an interval).
\end{convention}

\begin{definition}\label{def:size}
  For $X$ a locally simply-connected length space, $A$ a
  simply-connected auxiliary space (usually the interval), and
  $\gamma \co A \to X$ a continuous map, there are lifts
  $\wt\gamma \co A \to \wt X$ of~$\gamma$ to the universal cover
  of~$X$. The \emph{size} of~$\gamma$ is the diameter of~$\wt \gamma$ with respect to the lifted metric on $\wt X$:
  \[
    \size(\gamma) \coloneqq
      \max_{s,t \in I} d_{\wt X}(\wt\gamma(s), \wt\gamma(t)).
    \]
\end{definition}

The proof of the following lemma is straightforward. 
\begin{lemma}[Properties of size for paths]
\label{lem:size}
The notion of size for paths (with $A$ the interval) satisfies the
following properties.
\begin{enumerate}[start=0]
\item Well-defined: $\size(\gamma)$ is independent
  of which lift of~$\gamma$ to~$\wt X$ you take.
\item\label{item:size-bound-length} Bounded by length: when $\gamma\co
  I \to X$ is rectifiable, 
  $\size(\gamma) \leq \ell(\gamma)$.
\item Invariance under lifts: if $p\co X \to Y$ is a covering map,
  $Y$~is a length space, $X$~is equipped with the lifted metric,
  $\gamma\co I \to Y$ is a path, and $\wt \gamma\colon I \to X$ is
  a lift of $\gamma$ under~$p$, then
  $\size(\wt \gamma) = \size(\gamma)$.
\item Sub-additive under path composition: $\size(\gamma_1 * \gamma_2) \leq \size(\gamma_1)+\size(\gamma_2)$.
\item Shortening: If $f \co X \to Y$ is $\lambda$-Lipschitz and
  $\gamma \co I \to X$ is a path in~$X$, then
  $\size(f \circ \gamma) \le \lambda\cdot \size(\gamma)$.
\end{enumerate}
\end{lemma}
As a result of point~\eqref{item:size-bound-length}, we will prefer to give
statements with hypotheses on the \emph{size} of paths and construct
paths with bounds on \emph{length}, even if we don't
necessarily need the length bounds for our applications.

Another central feature of our development is the following.
\begin{definition}
  For $X$ a locally simply-connected length space, $A$ an auxiliary
  space---now not necessarily simply-connected---and $H \co I \times A \to X$ a homotopy of maps from~$A$
  to~$X$, a \emph{trace} of $H$ is a path of the form
  $t \mapsto H(t,a)$ for fixed $a \in A$.
  The \emph{trace size} of~$H$
  is the maximum size of a trace:
  \[
    \tracesize(H) \coloneqq \sup_{a \in A}\, \size(H(\cdot,a)).
  \]
  If two maps $f,g \co A \to X$ are homotopic by a homotopy of trace
  size at most~$K$, then we write $f \sim_K g$.
\end{definition}

For a homotopy $H\co I \times I \to X$ between two paths $\gamma_0$ and $\gamma_1$, be
careful to distinguish between its \emph{size} and \emph{trace size}.
For instance, if $H$ has bounded size, then the $\gamma_i$ must also
have bounded size, while two paths that are very long can still have a
homotopy of bounded trace size.

The trace size of a homotopy between two paths is sensitive to the parameterization of
the domain of
the two paths, which in turn is sensitive to details like exactly how
one defines the concatenation operation on paths. We will specify the
parameterization when necessary.

One
thing to note now is that, if $\beta$ is a path of size~$R$ and
any $\epsilon > 0$, then, for any concatenatable $\gamma_1,\gamma_2$ and
suitable parameterization of the domain,
\[
  \gamma_1 * \beta * \overline{\beta} * \gamma_2 \sim_{R+\epsilon}
  \gamma_1 * \gamma_2.
\]
(The parameterization to make this work uses
a very small interval in the domain for $\beta * \overline{\beta}$ on the left hand side.)

\begin{lemma}
\label{lem:comphom}
Suppose $X, Y, Z$ are length spaces. If $f\colon X \to Y$ and
$g_0, g_1\co Y \to Z$ are maps with $g_0 \sim_K g_1$ via the homotopy
$\alpha\co I \times Y \to Z$, then $g_0 \circ f \sim_K g_1 \circ f$ via the
homotopy $f^*\alpha\co I \times X \to Z$ given by $(t,x) \mapsto
\alpha(t,f(x))$.
\end{lemma}
\begin{proof}
The only non-trivial point is the bound on the trace size, which
follows since every trace of $f^*\alpha$ is a trace of~$\alpha$.
\end{proof}

\subsection{Approximate path-lifting}
\label{subsecn:apl}
Suppose $X$, $Y$ are length spaces, $A$ is a topological space, and suppose we are given
maps $\phi\co X \to Y$ and $g\co A \to Y$. An \emph{approximate lift}
of~$g$ under~$\phi$ with constant~$K$ is a map $g'\co A \to X$ such that
$\phi\circ g' \sim_K g$.

Now suppose further that $X, Y$ are finite connected CW complexes,
equipped with complete length metrics, and $\phi\co X \to Y$ is
continuous and surjective on fundamental group. We consider the
problem of approximately lifting paths in~$Y$ under~$\phi$ to paths in~$X$.

In this subsection, the constants appearing in the conclusions depend
on $\phi\co X \to Y$ and on the other constants appearing
in the statements. We suppress their dependence on~$\phi$.

\begin{proposition}[Controlled approximate path-lifting]
\label{prop:apl}
Suppose $X$ and~$Y$ are finite connected CW complexes equipped with length metrics,
$\phi\co X \to Y$ is continuous, and
$\phi_*\co \pi_1(X) \to \pi_1(Y)$ is surjective. Then there exists $K>0$ so that for any
path $\gamma\co I \to Y$ joining endpoints $y_0$ to $y_1$, and any
preimages $x_i \in \phi^{-1}(y_i)$ for $i=0,1$ of these endpoints, there
exists an approximate lift $\gamma'\co I \to X$ with
$\gamma'(i)=x_i$ for $i=0,1$ with homotopy $H\co I \times I \to Y$
between $\gamma$ and $\phi \circ \gamma'$ of
trace size at most~$K$. Concretely,
\begin{itemize}
\item $H(0,t)=\gamma(t)$ and $H(1,t)=\phi \circ \gamma'(t)$ for
  $0 \leq t \leq 1$;
\item  $H(s,i)=y_i$ for $i=0,1$ and $0 \leq s \leq 1$; and
\item $\size(s \mapsto H(s, t)) \le K$ for each $t \in [0,1]$.
\end{itemize}
There exist constants $C_0$ and $C_1$ so that if $\gamma$ is rectifiable, then so is $\gamma'$, and 
$\ell(\gamma') < C_0+ C_1 \ell(\gamma)$.
\end{proposition}
\begin{figure}
  \[
    \mathcenter{\begin{tikzpicture}
        \node (X) at (0cm,0cm) {$\mfig{apl-1}$};
        \node (Y) at (5.5cm,0cm) {$\mfig{apl-0}$};
        \draw[->] (X) to node[above,cdlabel]{\phi} (Y);
      \end{tikzpicture}}
  \]
  \caption{Approximate path lifting. The traces of the homotopies are
    the red lines on the right; these must be of bounded size.}
  \label{fig:apl}
\end{figure}
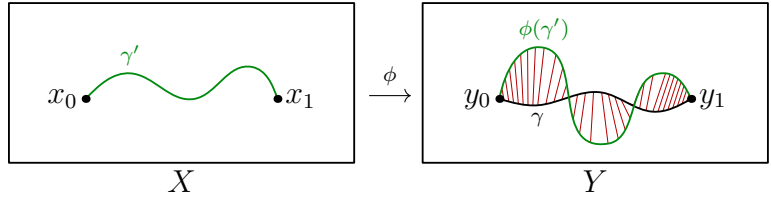

In other words, approximate lifting increases lengths by controlled amounts, and the failure of a
path to lift  is measured by a homotopy whose traces are
uniformly bounded in size, independent of the path~$\gamma$; see Figure~\ref{fig:apl}.
(We do not use the fact that lengths are increased by controlled
amounts in this paper, but it helps add motivation.)
We will
also say that
$\phi\co X \to Y$ satisfies the \emph{$K$-APL} condition.

Before proving the general statement, we first prove it for
loops~$\gamma$ of bounded length.

\begin{lemma}[Controlled approximate loop-lifting]
\label{lemma:call}
In the setup of Proposition \ref{prop:apl}, fix $C>0$. Then there
exist constants $L$ and $K$ with the following property. For any
$x \in X$, $y = \phi(x)$, and loop $\lambda\co I \to Y$ based at
$y$ with $\size(\lambda)\leq C$, there is a rectifiable loop
$\lambda'\co I \to X$ based at $x$ with $\ell(\lambda')<L$ so that
$\phi\circ \lambda' \sim_K \lambda$.
\end{lemma}

\begin{proof} First, fix $x \in X$ and $y = \phi(x)$. The length
  metrics on $X$ and~$Y$ induce norms
  on the fundamental groups $\pi_1(X,x)$ and $\pi_1(Y,y)$. Since
  $\size(\lambda) \le C$, the norm of $[\lambda] \in \pi_1(Y,y)$ is also bounded
  by~$C$.
  Since
  $\phi_*\co \pi_1(X,x)\to\pi_1(Y,y)$ is surjective, there exists
  $L(x)$ such that the image of the ball of radius $L(x)$ in
  $\pi_1(X,x)$ contains the ball of radius $C$ in $\pi_1(Y,y)$.  Now
  vary $x \in X$. Since the induced norms on $\pi_1(X,x)$ and $\pi_1(Y,\phi(x))$ vary
  continuously as a function of~$x$, we can take $L$ to be continuous
  and so, by
  compactness, $\sup_{x \in X}L(x)$ is finite; call this supremum~$L$. Thus
  there
  exists a loop $\lambda'$ based at $x$ of length at most $L$ for
  which $\phi\circ\lambda' \sim \lambda$.
  
  We must bound the trace size of the homotopy; in fact we bound its
  size. For any~$\lambda$ and
  $\lambda'$ as above,
  we can lift $\phi\circ \lambda'$ and $\lambda$ to paths in
  the universal cover~$\wt Y$. Fix lifts
  $\widetilde{\phi\circ \lambda'}$ and $ \widetilde{\lambda}$, respectively, joining
  common endpoints. The concatenation
  $\widetilde{\alpha}\coloneqq\widetilde{\phi\circ \lambda'}
  *\overline{\widetilde{\lambda}}$ is a loop in the simply-connected
  length CW complex~$\wtY$, so, since $\widetilde{\phi}$ is uniformly continuous,
  \[
    \diam(\widetilde{\alpha}) \leq
    \diam(\widetilde{\phi\circ\lambda'})+\diam(\widetilde{\lambda})   \leq
    \omega_{\widetilde{\phi}}(L)+C\eqqcolon D_1.
  \]
  By \cite[Lemma
  9.51]{drutu:kapovich:ggt}, $\wtY$ is uniformly simply-connected.
  This implies that $\widetilde{\alpha}$ is homotopic to a constant
  map via a homotopy whose image has diameter $K\coloneqq K(D_1)$, as desired.
\end{proof}

\begin{proof}[Proof of Proposition~\ref{prop:apl}]
  Pick (arbitrarily) a basepoint $x_* \in X$, set $y_* \coloneqq \phi(x_*)$, and
  pick an arbitrary constant $C > 0$. Divide $\gamma$ into sub-paths
  \[
    \gamma = \gamma_0 * \dots * \gamma_{n-1}
  \]
  with $\size(\gamma_i) \le C$. Let the endpoints of $\gamma_i$ be
  $z_i,z_{i+1} \in Y$, and for each $i$ pick a path $\rho_i$ from
  $y_*$ to $z_i$, of length less than $\diam(Y)$.
  Pick also paths $\rho_0', \rho_n'$ from $x_*$ to $x_0, x_1$,
  respectively, of length less than $\diam(X)$.
  Then $\lambda_i \coloneqq \rho_i * \gamma_i * \overline{\rho_{i+1}}$ is a loop based
  at~$y_*$ of size less then $2\diam(X)+C$. By Lemma~\ref{lemma:call}, there
  are constants $L_1, K_1$ so that, for each~$i$, there is a loop $\lambda_i'$ in $X$
  based at~$x_*$ with $\ell(\lambda_i') < L_1$ and
  $\phi \circ \lambda_i' \sim_{K_1} \lambda_i$.
  In addition,
  $\sigma_0 \coloneqq (\phi \circ \rho_0') * \overline{\rho_0}$ is a
  loop of size less than
  $\omega_{\wt \phi}(\diam Y) + \diam X$, so there are constants
  $L_2,K_2$ and a
  loop $\sigma_0'$ based at $x_*$ of length less than $L_2$ so that
  $\phi \circ \sigma_0' \sim_{K_2} \sigma_0$. Similarly pick an approximate lift
  $\sigma_n'$ of $\sigma_n \coloneqq \rho_n * \overline{(\phi \circ \rho_n')}$.
  Now set
  \begin{align*}
    K_3 &\coloneqq \max(K_1,K_2)\\
    D &\coloneqq \max(\diam(X), \omega_{\wt \phi}(\diam(Y))\\
    \gamma' &\coloneqq \overline{\rho'_0} * \sigma_0' * \lambda_1' * \lambda_2' *
              \dots * \lambda_{n-1}' * \sigma_n' * \rho'_{n}
  \end{align*}
   (with suitable parameterization of the domain for~$\gamma'$), so
   that
  \begin{align*}
    \phi \circ \gamma'
      &\sim_{K_3} \overline{(\phi \circ \rho_0')} * \sigma_0 *
        \lambda_1 * \lambda_2 * \dots * \lambda_{n-1} *
        \sigma_n * (\phi \circ \rho_{n}')\\
      &\sim_{D+\epsilon} \gamma_0 * \gamma_1 * \dots * \gamma_{n-1}
      = \gamma,
  \end{align*}
  as desired.

  To get the bounds on length in the case that $\gamma$ is
  rectifiable, choose the initial decomposition of $\gamma$ into
  sub-paths~$\gamma_i$ so that, for $i > 0$, $\ell(\gamma_i) \ge C$;
  then we have $n \le 1+\ell(\gamma)/C$. We can choose the paths
  $\lambda_i'$, $\rho_i'$, and $\sigma_i'$ all to have length bounded
  by a constant, which then gives the desired bound on $\ell(\gamma)$.
\end{proof}

\begin{remark}
  Lemma~\ref{lemma:call} gives a
  bound on the overall size of the homotopy, but
  Proposition~\ref{prop:apl} only gives a bound on the trace sizes.
\end{remark}

\subsection{Multi-valued dynamical systems}
\label{subsecn:multivalued}
We turn our attention back to dynamics. We think of two spaces and a pair of
maps between them:
\begin{center}
\begin{tikzcd}
& X_1 \arrow{ld}[swap]{\phi} \arrow{rd}{\pi} \\
X_0 & & X_0  
\end{tikzcd}
\end{center}
as a
\emph{multi-valued dynamical system}. We introduce an associated limit
space
and describe it
in two different ways, as in \cite{ishii:smillie:homotopy-shadowing},
but adopting slightly different notation.  Here is how to translate
between their (IS) and our (PT) notation:
\begin{align*}
  X^i_{IS}&=X_{i, PT};\\
  (\iota, \sigma)_{IS} = (\iota, f)_{IS} &= (\phi, \pi)_{PT};\\
  X^{+\infty}_{IS} &= \cJ_{PT}.
\end{align*}
(There are also minor differences in the indexing.)

Suppose $X_1$ and $X_0$ are compact topological spaces and
$\pi, \phi\co X_1 \rightrightarrows X_0$ are two continuous maps.
An \emph{orbit} is a
sequence $x=(x_1, x_2, \ldots)$ of points $x_i \in X_1$ with
$\pi(x_i)=\phi(x_{i+1})$ when both sides are defined. If $x_i$ is
defined for $1 \le i \le n$ for some $n \geq
1$, we get the space $X_n$ of orbits of
length~$n$. If $x_i$ is defined for all $1 \le i$, we get the \emph{limit space}
$\cJ \subset X_1^\mathbb{N}$ of one-sided infinite orbits. Note the typographical distinction between the abstract limit
space~$\cJ$ of a virtual endomorphism and the concrete limit space
$J_R\subset\CCa$ of a rational map~$R$.

With this setup, there are two families of canonical maps
\begin{align*}
  \phi^{n+1}_n, \pi^{n+1}_n\co X_{n+1} &\to X_n,\\
  \phi^{n+1}_n(x_1, x_2, \ldots, x_n, x_{n+1}) &\coloneqq (x_1, x_2, \ldots, x_n)\\
  \pi^{n+1}_n(x_1, x_2, \ldots, x_n, x_{n+1}) &\coloneqq (x_2, \ldots, x_n,x_{n+1}).
\end{align*}
We also set $\phi^1_0 \coloneqq \phi$ and $\pi^1_0 \coloneqq \pi$.
We can compose these to get maps $\phi^n_k, \pi^n_k \co X_n \to
X_k$ for $0 \le k < n$.
We recall a convention from
\S\ref{subsecn:notation}.

\begin{convention}
In our indexing convention, the index of the domain appears
as a superscript and the index of the codomain appears as a subscript.
This way, composition corresponds to ``contraction'' of
indices, as is conventional in tensor notation.
\end{convention}

We can
present the space $\cJ$ as an inverse limit of the
sequence~$\phi^{n+1}_n$:
\[
  \begin{tikzpicture}[x=1.6cm,y=1.5cm]
    \node (Xdots) at (0,0) {$\cdots$};
    \node (Xn) at (1,0) {$X_n$};
    \node (Xn1) at (2,0) {$X_{n-1}$};
    \node (Xdots2) at (3,0) {$\cdots$};
    \node (X1) at (4,0) {$X_1$};
    \node (X0) at (5,0) {$X_0$};
    \node (J) at (0,-1) {$\cJ$};
    \draw[->] (Xdots) to node[above,cdlabel]{\phi^{n+1}_n} (Xn);
    \draw[->] (Xn) to node[above,cdlabel]{\phi^n_{n-1}} (Xn1);
    \draw[->] (Xn1) to node[above,cdlabel]{\phi^{n-1}_{n-2}} (Xdots2);
    \draw[->] (Xdots2) to node[above,cdlabel]{\phi^2_1} (X1);
    \draw[->] (X1) to node[above,cdlabel]{\phi^1_0} (X0);

    \draw[->] (J) to node[above left=-2mm,cdlabel]{\phi^\infty_n} (Xn);
    \draw[->,bend right=8] (J) to node[above left=-2mm,cdlabel,pos=0.7]{\phi^\infty_{n-1}} (Xn1);
    \draw[->,bend right=8] (J) to node[above left=-1.5mm,cdlabel,pos=0.72]{\phi^\infty_1} (X1);
    \draw[->,bend right=15] (J) to node[above left=-1.5mm,cdlabel,pos=0.83]{\phi^\infty_0} (X0);
  \end{tikzpicture}
\]
where the $\phi^\infty_n\co \cJ \to X_n$ are analogues of $\phi^n_k$:
\[
  \phi^\infty_n (x_1,x_2,\dots) \coloneqq (x_1,x_2,\dots,x_n).
\]
There is also a canonical map $f\colon \cJ \to \cJ$ induced by the
one-sided shift, a kind of analogue of $\pi^{n+1}_n$:
\[ f(x_1,x_2,x_3 \ldots) \coloneqq (x_2,x_3, \ldots).\]

There are two natural modifications of a pair of maps $(\pi,\phi)$: we
can \emph{iterate} it, replacing the pair by
\begin{equation}\label{eq:iterate}
  \pi^n_0, \phi^n_0 \co X_n \rightrightarrows X_0,
\end{equation}
or we can \emph{reindex} it, replacing the pair by
\begin{equation}\label{eq:reindex}
  \pi^n_{n-1}, \phi^n_{n-1} \co X_n \rightrightarrows X_{n-1}.
\end{equation}
Neither of these operations changes the limit space~$\cJ$, but
iterating replaces the dynamics of $f$ on $\cJ$ by $f^{\circ n}$, while
reindexing does not change~$f$.

We now restrict attention to expanding systems, as in
\cite{ishii:smillie:homotopy-shadowing}, but adopting terminology of
Nekrashevych \cite{Nekrashevych14:CombModel} and the second author
\cite{Thurston20:Characterize}.
We continue with some definitions.

\begin{definition}\label{def:ve}
  \textbf{Virtual endomorphisms.} A pair of
continuous maps $\pi, \phi\co X_1 \rightrightarrows X_0$ between
topological spaces is a
\emph{virtual endomorphism} if $\pi$ is a covering map of finite
degree.
\end{definition}

\begin{convention}
In this section we are exclusively concerned with virtual endomorphisms
where $X_0, X_1$ are finite connected CW complexes, $X_0$ is equipped with a
complete length metric $d_0$, and $X_1$ is equipped with the
length metric~$d_1$ induced by the covering~$\pi$, i.e., the length metric on~$X_1$
so that $\pi$ is a local isometry.
\end{convention}

\begin{definition}\label{def:contracting}
  Suppose $0<\lambda<1$.  The virtual endomorphism $\pi, \phi\co X_1 \rightrightarrows X_0$ is \emph{$\lambda$-backward-contracting} if $d_0(\phi(a),\phi(b))\leq \lambda d_1(a,b)$ for all
$a,b \in X_1$, i.e., $\phi$ is a uniform contraction. Equivalently, we
say the virtual endomorphism is $\lambda^{-1}$-\emph{forward-expanding}.  A virtual endomorphism is backward contracting
(equivalently, forward expanding) if it is
$\lambda$-backward-contracting for some $0<\lambda<1$.
\end{definition}

\begin{definition} \label{def:recurrent} 
  The virtual endomorphism $\pi, \phi\co X_1 \rightrightarrows X_0$ is
  \emph{recurrent} if $X_0$ and $X_1$ are connected and
  $\phi_*\co\pi_1(X_1) \to \pi_1(X_0)$ is surjective.
\end{definition}
If $\pi, \phi\co X_1 \rightrightarrows X_0$ is recurrent, then $X_n$ is connected for each $n$. 

In the setting of a virtual endomorphism, the map $\pi^{n+1}_n \co
X_{n+1} \to X_n$ defined above is
also a covering map.
We will use the fact that $X_{n+1}$ is a pullback, which
concretely gives the following lemma, among other pullback diagrams.

\begin{lemma}
  \label{lem:pullback}
Given maps $F\colon A \to X_1$ and $G\co A \to X_n$ for which $\phi
\circ F = \pi^n_0\circ G$, there exists a unique map $A \to X_{n+1}$
such that the following diagram commutes:
\[ 
\begin{tikzpicture}[x=1.8cm,y=1.8cm]
  \node (X00) at (1,0) {$X_1$};
 \node (X01) at (2,0) {$X_0$};
 \node (X02) at (1,1) {$X_{n+1}$};
 \node (X03) at (2,1) {$X_{n}$};
 \node (X04) at (0,2) {$A$};
 
 \draw[->,color=black,bend left=20] (X04) to node[above right=-1mm,cdlabel]{G} (X03);
 \draw[->,color=black,bend right=20] (X04) to node[left=1mm,cdlabel]{F} (X00);
 \draw[->,densely dashed, color=black] (X04) to node[above left=-1mm,cdlabel]{} (X02);
 \draw[->,color=black] (X03) to node[right,cdlabel]{\pi^n_0} (X01);
 \draw[->,color=black] (X00) to node[above,cdlabel]{\phi} (X01);
 \draw[->, color=black] (X02) to node[above,cdlabel]{\phi^{n+1}_n} (X03);
 \draw[->, color=black] (X02) to node[right,cdlabel]{\pi^{n+1}_1} (X00);
\end{tikzpicture}
\]
\end{lemma}

Via pullback of length metrics under~$\pi^n_0$, each space $X_n$
inherits a length metric $d_n$. The degree of $\pi^n_0$ is $d^n$, and the complexes $X_n$ are locally finite uniformly in $n$, so $\diam X_n \to \infty$ as $n \to \infty$.

\begin{remark}
  From a virtual endomorphism
  $\pi,\phi \co X_1 \rightrightarrows X_0$, we will not see every cover
  of~$X_0$ among the covers $\pi^n_0 \co X_n \to X_0$ (or their normal
  closures). Studying the
  exact covers that appear leads to a
  very interesting group, the \emph{iterated monodromy group}
  $\IMG(\pi,\phi)$, which is a natural quotient of  $\pi_1(X_0)$. This group is the deck group of a 
Galois cover $\wt X_0'$ of~$X_0$ usually different from  the universal cover.
  Instead of measuring
  the size by lifting paths to the universal cover as in
  Definition~\ref{def:size}, we could get a
  different (smaller) notion by lifting to $\wt X_0'$ instead.
  The difference is inessential for this paper.
\end{remark}

\begin{definition}
A \emph{homotopy pseudo-orbit} $(x, \alpha)=(x_i, \alpha_i)$ is a sequence, for $i \ge 1$, of points $x_i \in X_1$ and of paths $\alpha_i \co [0,1] \to X_0$ such that 
\begin{itemize}
\item $\alpha_i(0)=\pi(x_i)$, $\alpha_i(1)=\phi(x_{i+1})$
\item $\ell(\alpha_i)\leq C$ for some $C \geq 0$ independent of $i$. 
\end{itemize}
Two homotopy pseudo-orbits $(x,\alpha)=(x_i, \alpha_i)$ and $(x', \alpha')=(x'_i, \alpha'_i)$ are
\emph{homotopic} if there is a sequence $\beta = (\beta_i)$, $i \geq 1$, 
of paths $\beta_i\co [0,1] \to X_1$ with 
\begin{itemize}
\item $\beta_i(0)=x_i$ and $\beta_i(1)=x_i'$;
\item  for $i \geq 1$, $\alpha_i * (\phi\circ \beta_{i+1}) \sim (\pi\circ \beta_i) * \alpha_i'$; and 
\item $\ell(\beta_i) < B$ for some $B>0$ independent of $i$.  
\end{itemize}
\end{definition}
See Figure \ref{fig:homotopic-p-orbits} and \cite[Definition 6.4;
Figure 5a]{ishii:smillie:homotopy-shadowing}.
Following Ishii and Smillie, we use lengths
in the condition on $\alpha_i$ and $\beta_i$, but size would work
equally well, since these are paths in a CW complex.

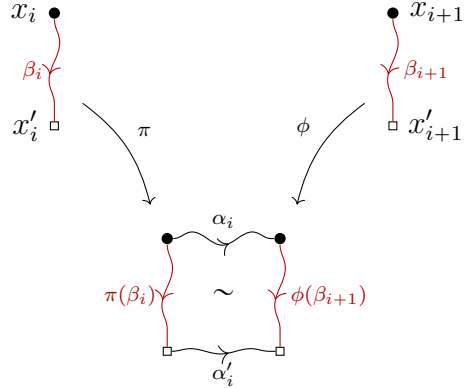
\begin{figure}
  \[
    \begin{tikzpicture}[x=1.5cm,y=1.5cm]
      \node[shape=circle,fill,inner sep=1.5pt,label=left:{$x_i$}] (xi1) at (0,3) {};
      \node[shape=rectangle,draw,inner sep=1.5pt,label=left:{$x_i'$}] (xi1') at (0,2) {};
      \draw[decorate,decoration={snake,amplitude=2pt,segment length=1cm},
        color=dark-red,markdir] (xi1) to node[left,cdlabel]{\beta_i} (xi1');

      \node[shape=circle,fill,inner sep=1.5pt] (yi1) at (1,1) {};
      \node[shape=rectangle,draw,inner sep=1.5pt] (yi1') at (1,0) {};
      \draw[decorate,decoration={snake,amplitude=2pt,segment
        length=1cm},color=dark-red,markdir] (yi1) to node[left,cdlabel]{\pi(\beta_i)} (yi1');
      \node[shape=circle,fill,inner sep=1.5pt] (yi) at (2,1) {};
      \node[shape=rectangle,draw,inner sep=1.5pt] (yi') at (2,0) {};
      \node at (1.5,0.5) {$\sim$};
      \draw[decorate,decoration={snake,amplitude=2pt,segment length=1cm},
        color=dark-red,markdir] (yi) to node[right,cdlabel]{\phi(\beta_{i+1})} (yi');
      \draw[decorate,decoration={snake,segment length=0.8cm},
        markdir] (yi1) to node[above,cdlabel]{\alpha_i} (yi);

      \draw[decorate,decoration={snake,amplitude=1pt,segment length=0.8cm},
        markdir] (yi1') to node[below,cdlabel]{\alpha'_i} (yi');
      
      \node[shape=circle,fill,inner sep=1.5pt,label=right:{$x_{i+1}$}] (xi) at (3,3) {};
      \node[shape=rectangle,draw,inner sep=1.5pt,label=right:{$x_{i+1}'$}] (xi') at (3,2) {};
      \draw[decorate,decoration={snake,amplitude=2pt,segment length=1cm},
        color=dark-red,markdir] (xi) to node[right,cdlabel]{\beta_{i+1}} (xi');

      \draw[bend left=20,->] (0.25,2.2) to node[above right,cdlabel]{\pi}
        (0.85,1.3);
      \draw[bend right=20,->] (2.75,2.2) to node[above left,cdlabel]{\phi}
        (2.15,1.3);
    \end{tikzpicture}
  \]
  \caption{In homotopic pseudo-orbits, $\alpha_i * \phi(\beta_{i+1})
    \sim \pi(\beta_i) \cdot \alpha_i'$.}
  \label{fig:homotopic-p-orbits}
\end{figure}

The following result appears in \cite[\S 7, 8]{ishii:smillie:homotopy-shadowing}. 
\begin{citethm}[Homotopy shadowing]
\label{thm:homotopy_shadowing}
Suppose $\pi, \phi\co X_1 \rightrightarrows X_0$ is forward-expanding.
Then every homotopy pseudo-orbit is homotopic to an orbit, and this
orbit is unique.
\end{citethm}

We will need a generalization of the homotopy shadowing theorem to a
setting where the orbit depends on a parameter $a$ lying in a space
$A$.  We develop this notion in close parallel to the above notions
and Ishii and Smillie.

\begin{definition} \label{def:homotopy-pseudo-orbit}
  For $\pi,\phi \co X_1 \rightrightarrows X_0$ a virtual endomorphism
  between locally simply-connected length spaces and $A$ an auxiliary
  space, a \emph{family $(x,\alpha)$ of homotopy pseudo-orbits parameterized by~$A$ of trace size at most~$K$}
  is a sequence of maps $x\coloneqq (x_i \co A \to X_1)_{i \geq 1}$ and a sequence of homotopies
  $\alpha\coloneqq (\alpha_i \co I \times A \to X_0)_{i \geq 1}$, so that 
  \begin{enumerate}
  \item $\alpha_i$ is a
  homotopy from $\pi \circ x_i$ to $\phi \circ x_{i+1}$, in the sense that
  $\alpha_i(0,\cdot)=\pi \circ x_i$ and $\alpha_i(1, \cdot)=\phi \circ
  x_{i+1}$, and
  \item there exists a constant $K<\infty$ so that for each $i \geq
    1$, we have
    \[\tracesize(\alpha_i) \leq K.\]
  \end{enumerate}
   See Figure~\ref{fig:homotopy-pseudo-orbit}.
\end{definition}

\begin{definition}
  Two families $(x, \alpha)$ and $(x', \alpha')$ of homotopy
  pseudo-orbits parameterized by $A$ are \emph{homotopic} if there
  exists a constant $B$ and a sequence $\beta=(\beta_i)_{i \geq 1}$ of
  homotopies $\beta_i\co I \times A \to X_1$ such that
  \begin{enumerate}
  \item $\beta_i$ is a homotopy from $x_i$ to $x_i'$, in the sense
    that $\beta_i(0,\cdot)=x_i$ and $\beta_i(1,\cdot)=x_i'$;
  \item for each $i \geq 1$, we have $\tracesize(\beta_i) \leq B$; and
  \item for each $i \geq 1$, the map
    $\alpha_i * (\phi\circ \beta_{i+1})$ is homotopic to
    $(\pi\circ \beta_i) * \alpha_i'$ in the following sense. There is
    a map $H_i\co I \times I \times A \to X_0$ such that
    $H_i(0, \cdot, \cdot)=\alpha_i * (\phi\circ \beta_{i+1})$,
    $H_i(1, \cdot, \cdot)=(\pi\circ \beta_i) * \alpha_i'$, and for all
    $0 \leq t \leq 1$, $H_i(t,0,\cdot)=\pi \circ x_i$ and
    $H_i(t,1,\cdot)=\phi\circ x_{i+1}'$.
  \end{enumerate}
  \end{definition}
  These conditions in (3) guarantee that the homotopic squares appearing in
  Figure \ref{fig:homotopic-p-orbits} remain squares throughout the
  interpolating maps $H_i(\cdot,\cdot,a)$. We do not require a size
  bound on the homotopies~$H_i$ in condition (3). 
  
  Here is our generalization of Ishii-Smillie's homotopy shadowing result. 

\begin{figure}
  \[
    \begin{tikzpicture}[x=1cm,y=1cm]
      \path[use as bounding box] (-1.6,-1.7) rectangle (8.1,1.1);
      \node (X00) at (0,0) {$X_0$};
      \node (X01) at (2,0) {$X_0$};
      \node (X02) at (4,0) {$X_0$};
      \node (X03) at (6,0) {$X_0$};
      \node (X04) at (8,0) {$\cdots$};

      \node (X10) at (1,1) {$X_1$};
      \node (X11) at (3,1) {$X_1$};
      \node (X12) at (5,1) {$X_1$};
      \node (X13) at (7,1) {$X_1$};
      \node at (8,1) {$\cdots$};
      
      \node(A) at (-1.5,-1.5) {$A$};
      
      \draw[->,color=dark-green] (X10) to node[above left=-1mm,cdlabel]{\phi} (X00);
      \draw[->,color=dark-green] (X11) to node[above left=-1mm,cdlabel]{\phi} (X01);
      \draw[->,color=dark-green] (X12) to node[above left=-1mm,cdlabel]{\phi} (X02);
      \draw[->,color=dark-green] (X13) to node[above left=-1mm,cdlabel]{\phi} (X03);

      \draw[->,color=dark-red] (X10) to node[above right=-1mm,cdlabel]{\pi} (X01);
      \draw[->,color=dark-red] (X11) to node[above right=-1mm,cdlabel]{\pi} (X02);
      \draw[->,color=dark-red] (X12) to node[above right=-1mm,cdlabel]{\pi} (X03);
      \draw[->,color=dark-red] (X13) to node[above right=-1mm,cdlabel]{\pi} (X04);

      \draw[->] (A) to[out=10,in=270] node[left,cdlabel,pos=0.65]{x_1} (X10);
      \draw[->] (A) to[out=5,in=270] node[right,cdlabel,pos=0.75]{x_2} (X11);
      \draw[->] (A) to[out=0,in=270] node[right,cdlabel,pos=0.835]{x_3} (X12);
      \draw[->] (A) to[out=-5,in=270] node[right,cdlabel,pos=0.866]{x_4} (X13);
      
      \node at (1.6,-0.55) {$\sim_K$};
      \node at (3.75,-0.55) {$\sim_K$};
      \node at (5.8,-0.55) {$\sim_K$};
    \end{tikzpicture}
  \]
  \caption{A homotopy pseudo-orbit of maps of $A$}
  \label{fig:homotopy-pseudo-orbit}
\end{figure}

\begin{theorem}
  \label{thm:Ashadowing}
  Suppose $\pi, \phi\co X_1 \rightrightarrows X_0$ is a
  $\lambda$-backward-contracting virtual endomorphism of CW complexes
  with induced dynamics $f \co \cJ \to \cJ$ on its limit space.
  Let $(x, \alpha)=(x_i, \alpha_i)_{i\geq 1}$ be a
  family of homotopy pseudo-orbits
parameterized by $A$ of trace size at most $K$.
\begin{enumerate}
\item There exists 
\begin{enumerate}
\item a map $x^\infty \co A \to \cJ$, i.e., a family of orbits of~$f$
  parameterized by $A$, and 
\item a homotopy $\beta=(\beta_i^\infty)_{i \geq 1}$ from $(x,\alpha)$ to $x^\infty$.
\end{enumerate}
\item We have $\tracesize(\beta^\infty_i) \le K'\coloneqq K/(1-\lambda)$ for all $i \geq 1$.  
\end{enumerate}
Furthermore, $x^\infty$ is unique among all such maps with uniformly
bounded $\tracesize(\beta_i^\infty)$.
\end{theorem}

To make sense of part 1(b) of the statement, we regard an
orbit as a family of homotopy pseudo-orbits parameterized by~$A$,
namely $(x^\infty_i, \alpha^\infty_i)_{i \geq 1}$, where each
$\alpha^\infty_i$ is a constant homotopy.

\begin{proof}
  This is a straightforward modification of the proof of \cite[Theorem
  7.1]{ishii:smillie:homotopy-shadowing}. Since we will need the
  notation later, we shamelessly copy their proof, more or less word
  for word,
  with slight adjustments to indexing.
  We let the pseudo-orbit $x$ now depend on a parameter $a\in A$,
  so that $x\coloneqq(x_i\co A \to X_1)_{i \geq 1}$,
  and we denote the
  collection of homotopies by
  $\alpha\coloneqq(\alpha_i\co I \times A \to X_0)_{i \geq 1}$. To ease
  notation, we think of our homotopies $\alpha_i$ and $\beta_j$ below as
  paths in the space of continuous maps from $A$ to $X_0$ and $X_1$,
  respectively.

  We inductively define a sequence of families of homotopy
  pseudo-orbits with successively smaller traces as follows. Set
  $x_i^0\coloneqq x_i$ and
  $\alpha_i^0\coloneqq\alpha_i$. Suppose that a family of homotopy
  pseudo-orbits $(x_i^n,\alpha_i^n)_{i \geq 1}$ is defined.
  Then, since $\pi$ is a covering and
  $\alpha_i^n(0,\cdot)=\pi\circ x_i^n$, there exists a unique lift
  $\beta_i^n\co I \times A \to X^1$ of $\alpha_i^n$ by $\pi$ so that
  $\beta_i^n(0,\cdot)=x_i^n$, by the homotopy lifting property
  of covering maps. Put
  $\alpha_i^{n+1}\coloneqq\phi\circ \beta_{i+1}^n$ and
  $x_i^{n+1}\coloneqq\beta_i^n(1)$. Then, we have
  $\pi\circ x_i^{n+1}=\pi\circ\beta_i^n(1)=\alpha_i^n(1)=\phi(x_{i+1}^n)=\phi(\beta_{i+1}^n(0))=\alpha_i^{n+1}(0)$
  and $\phi\circ x_{i+1}^{n+1}=\phi\circ \beta_{i+1}^n(1)=\alpha_i^{n+1}(1)$.
  This means that, once we verify a trace size bound,
  $((x_i^{n+1}), (\alpha_i)^{n+1}))_{n+1}\eqqcolon
  (x^{n+1}, \alpha^{n+1})$ is a family of homotopy pseudo-orbits.

Contraction implies that the length of the traces of the homotopies
$\alpha_i^n$ are bounded by $K\lambda^n$ for $n \geq 1$; this is
\cite[Lemma 7.2]{ishii:smillie:homotopy-shadowing}. Concatenating the
homotopies $\alpha_i^n$ for $n=1, 2, \ldots$ and scaling the time
parameters in the homotopy to consecutive intervals in $[0,1)$  as in
their proof, we obtain, for $i \ge 1$, a sequence of maps
$\alpha^\infty_i\co [0,1)\times A \to X_0$ of trace size
$K' \coloneqq K/(1-\lambda)$.

To get a map defined on $[0,1]\times A$, we need to say a little more.
First, as in their proof, for fixed $a \in A$, the path
$\alpha^\infty_i(t,a)$ is Cauchy as $t \to 1$ in the sense that, for
any $\epsilon > 0$, there is $\delta < 1$ so that for
$t_0,t_1 > \delta$, we have
$d_{X_0}(\alpha^\infty_i(t_0,a),\alpha^\infty_i(t_1,a)) < \epsilon$.
Furthermore these paths are uniformly Cauchy as $a$ varies. There is
therefore a well-defined limit $\alpha_i^\infty(1,a)$, and the
continuous functions $\alpha_i^\infty(t,\cdot) \co A \to X_0$ converge
uniformly to $\alpha_i^\infty(1,\cdot)$. By the Uniform Limit Theorem,
the limiting function $\alpha_i^\infty$ restricted to $\{1\}\times A$
is therefore continuous. The standard proof of
the Uniform Limit Theorem shows that, in fact, we get a continuous function
$\alpha^\infty_i\co [0,1]\times A \to X_0$, as desired.

We also have sequences of maps $\beta_i^n\co I \times A \to X_1$ and
concatenating the $\beta_i^n$'s for $n=1, 2, 3, \ldots$ and extending
by the Uniform Limit Theorem yields $\beta_i^\infty$, a lift of
$\alpha_i^\infty$ under $\pi$. We put
$x^\infty_i\coloneqq\beta_i^\infty(\cdot, 1)$ and note that $x^\infty$
defines a family of orbits, i.e., a map $x^\infty\co A \to \cJ$; the
associated homotopies in Def.~\ref{def:homotopy-pseudo-orbit} are constant. By
construction, $\beta^\infty$ gives a homotopy between the family of
homotopy pseudo-orbits $(x,\alpha)$ and the family of
orbits~$x^\infty$. Since $\pi$ is an isometry, the trace sizes of the
$\beta_i^\infty$ are bounded by $K'$, as
required.
\end{proof}

In our later applications, it will be useful to restate
Theorem~\ref{thm:Ashadowing} in terms of maps into
the~$X_n$.  For its proof we will need the following lemma,
illustrated in Figure~\ref{fig:homotopy-lift}.

\begin{lemma}\label{lem:homotopy-lift} 
  Suppose $(x, \alpha)$ is a family of homotopy pseudo-orbits
  parameterized by $A$, with
  trace size at most~$K$. Then for each $n \geq 1$ there are unique lifted
  maps
  $\wt x_n \co A \to X_n$ and homotopies $\wt \alpha_n \co I \times A \to X_n$
  so that $\pi^n_1 \circ \wt x_n = x_n$, $\pi^n_0 \circ \wt \alpha_n = \alpha_n$,
  and $\wt \alpha_n$ is a homotopy from $\wt x_n$ to $\phi^{n+1}_n \circ \wt x_{n+1}$ of trace size at most~$K$.
\end{lemma}
\begin{figure}
  \[
    \begin{tikzpicture}[x=1.3cm,y=1.3cm]
      \path[use as bounding box] (-1.6,-1.8) rectangle (8.2,4.6);
      \node (X00) at (0,0) {$X_0$};
      \node (X01) at (2,0) {$X_0$};
      \node (X02) at (4,0) {$X_0$};
      \node (X03) at (6,0) {$X_0$};
      \node (X04) at (8,0) {$X_0$};

      \node (X10) at (1,1) {$X_1$};
      \node (X11) at (3,1) {$X_1$};
      \node (X12) at (5,1) {$X_1$};
      \node (X13) at (7,1) {$X_1$};

      \node (X20) at (2,2) {$X_2$};
      \node (X21) at (4,2) {$X_2$};
      \node (X22) at (6,2) {$X_2$};

      \node (X30) at (3,3) {$X_3$};
      \node (X31) at (5,3) {$X_3$};

      \node (X40) at (4,4) {$X_4$};
      
      \node(A) at (-1.5,-1.5) {$A$};
      
      \draw[->,color=dark-green] (X10) to node[above left=-1mm,cdlabel]{\phi} (X00);
      \draw[->,color=dark-green] (X11) to node[above left=-1mm,cdlabel]{\phi} (X01);
      \draw[->,color=dark-green] (X12) to node[above left=-1mm,cdlabel]{\phi} (X02);
      \draw[->,color=dark-green] (X13) to node[above left=-1mm,cdlabel]{\phi} (X03);

      \draw[->,color=dark-green] (X20) to node[above left=-1mm,cdlabel]{\phi^2_1} (X10);
      \draw[->,color=dark-green] (X21) to node[above left=-1mm,cdlabel]{\phi^2_1} (X11);
      \draw[->,color=dark-green] (X22) to node[above left=-1mm,cdlabel]{\phi^2_1} (X12);

      \draw[->,color=dark-green] (X30) to node[above left=-1mm,cdlabel]{\phi^3_2} (X20);
      \draw[->,color=dark-green] (X31) to node[above left=-1mm,cdlabel]{\phi^3_2} (X21);

      \draw[->,color=dark-green] (X40) to node[above left=-1mm,cdlabel]{\phi^3_2} (X30);

      \draw[->,color=dark-red] (X10) to node[above right=-1mm,cdlabel]{\pi} (X01);
      \draw[->,color=dark-red] (X11) to node[above right=-1mm,cdlabel]{\pi} (X02);
      \draw[->,color=dark-red] (X12) to node[above right=-1mm,cdlabel]{\pi} (X03);
      \draw[->,color=dark-red] (X13) to node[above right=-1mm,cdlabel]{\pi} (X04);

u      \draw[->,color=dark-red] (X20) to node[above right=-1mm,cdlabel]{\pi^2_1} (X11);
      \draw[->,color=dark-red] (X21) to node[above right=-1mm,cdlabel]{\pi^2_1} (X12);
      \draw[->,color=dark-red] (X22) to node[above right=-1mm,cdlabel]{\pi^2_1} (X13);

      \draw[->,color=dark-red] (X30) to node[above right=-1mm,cdlabel]{\pi^3_2} (X21);
      \draw[->,color=dark-red] (X31) to node[above right=-1mm,cdlabel]{\pi^3_2} (X22);

      \draw[->,color=dark-red] (X40) to node[above right=-1mm,cdlabel]{\pi^4_3} (X31);

      \draw[->] (A) to[out=10,in=270] node[right,cdlabel,pos=0.65]{x_1=\wt x_1} (X10);
      \draw[->] (A) to[out=5,in=270] node[right,cdlabel,pos=0.75]{x_2} (X11);
      \draw[->] (A) to[out=0,in=270] node[right,cdlabel,pos=0.835]{x_3} (X12);
      \draw[->] (A) to[out=-5,in=270] node[right,cdlabel,pos=0.866]{x_4} (X13);

      \draw [->] (A) to[out=80,in=160] node[above,cdlabel,pos=0.8]{\wt x_2} (X20);
      \draw [->] (A) to[out=85,in=160] node[above,cdlabel,pos=0.83]{\wt x_3} (X30);
      \draw [->] (A) to[out=90,in=160] node[above,cdlabel,pos=0.86]{\wt x_4} (X40);
      
      \node at (1.8,-0.5) {$\sim_K$};
      \node at (3.6,-0.5) {$\sim_K$};
      \node at (5.6,-0.5) {$\sim_K$};

      \node at (0.8,1.7) {$\sim_K$};
      \node at (1.8,2.7) {$\sim_K$};
      \node at (2.8,3.7) {$\sim_K$};
    \end{tikzpicture}
  \]
  \caption{Lifting a family of homotopy pseudo-orbits parameterized by $A$ to covering spaces}
  \label{fig:homotopy-lift}
\end{figure}
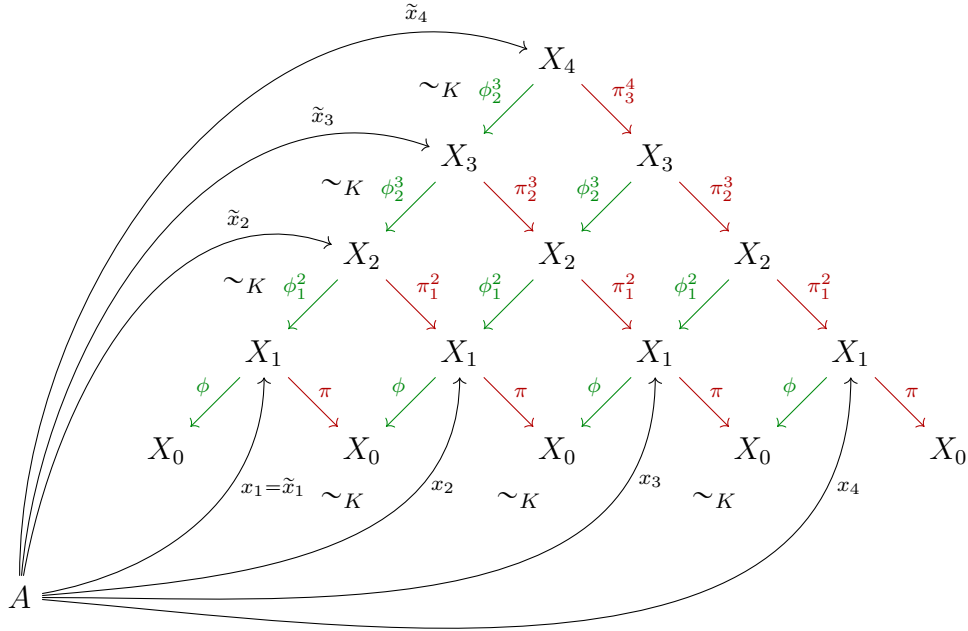

\begin{proof}
  We proceed by induction on~$n$, first constructing $\wt \alpha_n$ and
  then $\wt x_{n+1}$. Recall that $\alpha_n(0)=\pi \circ x_n$ and
  $\alpha_n(1)=\phi \circ x_{n+1}$. We start by setting
  $\wt x_1 \coloneqq x_1$. If we have constructed $\wt x_n$, then by
  unique homotopy lifting applied to the covering map $\pi^n_0$, there is
  a unique function $\wt \alpha_n$ that is a lift of~$\alpha_n$ with
  starting point $\wt x_n$; since $\wt \alpha_n$ is a lift, the trace
  size is still~$K$, as desired. We next construct $\wt x_{n+1}$. Let
  $\wt x_{n+1,n}\co A \to X_n$ be $\wt \alpha_n(1)$.
  Then, by the defining property of~$\alpha_n$, we have
  $\pi^n_0 \circ \wt x_{n+1,n} = \phi \circ x_{n+1}$.
  Since we have a pullback square (Lemma~\ref{lem:pullback}), there
  is a unique map
  $\wt x_{n+1} \co A \to X_{n+1}$ compatible with these two
  projections.
\end{proof}

\begin{taggedthm}{G\/$'$}
\label{thm:Ashadowing-cover}
  Suppose $\pi, \phi\co X_1 \rightrightarrows X_0$ is a
  $\lambda$-backward-contracting virtual endomorphism of CW complexes
  with limit space~$\cJ$, and $(x,\alpha)=(x_i, \alpha_i)_{i\geq 1}$ is a
  family of homotopy pseudo-orbits parameterized by $A$ of trace size at most~$K$.
Let $x^\infty$ and~$\beta$ be the families of orbits parameterized by $A$ and
homotopies given by Theorem~\ref{thm:Ashadowing}, and let
$\wt x\coloneqq (\wt x_n\co A \to X_n)_{n \geq 1}$ and
$\wt x^\infty=(\wt x^\infty _n\co A \to X_n)_{n \geq 1}$ be the families of maps
given by Lemma~\ref{lem:homotopy-lift} from $x$ and~$x^\infty$. Then for each $n \geq 0$, we
have $\wt x^\infty_n = \phi^{n+1}_n \circ \wt x^\infty_{n+1}$ and
$\wt x_n \sim_{K'} \wt x^\infty_n$, where $K'\coloneqq K/(1-\lambda)$.
\end{taggedthm}

\begin{proof}
  Immediate from Theorem~\ref{thm:Ashadowing} and
  Lemma~\ref{lem:homotopy-lift}.
\end{proof}

\subsection{Proof of Theorem \ref{thm:lc}\label{subsecn:lc}}

 An inverse limit of connected spaces is connected, so $\cJ$ is connected.  

  We now show that
  $f$ is a covering map. Let
  $x^\infty=(x_0, x_1, \ldots) \in X_0 \times X_1 \times \ldots$
  represent an element of $\cJ$. A finite $CW$ complex is locally
  contractible \cite[Prop.~A.4]{Hatcher02:AlgebraicTopology}, so there is a
  connected neighborhood $U$ of $x_0\in X_0$ which is contained in a
  contractible set. The set $U$ is then evenly covered by $\pi$; let
  $\wtU_i$, for $i=1, \ldots, \deg(\pi)$, be the components of its
  preimages in~$X_1$.
  The induced map $f\colon \cJ \to \cJ$ is the
  pullback of the covering map~$\pi$ under $\phi^\infty_0\co \cJ \to X_0$.
  This  implies that the neighborhood
  $(\phi_0^\infty)^{-1}(U)$ of $x^\infty$ is evenly covered by the
  neighborhoods $(\phi^\infty_1)^{-1}(\wtU_i)$ under~$f$.

  We now show that $\cJ$ is locally path-connected and hence locally
  connected. In fact we will
  show that $\cJ$ is \emph{weakly locally path-connected} at every
  point: for all $x \in \cJ$ and every open neighborhood $U \ni x$,
  there is a smaller open neighborhood $x \in V \subset U$ so that any
  two points in~$V$ can be connected by a path in~$U$.
  This is enough, since weak
  local path-connectivity implies that path components of open sets
  are open.
  
  We switch to thinking of $\cJ$ as a subset of $(X_1)^\infty$.
  Let $d_1$ denote the length metric on~$X_1$.
  Fix $p^\infty = (p_1, p_2, \ldots) \in \cJ$ and $\epsilon_0 > 0$. For $m \ge 1$ and
  $\epsilon < \epsilon_0$, let $U_{m,\epsilon}$ be those points
  $q^\infty=(q_1, q_2, \ldots) \in \cJ$ for which
  $d_1(p_i, q_i) < \epsilon$ for $i=1, \ldots, m$.
  The definition of the product topology says that the
  for any $\epsilon_0 > 0$ the $U_{m,\epsilon}$ are a
  neighborhood basis of $p^\infty$.
  Using local contractibility and Lemma~\ref{lemma:evenly_covered},
  choose $\epsilon_0$ so that balls in
  $X_0$ of radius
  $\epsilon_0$ or smaller are contained in a contractible neighborhood
  and hence evenly covered by~$\pi$, and so that $\pi$ is an
  isometry on these balls.
  Fix $m$ and~$\epsilon < \epsilon_0$ and focus attention on $U\coloneqq U_{m,\epsilon}$.

  Let $K$ be the constant for approximate path-lifting given by
  Proposition~\ref{prop:apl} for the map $\phi \co X_1 \to X_0$, and
  choose $n\ge m$ large enough that
  $\lambda^{n-m}K/(1-\lambda) < \epsilon/3$. Let $V\subset U$ be
  $U_{n,\epsilon/3}$, and fix
  $q^\infty \in V$. Using Theorem~\ref{thm:Ashadowing}, we are going
  to show $\cJ$ is weakly locally path-connected by
  constructing a path $x^\infty\co I \to \cJ$ from $p^\infty$ to
  $q^\infty$ which is contained in~$U$.

  To construct $x^\infty$, we will first construct a family
  $x = (x_i\co I \to X_1)_{i \geq 1}$ of
  homotopy pseudo-orbits of paths
  joining $p^\infty$ to $q^\infty$
  parameterized by $A=I$, the unit interval; as a path, each $x_i$
  joins $p_i$ to~$q_i$. We first
  construct the $x_i$ for $1 \le i \le n$, where $n$ is the integer
  from the previous paragraph; in this range $x_i$ will be an actual
  orbit (i.e., for $i < n$ the homotopies $\alpha_i$ are constant).
  We begin by choosing the path~$x_{n}$.
  By construction,
  $d_1(p_n,q_n)<\epsilon/3$; let $x_n\co A \to X_1$ be a path
  exhibiting this.
  By decreasing induction, for $1 \le i < n$ define $x_i$ to be the
  lift of $\phi \circ x_{i+1}$ under~$\pi$ starting at $p_i$. Since $\phi$ is a
  contraction and $\pi$ is an isometry on balls of radius $\epsilon/3$,
  the image of $x_i$ is contained in the $\epsilon/3$-ball
  about~$p_i$.
  In particular, $q_i$ is the only element of
  $\pi^{-1}(\phi(q_{i+1}))$ in this ball, so $x_i$ ends at $q_i$.

  We complete the construction of $x$ by
  constructing $x_i$ for $i > n$, 
  by increasing induction
  starting with~$x_n$.
  For $i\geq n$, supposing we have defined $x_i$, let $x_{i+1}$ be an
  approximate lift
  under $\phi$ of $\pi \circ x_i$ joining $p_{i+1}$ to~$q_{i+1}$.
  By Proposition~\ref{prop:apl}, $\pi \circ x_i$ and $\phi \circ
  x_{i+1}$ are joined by a homotopy $\alpha_i$
  with trace size bounded by~$K$. This completes the construction
  of~$x$.

By Theorem \ref{thm:Ashadowing}, the family of homotopy pseudo-orbits
$(x,\alpha)=(x_i, \alpha_i)_{i\geq 1}$ defines
\begin{enumerate}[label=(\roman*)]
\item a family of orbits parameterized by $A$,
$x=(x_i^\infty\co A \to X_1)_{i \geq 1}$, yielding a map
$x^\infty\co A \to \cJ$, and
\item a sequence of homotopies
$\beta=(\beta_i^\infty\co I \times A \to X_1)_{i \geq 1}$ with
$\beta_i^\infty$ joining $x_i$ to $x_i^\infty$.
\end{enumerate}
The bounds on the trace size of the homotopies $\beta_i^\infty$
from $x_i$ to $x_i^\infty$ given in
Theorem~\ref{thm:Ashadowing} are not enough for
our purposes: to make sure the path $x^\infty$ remains within $U$, we
need to make sure that $\beta_i^\infty$ has trace size less than
$\epsilon$ for $1 \le i \le m$, while Theorem~\ref{thm:Ashadowing}
gives a constant trace size $K/(1-\lambda)$.
Thus we consider the proof of the Theorem,
which expresses each element
$\beta_i^\infty$  as a concatenation
$\beta_i^1 * \beta_i^2 * \cdots * \beta_i^k * \cdots$, where each
$\beta_i^k$ is obtained from $\alpha_{i+k}$ by repeatedly lifting by
$\pi$ (a total of $k+1$ times) and composing with~$\phi$ (a total of
$k$~times). Thus $\beta_i^k$ is constant (trace size~$0$) for
$i+k < n$, and otherwise
has trace size bounded by $\lambda^k K$. In particular, for $i \le m$,
we have
\[
  \tracesize(\beta_i^\infty) \le \sum_{k=n-i}^\infty \tracesize(\beta_i^k)
    \le \frac{\lambda^{n-i}}{1-\lambda} K
    \le \frac{\lambda^{n-m}}{1-\lambda} K < \epsilon/3.
\]

Finally, we estimate the size of $x_i^\infty$ for $1 \leq i \leq m$ to
see it remains within~$U$.
We have 
\[ \diam(x_i^\infty) \leq \tracesize(\beta^\infty_i) + \diam(x_i)+\tracesize(\beta^\infty_i)< \epsilon/3 + \epsilon/3 + \epsilon/3 = \epsilon.\]
This concludes the proof of weak local path connectivity, and hence
local connectivity.

It remains to prove that $f\colon \cJ \to \cJ$ is positively
expansive.
Let $\epsilon_0$ be the parameter chosen above, so that $d_0$-balls of
radius $\epsilon_0$ are contained in contractible sets.
Consider the neighborhood of the diagonal
$N\coloneqq U_{1,\epsilon_0} = \{(x,y) \in \cJ \times \cJ \mid
d_1(x_1, y_1)<\epsilon_0\}$.
To prove that $f$ is positively expansive, it suffices to show that,
if $(f^ix, f^iy) \in N$ for each $i \geq 0$, then $x = y$; so let us
suppose the iterates remain in~$N$. If we think
of $\cJ$ as a subset of
$X_1^{\mathbb{N}}$, the
map $f$ is given by the left-shift. Thus $d_1(x_i, y_i)<\epsilon_0$
for each $i \geq 0$.

Since $d_1$ is a length metric,  for each such $i$ there
exists a path $\beta_i\co [0,1] \to X_1$ with
$\ell(\beta_i) < \epsilon_0$ joining $x_i$ to $y_i$. By construction,
the paths $\pi(\beta_i)$ and $\phi(\beta_{i+1})$ join the same
endpoints for each $i \geq 1$. We have
$\diam(\pi(\beta_i))<\epsilon_0$ by construction,
while $\diam(\phi(\beta_{i+1})) < \lambda\epsilon_0$;
hence the union of
these paths lies in an $\epsilon_0$-ball in~$X_0$. Since the ball is
contractible, the two paths
are homotopic. The collection $\beta\coloneqq (\beta_i)_{i \geq 1}$ is therefore a
homotopy between the orbits $x=(x_i)_i$ and $y=(y_i)_i$. By
Theorem~\ref{thm:homotopy_shadowing}, we then have $x=y$, completing
the proof
that $f$ is positively expansive.\qed

\subsection{Homotopy sections}
\label{subsecn:sections}

Suppose $\phi\co X \to Y$ is a continuous map between topological spaces. A \emph{homotopy section} of $\phi$ is a continuous map $\sigma\co Y \to X$ such that $\phi \circ \sigma \sim_K \id_Y$ for some constant $K$. 
The main result of this section is the existence of homotopy sections to~$\cJ$.

\begin{proposition}
\label{prop:section}
Suppose $\phi, \pi\co X_1 \rightrightarrows X_0$ is a $\lambda$-backward-contracting and
recurrent virtual endomorphism of finite CW complexes. Suppose
$\sigma\co X_0 \to X_1$ is a homotopy section of $\phi$, with
$\phi \circ \sigma \sim_K \id_{X_0}$. Then for each $n \in
\mathbb{N}$, we have the following.
\begin{enumerate}
\item There is a canonically associated family
\[ (x^n,\alpha^n)=(x^n_i\co X_n \to X_1, \alpha^n_i\co I \times X_n \to
  X_0)_{i \geq 1}\]
of homotopy pseudo-orbits parameterized by $X_n$
with trace sizes at most~$K$.  When $X_n$ is identified with the set of orbits of length $n$, the map $X_n \to X_1 \times \ldots \times X_1$ determined by the first $n$ coordinates, $(x_1^n, x_2^n, \ldots, x_n^n)$ induces the identity map on $X_n$.   Concretely, we require that
$x_i = \pi^i_1 \circ \phi^n_i$ for $i \le n$ and $\alpha_i$ is
constant for $i < n$.
\item There exists $\tilde{x}^n_\infty\co X_n \to \cJ$ with
  $\phi^\infty_n\circ \tilde{x}^n_\infty \sim_{K'}\id_{X_n}$, where $K'\coloneqq K/(1-\lambda)$.
\end{enumerate}
\end{proposition}
Observe that in the lifted length metrics, the diameters of the
$X_n$ typically tend to infinity exponentially fast.
Conclusion~(2), however, says that the compositions $\phi^\infty_n \circ
\tilde{x}^n_\infty$ are at
uniformly bounded distance from the identity.

\begin{proof}
  We will construct a family of homotopy pseudo-orbits
  $x^n=(x_i\co X_n \to X_1)_{i \geq 1}$. Fix a homotopy
  $\alpha\co I \times X_0 \to X_0$ joining the identity to
  $\phi \circ \sigma$ with trace size~$K$.

 For $1 \leq i \leq n$, let $x^n_i\co X_n \to X_1$ be the natural
 maps as defined in the statement; likewise, for $1 \le i < n$, let
 $\alpha^n_i\co I \times X_n \to X_0$ be the
 constant homotopies.

 For $i \ge  n$, set by induction $x_{i+1} \coloneqq \sigma \circ \pi
 \circ x_i$. Then
\[ \phi \circ x_{i+1}=\phi \circ \sigma \circ \pi \circ x_i \sim_K
  \pi \circ x_i\]
with homotopy $\alpha_i = (\pi \circ x_i)^* (\alpha)$
(see
Lemma~\ref{lem:comphom}). This gives the desired family of homotopy
pseudo-orbits. 

The second assertion follows immediately from Theorem
\ref{thm:Ashadowing-cover}, applied to the family $(x,\alpha)$ of
homotopy pseudo-orbits parameterized by $X_n$ constructed above.
\end{proof}

To get started applying this result, we have the following lemma.
\begin{lemma}
\label{lemma:sections}
Suppose $G, H$ are finite connected graphs and $\phi\co G \to H$ is
surjective on the fundamental group. 
Then there exists a homotopy section $\sigma\co H \to G$ of $\phi$. 
\end{lemma}

\begin{proof}
  The statement is invariant under homotopy
  equivalence, so we may assume
  $H$ is a rose of $k$ circles
  with basepoint~$y$.  Let $x \in \phi^{-1}(y)$ be a basepoint for
  $G$.  Fix one of the $k$
  circles, say $\beta \subset H$. The assumption that $\phi$
  induces a  surjection between fundamental groups implies there is a
  loop $\alpha$ based at $x$ for which $\phi(\alpha) \sim_{K(\beta)} \beta$
  relative to $y$.  We set $\sigma|_\beta=\alpha$. Doing this for each
  circle and putting $K\coloneqq \max_\beta K(\beta)$ proves the claim. 
\end{proof}

As a corollary of Proposition~\ref{prop:section}, we then have 

\begin{corollary}
\label{cor:recurrrent_graph_ve}
If $\pi, \phi\co G_1 \rightrightarrows G_0$ is a backward-contracting
recurrent virtual endomorphism of graphs with limit space $\cJ$, then
there exists a constant $K>0$ and a family of maps $\sigma^n_\infty\co
G_n \to \cJ$ for $n \in \mathbb{N}$  such that $\phi^\infty_n \circ
\sigma^n_\infty \sim_K \mathrm{id}_{G_n}$.
\end{corollary}


\section{The conformal gauge}
\label{sec:gauge}

We recall here from \cite{kmp:ph:cxci} the construction of two natural classes of metrics, one larger than the other, associated to certain expanding dynamical systems.

\subsection{Convention}
\label{subsecn:convention}
Throughout this section, we suppose $\cJ$ is compact, connected, and
locally connected, and $f\co \cJ \to \cJ$ is a positively expansive
self-cover of degree $d \geq 2$. Proposition~\ref{prop:shrink} and
Lemma~\ref{lemma:expansion_enough} imply that the dynamics of
$f\co \cJ \to \cJ$ is topologically cxc. It follows that there exists a
finite open cover $\cU_0$ by connected sets such that, as in the notation from \S
\ref{subsecn:cxc}, the mesh of the coverings $\cU_n$ tends to zero as
$n \to \infty$, and in addition, for all $\wtU \in \cU_n$ with
$f^n\co \wtU \to U \in \cU_0$, we have $\deg(f^n\co \wtU \to U)=1$, i.e.,
each such $U$ is evenly covered by each iterate.

\subsection{Visual metrics} 
\label{subsecn:visual}
The metrics we construct are most conveniently defined coarse-geometrically as visual metrics on the boundary of a certain rooted Gromov hyperbolic 1-complex. Before launching into technicalities, we quickly summarize the development.
The \emph{visual metrics}
on~$\cJ$ have the properties that there exists a
constant $0<\theta<1$ such that for any $n \in \mathbb{N}$ and any
$U \in \cU_n$, we have $\diam U \asymp \theta^n$, and these $U$ are
uniformly nearly round. See Theorem \ref{thm:visual} below for the
precise statements. The \emph{snowflake gauge} is the set of
metrics bilipschitz equivalent to some power of a visual metric. The
snowflake gauge is an invariant of the topological dynamics. Visual
metrics are Ahlfors regular with respect to the Hausdorff measure in
their Hausdorff dimension, and this measure is comparable to the
measure of maximal entropy. The \emph{Ahlfors-regular gauge} of
metrics is the larger set of all Ahlfors-regular spaces
quasi-symmetrically equivalent to a visual metric; it too is an
invariant of the topological dynamics.

\subsubsection*{\textalt{$1$}{1}-complex} From coverings as above, we define a
rooted hyperbolic 1-complex~$\Sigma$ to get the visual metrics. In
addition to the $\cU_n$ for $n \ge 0$, let
$\cU_{-1}$ be the trivial covering $\{\cJ\}$. The vertices of~$\Sigma$ are the elements of the
$\cU_n$ for $n \ge -1$, with root $\cJ \in \cU_{-1}$. If $U \in \cU_n$,
the \emph{level} of~$U$ is $|U|\coloneqq n$. Edges of~$\Sigma$ are of two types.
\begin{itemize}
\item \emph{Horizontal} edges join $U, U' \in \cU_n$ if
  $U \cap U' \neq \emptyset$.
\item \emph{Vertical} edges join $U, U'$ if $\bigl| |U| - |U'| \bigr|=1$ and
  $U \cap U' \neq \emptyset$.
\end{itemize}
We equip $\Sigma$ with the
word length metric in which each edge has length~$1$. The \emph{level} of
an edge is the maximum of the levels of its endpoints.

\subsubsection*{Compactification} We compactify $\Sigma$ in the spirit of W. Floyd rather than of M.
Gromov, as follows. Let $\epsilon>0$ be a parameter, and let
$d_\epsilon$ be the length metric on $\Sigma$ obtained by scaling so
edges at level $n$ have length $\theta^n$ where
$\theta=e^{-\epsilon}$. The metric space $(\Sigma, d_\epsilon)$ is not complete. Its completion $\overline{\Sigma}_\epsilon$, sometimes called the \emph{Floyd completion}, adjoins the corresponding boundary $\partial_\epsilon\Sigma\coloneqq \overline\Sigma_\epsilon - \Sigma$. 
The extension of $d_\epsilon$ to a
metric on the boundary $\partial_\epsilon\Sigma$ is called a
\emph{visual metric} on the boundary.

\subsubsection*{Snapshots} To set up the statement of the next theorem, we need a definition.

\begin{definition}
\label{defn:snapshots} Suppose $X$ is a metric space and $0<\theta<1$. A sequence $(\cS_n)_n$ of finite coverings of $X$ is called a \emph{sequence of snapshots of $X$ with scale parameter $\theta$} if there exists a constant $C>1$ such that 
\begin{enumerate}
\item (scale and roundness) For all $n$ and all $s \in \cS_n$, there exists $x_s \in s$ with 
  \[ B(x_s, C^{-1}\theta^n) \subset s \subset B(x_s, C\theta^n),\]
\item (nearly disjoint) For all $n$, the collection of pairs
    $\{(x_s,s)\mid s \in \cS_n\}$ may be chosen so that in addition 
  \[ B(x_s, C^{-1}\theta^n) \cap B(x_{s'}, C^{-1}\theta^n)\neq \emptyset \]
  whenever $s$ and $s'$ are distinct elements of~$\cS_n$.
  \end{enumerate}
\end{definition}
The elements $s \in \cS_n$ need not be open, nor connected. 

\subsubsection*{Properties of visual metrics} Theorem \ref{thm:visual} summarizes highlights of \cite[Chapter 3]{kmp:ph:cxci}, specialized to the conventions in \S \ref{subsecn:convention}. 

\begin{citethm}[Visual metrics]
\label{thm:visual}
Suppose the topological dynamical system $f\co \cJ \to \cJ$ and sequence of open covers $(\cU_n)_n$ satisfy the assumptions in \S \ref{subsecn:convention}. Let $\Sigma$ be the associated 1-complex.

There exists $\epsilon_0>0$ such that for all $0<\epsilon<\epsilon_0$,
the boundary $\partial_\epsilon\Sigma$ equipped with the visual metric
$d_\epsilon$ is naturally homeomorphic to $\cJ$. Moreover, for each
$\epsilon$ in this range there
exists $C>1$ such that the following hold. Let
$\theta=e^{-\epsilon}<1$; we denote by $B_\epsilon$ an open ball with
respect to~$d_\epsilon$.
\begin{enumerate}
\item (Snapshot property) The sequence $(\cU_n)_n$ is a sequence of snapshots of $\cJ$ with scale parameter $\theta$.
 \item\label{item:visual-regular} For $q\coloneqq\frac{\log d}{\epsilon}$, for all $x \in \cJ$ and all $r < \diam(\cJ)$, the Hausdorff $q$-dimensional measure satisfies 
  \[ C^{-1}r^q < \cH^q(B_\epsilon(x,r))< Cr^q.\]
\item There exists a unique $f$-invariant probability measure $\mu_f$
  of maximal entropy $\log d$ supported on $\cJ$. The support of
  $\mu_f$ is equal to all of~$\cJ$, and for all Borel sets $E$, we
  have $C^{-1} < \mu_f(E)/\cH^q(E) < C$.
\item\label{item:scaling} There exists $r_0<\diam(\cJ)$ so that for
  all $r<r_0$, we have
  $f(B_\epsilon(x,r))=B_\epsilon(f(x),\theta^{-1}r)$, and the
  restriction $f\co B_\epsilon(x,r) \to B_\epsilon(f(x),\theta^{-1}r)$
  scales distances by the factor~$\theta^{-1}$.
\item\label{item:visual-indep} If two different parameters
  $\epsilon$, $\epsilon'$ and two different open covers $\cU_0$, $\cU'_0$
  are employed in the construction, the resulting metrics $d$, $d'$ are
  snowflake equivalent. 
  \end{enumerate}
\end{citethm}

\subsubsection*{Snowflake equivalence} Two metrics $d, d'$ are \emph{snowflake equivalent} if there exist parameters
$\beta, \beta'>0$ such that the ratio $d^\beta/(d')^{\beta'}$ is
bounded away from zero and infinity.  Put another way, snowflake equivalent means bilipschitz, after raising the metrics to appropriate powers. The \emph{snowflake gauge} of
$(f\co \cJ \to \cJ)$ is the snowflake equivalence class of
some (and by Theorem \ref{thm:visual}(5), equivalently, of any) visual
metric on $\cJ$. Thus, the snowflake gauge of the dynamical system
$f\co \cJ \to \cJ$ depends only on the topological dynamics, and not on
choices.  

A snowflake equivalence $d \mapsto d'$ preserves the property of being
a sequence of snapshots, though the scale parameter and constant $C$
may change.  However, if $(\cS_n)_n$ is a sequence of snapshots in a
metric space $X$, if $Y$ is another metric space, and
if $h\co X \to Y$
is only a quasi-symmetry, then
the transported sequence $(\cT_n)_n$
comprised of images of elements of $\cS_n$ under $h$ need not be a
sequence of snapshots: the scale condition in the definition of sequence
of snapshots (Definition \ref{defn:snapshots}(1)) can fail. (See
\S\ref{sec:ahlfors-regular} for a definition of quasi-symmetry.)

\subsubsection*{Conformal elevator and naturality of gauges}
Suppose now $f\co \widehat{\mathbb{C}} \to \widehat{\mathbb{C}}$ is a
hyperbolic rational function. Its Julia set $\cJ$ can now be equipped
with at least two natural metrics: the round spherical metric, and
a visual metric. The
Koebe distortion principles imply that small
spherical balls can be blown up via iterates of $f$ to balls of
definite size, with uniformly bounded distortion. The same is true for
the visual metric, by Theorem \ref{thm:visual}(4). Combining these
observations in a technique
known as the \emph{conformal elevator} implies the following.

\begin{proposition}
\label{prop:elevator}
On Julia sets of hyperbolic rational functions, the spherical and
visual metrics are quasi-symmetrically equivalent.
\end{proposition}

This is true much more generally
\cite[Theorems 2.8.2 and 4.2.4]{kmp:ph:cxci}. In all but the most restricted cases, however, the
spherical and visual metrics are not snowflake equivalent. In the
visual metric, the map $f$ is locally a homothety with a constant
factor which is the same at all points. In the spherical metric, the image of a ball under $f$ need not be a ball, and what's more, typically, the magnitude of the derivative $\abs{f'(z)}$ varies as $z$ varies in $J_f$. The exceptions include maps such as $f(z)=z^d$.

\subsection{Ahlfors-regular spaces}\label{sec:ahlfors-regular}

We collect several facts here; see \cite{heinonen:analysis}.  
A metric space is \emph{doubling} if there is an integer $N$ such that
any ball of radius~$r$ is covered by at most $N$ balls of radius
$r/2$. Doubling is a finite-dimensionality condition: the Assouad
Embedding Theorem asserts that a doubling metric space is snowflake
equivalent to a subset of a finite dimensional spherical space.

Ahlfors regularity is a homogeneity condition that implies doubling.
A space with both a metric and a measure $(Z, d, \mu)$ is
\emph{Ahlfors regular} of exponent $q$ if $\mu(B(z,r)) \asymp r^q$ for
each $r<\diam(Z)$, where the implicit constant is independent of $z$
and of~$r$.  The Hausdorff dimension of such a space is necessarily
equal to~$q$, and in fact
the given measure~$\mu$ is comparable to the
$q$-dimensional Hausdorff measure.
A metric space is \emph{Ahlfors regular} if its Hausdorff measure in
its Hausdorff dimension is Ahlfors regular.

Suppose $(Z, d, \mu)$ and $(Z', d', \mu')$ are connected compact
doubling metric spaces. A homeomorphism $h\co Z \to Z'$ is a
\emph{quasi-symmetry} (qs) if there is a constant $H \geq 1$ such that
for each $r < \diam(Z)$ and each $z \in Z$, there is an $s>0$ such
that
\[ B_{d'}(h(z), s) \subset h(B_d(z,r)) \subset B_{d'}(h(z), Hs).\]
In other words: the image of a round ball is nearly round.\footnote{Strictly speaking, this is the definition of a \emph{weakly qs} map; this is equivalent to the standard but less intuitive definition of a qs map in our setting; see \cite{heinonen:analysis}.} 
A quasi-symmetric map does not, in general, preserve the property of
Ahlfors regularity, though it does preserve the property of being
doubling.

The \emph{Ahlfors-regular conformal gauge} $\cG(Z)$ of a metric space
$Z$ is the set of all Ahlfors-regular metric spaces qs equivalent to
$Z$; it may be empty. Let $f\co \cJ \to \cJ$ be a positively expansive
self-cover as in the conventions of \S \ref{subsecn:convention} and let $d_\epsilon$ be a visual metric from
Theorem~\ref{thm:visual}. Conclusion \eqref{item:visual-regular} of
that theorem shows that visual metrics are Ahlfors regular and so
$\cG(f\co \cJ \to \cJ)\coloneqq\cG(\cJ, d_\epsilon)$ is nonempty. For
example, if $f$ is a hyperbolic rational function with Julia set
$\cJ$, the spherical metric $d_{sph}$ on $\cJ$ is Ahlfors regular, with exponent $q=\hdim(\cJ, d_{sph})$; see \cite[Corollary 9.1.7]{MR2656475}.
Proposition~\ref{prop:elevator} then
implies that the spherical and visual metrics on $\cJ$ both belong to
the gauge $\cG(f\co \cJ \to \cJ)$.

The \emph{Ahlfors-regular conformal dimension} $\ARCdim(f\co\cJ \to \cJ)$ is the infimum of the Hausdorff dimensions of the metrics in the Ahlfors regular conformal gauge. It is thus another numerical invariant of the topological dynamics.  
Note that by definition, for any Ahlfors-regular metric $d$ in $\cG(f\co\cJ \to \cJ)$, we have 
\[ \ARCdim(f\co\cJ \to \cJ) \leq  \hdim(\cJ, d).\]

\subsection{Approximately self-similar spaces} Our proof of Theorem
\ref{thm:crit-sandwich} uses a result of Carrasco and
Keith-Kleiner that says that for certain classes of spaces, the
Ahlfors-regular conformal dimension is a critical exponent of
combinatorial modulus.
(See Theorem~\ref{thm:keith-kleiner}.)
Here, we introduce that class of spaces. 

 The following definition appears in \cite[\S3]{Kleiner06:ICM}.
\begin{definition}[Approximately self-similar]
\label{defn:approx_self_similar}
A compact metric space $(Z,d)$ is \emph{approximately self-similar} if
there exists a constant $L\geq 1$ such that for every ball $B(z,r) \subset Z$ with radius 
$0 < r < \diam(Z)$, there exists an open set $U \subset Z$ which
is $L$-bilipschitz equivalent to the rescaled ball
$(B(z,r), \frac{1}{r}d)$.
\end{definition}

An immediate consequence of Theorem \ref{thm:visual} is the following.

\begin{corollary}[Visual metrics are self-similar]
\label{cor:ss}
In the setting of Theorem \ref{thm:visual}, visual metrics are approximately self-similar.
\end{corollary}
\begin{proof} Let $\epsilon$, $r_0$, and $\theta$ be as in Theorem
  \ref{thm:visual}.  Set $L:=\diam(Z)/(\theta r_0)$. Choose any $0<r<\diam(\cJ)$. Let
  $B_\epsilon(x,r)$ be any ball. If $r>\theta r_0$ we take $U\coloneqq
  B_\epsilon(x,r)$ and note that the metric space $(B_\epsilon(x,r),
  d_\epsilon)$ is $L$-bilipschitz equivalent to its rescaling 
  $(B_\epsilon(x,r), \frac{1}{r} d_\epsilon)$. If $r<\theta r_0$, let
  $n$ be the unique integer for which $\theta r_0 < \theta^{-n}r \leq
  r_0$.  Set $U\coloneqq B_\epsilon(f^nx, \theta^{-n}r)$. By Theorem \ref{thm:visual}(\ref{item:scaling}), the open set $U$ is isometric to the rescaled metric space $(B_\epsilon(x,r), \theta^{-n}d_\epsilon)=(B_\epsilon(x,r), \frac{1}{r}\cdot r\theta^{-n}d_\epsilon)$ which is in turn $L$-bilipschitz equivalent to $(B_\epsilon(x,r), \frac{1}{r}d_\epsilon)$ by our choice of $n$.
\end{proof}


\section{Energies of graph maps}
\label{sec:energies}
In this section, we introduce asymptotic $q$-conformal energies associated to
virtual endomorphisms of graphs and related analytical notions of
extremal length.  This section is a review of concepts from
\cite{Thurston19:Elastic}, which contains further motivation,
especially in its Appendix~A.

In this section, all graphs are assumed to be finite.

\subsection{Weighted and \textalt{$q$}{q}-conformal graphs}
\label{subsecn:pcg} 
\begin{definition}
\label{defn:pcg}
Suppose $G$ is a graph, and $1 \leq q \leq \infty$. 
\begin{enumerate}
\item For $1 < q \leq \infty $, a \emph{$q$-conformal structure} on
  $G$
  is a positive
  \emph{$q$-length} $\alpha(e)$ on each edge~$e$, giving a length
  metric $|dx|$ in which~$e$
  has length $\alpha(e)$. For $q=\infty$, this is also called a
  \emph{length graph}.
\item
  For $q=1$, a \emph{$1$-conformal structure} on $G$
  is a positive weight $w(e)$ on each edge~$e$.
  These weights do not determine a length structure. (For instance, in
  cases where
  we take derivatives for maps from a $1$-conformal graph, the length
  structure is arbitrary.)
  A $1$-conformal graph is also called a \emph{weighted graph}.
\end{enumerate}
We will write $G^q$ for a graph together with a choice of $q$-conformal
structure on it, and denote its underlying graph by $G$. 
\end{definition}

\begin{remark}
  Although $q$-conformal structures for $q \in (1,\infty]$ are all
  formally the same data, we distinguish the value of~$q$ for two
  reasons. First, it helps us keep track of which energies $E^p_q$ to
  consider. Secondly, from another point of view a $q$-conformal
  structure is naturally thought of as an equivalence class of pairs
  $(\ell,w)$ of a length structure~$\ell$ and weights~$w$ under a
  rescaling depending on~$q$ \cite[Def.\ A.17]{Thurston19:Elastic},
  and from that point of view there are two natural length-like
  structures in the picture.

The distinctions between $q$-conformal graphs as $q$ varies arise from
how various related
analytical quantities are defined and scale under changes of weights.
Imagine each edge as ``thickened'' with an extra $(q-1)$-dimensional
space, to obtain a ``rectangle''
equipped with $q$-dimensional Hausdorff measure $\mathcal{H}^q$.
Scaling the length of an edge by a factor $\lambda$ changes the total
$\mathcal{H}^q$-measure by the factor $\lambda^q$ and therefore scales
in the imaginary direction ``orthogonal'' to the edge by the factor
$\lambda^{q-1}$.
\end{remark}

An important special case of weighted graphs (i.e., $q=1$) is when the
underlying
space is a $1$-manifold $\bigcup_i J_i$, where each $J_i$ is
an interval or the circle. We may regard each $J_i$ as a graph, by
adding the endpoints of the interval, or an arbitrary vertex on the
circle. Up to isomorphism, the result is unique. A formal sum
$C = \sum w_i J_i$ with $w_i > 0$ then determines a weighted graph in a
unique way, by setting the weight of the unique edge in $J_i$ to
be~$w_i$. Ignoring the graph structure, we also call $C$ a weighted $1$-manifold.

\begin{definition}\label{def:multicurve}
  A \emph{curve} on a space~$X$ is a connected 1-manifold $C$ (either $I$ or $S^1$) together with a map $\gamma \co C \to X$. We refer
  to the curve by just the map~$\gamma$ to be short, but the underlying
  domain~$C$ is part of the data. (Note the distinction with a
  \emph{path},
  where, for us, the domain is a fixed interval of real numbers.) For a
  \emph{multi-curve} we drop the restriction that $C$ be connected,
  giving $C = \bigcup_i J_i$ where the $J_i \in \{I, S^1\}$ are the connected
  components and the map $\gamma$ is determined by
  $\gamma_i\co J_i \to X$. A \emph{strand} of a multi-curve is one of its
  component curves~$\gamma_i$. A \emph{weighted multi-curve} $\gamma\colon C \to X$ is similar, but
  where $C=\sum w_i J_i$ has the structure of a weighted graph; we also denote this by $\gamma = \sum w_i \gamma_i$. The homotopy class of a curve or multi-curve $\gamma$ is denoted $[\gamma]$.
\end{definition}

\subsection{Energies and extremal length} 
\label{subsec:el} 
For any
$1 \le p \le q \le \infty$ and piecewise linear (PL) map $\phi \co G^p \to H^q$ from a
$p$-conformal graph to a $q$-conformal graph, there is an energy
$E^p_q(\phi)$ with well-behaved properties. For any of these energies
and a homotopy class $[\phi]$ of maps as above, we denote by
$E^p_q[\phi]$ the infimum of the energy over maps in the class. In
this paper we only need a few special cases.

\subsubsection*{Case \textalt{$p=1$}{p=1} and \textalt{$1<
    q<\infty$}{1<q<infty}}
We switch names to consider a PL map
$\phi\co W^1 \to G^q$ from a weighted (or $1$-conformal) graph $W^1$ to a
$q$-conformal graph $G^q$, where $1 < q < \infty$. (In this paper
$W^1$ will typically be a 1-manifold, so this is the structure of a
weighted curve on~$G^q$.) Let
$\qdual=q/(q-1)$ be
the H\"older conjugate of~$q$.

Define
\[ E^1_q(\phi)\coloneqq \norm{n_\phi}_{\qdual, G^q}.\]
We now explain the notation.  For $y \in G^q$, the value 
\[ n_\phi(y)\coloneqq  \sum_{\phi(x)=y}w(e(x))\]
is the weighted number of pre-images
of~$y$, where for $x \in W^1$, the quantity $e(x)$ denotes an edge
containing $x$. The edge-weights $\alpha$ on $G^q$, when interpreted
as lengths, define a length measure $|dy|$ on the underlying set of
$G^q$ such that $\int_e |dy| = \alpha(e)$ for each edge $e$ of $G^q$.
The quantity $n_\phi(y)$ may be infinite (if, e.g., $\phi$ collapses an
edge to the point~$y$) or undefined (if, e.g., $x \in \phi^{-1}(y)$ is a
vertex incident to edges with different weights). However, the
assumption that $\phi$ is PL and our convention that the graphs are finite imply that the set of such $y$ is
finite, hence of $|dy|$-measure zero. The quantity $\norm{n_\phi}_{\qdual, G^q}$
is then the usual $L^{\qdual}$ norm of the function $n_\phi$ with respect to
the measure $|dy|$.
For instance,
$E^1_\infty(\phi)$ is the weighted total length of the image
of~$\phi$.

If $n_\phi$ is constant on edges of~$G^q$---automatic if $\phi$
minimizes energy in its homotopy class---the formula for the energy is
concretely given by
\begin{equation}\label{eq:E1q}
  E^1_q(\phi) = \biggl(\sum_{e \in \Edges(G)} \alpha(e)
  n_\phi(e)^{\qdual}\biggr)^{1/\qdual}.
\end{equation}

\begin{definition}
\label{defn:reduced}
A weighted multi-curve $\gamma\co C \to G$ on a graph $G$ is \emph{reduced} if the restriction to each strand is either constant, or has arbitrarily small perturbations that are locally injective.
\end{definition}
Thus the images of the strands of a reduced multi-curve have no
backtracking.
Reduced curves $\phi$ minimize 
$E^1_q$ in their homotopy class \cite[Proposition
3.8 and Lemma 3.10]{Thurston19:Elastic}.

We also define the \emph{$q$-extremal-length} of a homotopy class of
maps $\phi\co W^1 \to G^q$
by
\begin{equation}\label{eq:ELq}
  \EL_q[\phi] \coloneqq \bigl(E^1_q[\phi]\bigr)^q.
\end{equation}
To justify the terminology, note that as explained in
\cite[\S5.2]{Thurston16:RubberBands}, the
minimizer of extremal length over a suitable homotopy class of maps
exists and is realized by a map with nice properties, and the minimum
value, when formulated as an extremal problem, mimics the usual
definition of extremal length. See also \S\ref{sec:comb-mod-graph} below.

\subsubsection*{Case \textalt{$1 < p < \infty$}{1<p<infty} and \textalt{$q=\infty$}{q=infty}}
The \emph{$p$-harmonic energy} of a PL map $\phi \co
G^p \to K^\infty$ from a $p$-conformal graph to a length graph is
\[
  E^p_\infty(\phi) \coloneqq \norm{\abs{\phi'}}_{p,G^p}.
\]
Here $|\phi'|$ denotes the size of the derivative:
since $p>1$, the conformal graph $G^p$ has a length structure, so we
can differentiate. If the
derivative of~$\phi$ is constant on the edges of~$G^p$---automatic if
$\phi$ minimizes energy in its
homotopy class---this is
\[
  E^p_\infty(\phi) = \biggl(\sum_{e\in\Edge(G)}
    \alpha(e)^{1-p}\ell(\phi(e))^p\biggr)^{1/p},
\]
where $\ell(\phi(e))$ is the total length of the image of~$e$.

\subsubsection*{Case \textalt{$p = 1$}{p=1} and \textalt{$q=\infty$}{q=infty}} This
is the common limit of the above two cases, slightly modified since
$1$-conformal graphs have a weight~$w$ instead of a
$p$-length~$\alpha$. For a PL map $\phi \co W^1 \to K^\infty$ from a weighted
graph to a length graph, set
\[
  E^1_\infty(\phi) \coloneqq \int_{x \in W} w(x) \abs{\phi'(x)}\,dx
    = \int_{y \in K} n_\phi(y)\,dy.
\]
This is the weighted length of the image of~$\phi$.

\subsubsection*{Case \textalt{$1<p=q<\infty$}{1<p=q<infty}}  A PL map $\phi\co G^q \to H^q$ between $q$-conformal
graphs has \emph{filling function}
$\Fill^q(\phi)\co H^q \to \mathbb{R}$ given at generic points by
\begin{align}
   \Fill^q(\phi)(y)&\coloneqq\sum_{\phi(x)=y}\lvert \phi'(x)\rvert^{q-1}
   \label{eq:Fillq}\\
\intertext{and a $q$-conformal \emph{energy} given by}
  E^q_q(\phi)&\coloneqq\bigl(\norm{\Fill^q(\phi)}_{\infty, H}\bigr)^{1/q}.
\end{align}
If we interpret edges of conformal graphs as ``thickened to
rectangles'' as described above, the filling function sums the
``thicknesses'' of the ``rectangles'' over a fiber.

\subsubsection*{Case \textalt{$p=q=1$}{p=q=1}} A PL map $\phi \co G^1
\to H^1$ has an energy $E^1_1$ that is again a limit
of the above case, modified to account for weights rather than
$q$-lengths:
\[
  N(\phi) = E^1_1(\phi) \coloneqq \esssup_{y \in H} \frac{n_\phi(y)}{w(y)}.
\]
We will apply this in cases where the weights are all~$1$, in which
case this amounts to counting the essential maximum of the number of
preimages of any point. (``Essential'' as usual means that we can
ignore sets of linear measure zero, and in particular we can ignore images of edges of~$G^1$ that map to a
single point in~$H^1$.) We will also call this quantity~$N(\phi)$.

\subsubsection*{Properties of energies}  

It is not too hard to see that these energies above are all \emph{sub-multiplicative}: if
$\phi \co G^p \to H^q$ and $\psi\co H^q \to K^r$ are maps from a
$p$-conformal graph to a $q$-conformal graph to an $r$-conformal graph
(where the energies are defined), then
$E^p_r(\psi \circ \phi) \leq E^p_q(\phi)E^q_r(\psi)$ \cite[Prop.~A.12]{Thurston19:Elastic}.

Given $\phi\co G^q \to H^q$,
we denote by $[\phi]$ its homotopy class. It is natural to consider
minimizers of $E^q_q$ over the class $[\phi]$. The fact that
minimizers of these energies exist is not obvious, but more is true.
We first give some general context.

\begin{definition}[{\cite[Definition 1.33]{Thurston19:Elastic}}]
  Given maps $\phi \co G^p \to H^q$ and $\psi \co H^q \to K^r$ between
  graphs $G^p$, $H^q$, and $K^r$ with the respective conformal structures,
  we say that the sequence $\shortseq{G^p}{\phi}{H^q}{\psi}{K^r}$ is \emph{tight} if
  \[
    E^p_r[\psi \circ \phi] = E^p_r(\psi \circ \phi) = E^p_q(\phi) E^q_r(\psi).
  \]
\end{definition}
Together with sub-multiplicativity of energy, the existence of a tight
sequence implies that $\phi$
and $\psi$ both minimize energy in their respective homotopy classes,
and furthermore the sub-multiplicativity inequalities are sharp.
Then \cite[Theorem~6, Appendix~A]{Thurston19:Elastic} asserts that
\emph{any} map $\phi \co G^q \to H^r$ can be
homotoped to be part of a tight sequence (on either side). We will
state and use two
special cases, starting with the easier one.

\begin{proposition}\label{prop:duality}
  Pick $1 \le q \le \infty$ and let $G^q$ be a $q$-conformal graph.
  For any reduced multi-curve $\gamma \colon C^1 \to G^q$ on~$G$, there is a
  length graph~$K^\infty$ and map $\psi \colon G^q \to K^\infty$ so that
  \[
    E^1_q(\gamma) = E^1_q[\gamma] = \frac{E^1_\infty(\psi \circ \gamma)}{E^q_\infty(\psi)}
  \]
  and the maps $\psi$ and $\psi \circ \gamma$ minimize the energies $E^q_\infty$
  and $E^1_\infty$, respectively, in their homotopy classes.
  Furthermore, we can take $K^\infty$ to have the same underlying graph as~$G^q$
but with different edge lengths, and can take $\psi$ to be the
  identity.
\end{proposition}

This is the case $p=1$ of \cite[Theorem~6]{Thurston19:Elastic}; see
also \cite[Proposition~A.10]{Thurston19:Elastic}.
It can be viewed as the duality of $L^q$ and $L^{\qdual}$ norms. The
case $q=2$ appears as the duality map in
\cite[Definition~6.21]{Thurston19:Elastic}. Since the proof of this
case was omitted in the cited paper, we include it here.

\begin{proof}
  Define $K^\infty$ to
  have underlying graph $G$ with the length of
  an edge given by $\ell(e) \coloneqq \alpha(e) n_\gamma(e)^{1/(q-1)}$. Set
  $\psi$ to be the identity as a map from $G$ to itself.
  Then $\psi \circ \gamma \colon C^1 \to K^\infty$ is reduced. We thus have
  \begin{align*}
    E^1_\infty[\psi \circ \gamma]
      &= E^1_\infty(\psi \circ \gamma) \\
      &= \sum\nolimits_e n_\gamma(e) \cdot \ell(e) \\
      &= E^1_q(\gamma) E^q_\infty(\psi),
  \end{align*}
  where last equality is the equality case of the Hölder inequality
  $\norm{n_\gamma \cdot (\ell/\alpha)}_{1,G^q} \le \norm{n_\gamma}_{\qdual,G^q}
  \cdot \norm{\ell/\alpha}_{q,G^q}$, and
  is immediate if you expand the
  definitions. Thus the sequence
  $\shortseq{C^1}{\gamma}{G^q}{\psi}{K^\infty}$ is tight, yielding all the
  conclusions in the statement by
  \cite[Lemma~1.34]{Thurston19:Elastic}.
\end{proof}

For the other version, we
define another quantity, the \emph{stretch factor}.
Define
\[
  \SFto_q[\phi] \coloneqq \sup_{[\gamma] \co C \to G} \frac{E^1_q[\phi \circ \gamma]}{E^1_q[\gamma]}.
\]
where the supremum runs over all non-trivial homotopy classes of
weighted multi-curves $\gamma\co C^1\to G^q$ (which we
may take to be reduced). In other words, by Eq.~\eqref{eq:ELq}, the
$q$-stretch-factor is
a power of the maximum ratio of distortion of $q$-extremal-length of
homotopy classes of maps of weighted multi-curves.
(It is easy to see that the supremum is the same if we take it over
unweighted multi-curves or over unweighted curves, but in either of
these other cases the supremum is
not always realized.)

\begin{citethm}
\label{thm:sf} For $1 \leq q \leq \infty$ and $[\phi]\co 
G^q \to H^q$ a homotopy class of maps between $q$-conformal graphs,
there is a map $\psi \in [\phi]$, a weighted 1-manifold $C^1=\sum_i
w_iJ_i$, and a
map $\gamma \colon C^1 \to H^q$ fitting into a tight sequence
\[ C^1 \overset{\gamma}{\longrightarrow}G^q \overset{\psi}{\longrightarrow}H^q.\]
Explicitly, $\psi$ minimizes $E^q_q$ in $[\phi]$ and
\[ E^q_q(\psi)=E^q_q[\phi]=\frac{E^1_q(\psi \circ \gamma)}{E^1_q(\gamma)}=\SFto_q[\phi].\]
The multi-curves $\gamma\co C^1 \to G^q$ and $\psi \circ \gamma\co C^1 \to H^q$ are
reduced and thus minimize $E^1_q$ in their homotopy classes.
Furthermore, we can choose the curve so that
for each edge $e$ of $G^q$ and each strand $J_i$ of $C^1$, the image
of the restriction $\gamma|J_i$ meets $e$ at most twice.
\end{citethm}

Most of this is the case $p=q$ of
\cite[Theorem~6]{Thurston19:Elastic}.
The last conclusion in Theorem~\ref{thm:sf} (on $C^1$ meeting each edge
of~$G^q$ at most twice) is a consequence of
\cite[Proposition~3.19]{Thurston19:Elastic}.

\subsection{Asymptotic energies of virtual graph endomorphisms}
In this subsection, we summarize some results from \cite{Thurston20:Characterize}.

We now assume $\pi, \phi\co G_1 \rightrightarrows G_0$ is a virtual
endomorphism of graphs. Following the notation from
\S\ref{subsecn:multivalued}, for $n=1, 2, \ldots$ let
$\phi^n_0\co G_n \to G_0$ and $\pi^n_0\co G_n \to G_0$ be the induced maps.
For ease of reading, we write $\phi^n\coloneqq\phi^n_0$ and $\pi^n\coloneqq\pi^n_0$. 

Fix $1 \leq q \leq \infty$. Fix a
$q$-conformal structure on~$G_0$: for
$q > 1$, pick $q$-lengths~$\alpha_0$, or for $q=1$, pick weights~$w_0$. This
$q$-conformal structure can be lifted under the coverings
$\pi^n\co G_n \to G_0$ to yield a $q$-conformal structure
$\alpha_n$ or~$w_n$ on~$G_n$ that we denote $G_n^q$. 

\begin{definition}[Asymptotic $q$-conformal energy]
\label{defn:ae}
Suppose $1 \leq q \leq \infty$. The \emph{asymptotic $q$-conformal energy} of a
virtual endomorphism $(\pi,\phi)$ is
\[ \oE^q(\pi, \phi)\coloneqq \lim_{n \to \infty} E^q_q[\phi^n]^{1/n}\]
where $\phi^n\colon G_n^q \to G_0^q$ and we use the lifted $q$-conformal structure as above.
\end{definition}
The limit exists, and is equal to the infimum of the terms, since the
energies are invariant under passing to covers, and are
sub-multiplicative; see \cite[Prop.\ 5.6]{Thurston20:Characterize}.
Since we take homotopy classes of $\phi^n$ on the right-hand side, it follows
easily (\cite[Prop.\ 5.7]{Thurston20:Characterize}) that $\oE^q(\pi,\phi)$ only depends on the homotopy
class of $(\pi,\phi)$ as defined in \cite[Def.\ 2.2]{Thurston20:Characterize}.  In particular, it is independent of the choice of $q$-conformal structure on $G_0$.  So
we will henceforth write the asymptotic energy as $\oE^q[\pi,\phi]$. In addition, the asymptotic energy is continuous and non-increasing in the exponent~$q$ \cite[Prop.\ 6.10, Cor.\ 6.12]{Thurston20:Characterize}. We summarize these facts as follows:

\begin{proposition}[{\cite[Prop.\ 6.10]{Thurston20:Characterize}}]
\label{prop:energy}
Suppose $1 \leq q \leq \infty$ and $\pi, \phi\co G_1 \rightrightarrows G_0$
is a virtual endomorphism of graphs. Then the
asymptotic $q$-conformal energy $\oE^q[\pi,\phi]$ depends only on the
homotopy class of $[\pi, \phi]$, and is a continuous and
non-increasing as a function of~$q$.
\end{proposition}

We will assume that $\oE^\infty[\pi, \phi]<1$; equivalently,
after passing to an iterate, there is a metric on $G_0$ such that
$\phi\co G_1 \to G_0$ is contracting \cite[\S 6, p. 36]{Thurston20:Characterize}. This places us in the
setup of \S\ref{sec:ve_and_dynamics}, so we have, by Theorem \ref{thm:lc}, a locally connected limit space $\cJ$ and a positively expansive self-cover $f\co \cJ \to \cJ$. The asymptotic
energies $\oE^q[\pi, \phi]$ become numerical invariants of our
presentation of the dynamical system $f\co \cJ \to \cJ$.
\begin{question}\label{quest:invariance}
  Is $\oE^q[\pi,\phi]$ an invariant of the topological conjugacy class
  of $f \co \cJ \to \cJ$? That is,
  if two different contracting graph virtual endomorphisms
  $\pi,\phi \co G_1 \rightrightarrows G_0$ give homeomorphic
  limit spaces $\cJ$ and topologically conjugate maps~$f$ on them, do
  the energies $\oE^q$ coincide? Is this the case if they are combinatorially equivalent in the sense of \cite{Nekrashevych14:CombModel}?
\end{question}

\begin{remark}\label{rem:invariance}
We expand on Question~\ref{quest:invariance}. If we start with a
hyperbolic pcf rational map $f \co \CCa \to \CCa$, then we can take
$G_0$ to be a spine of $\CCa$ minus the post-critical set $P_f$
of~$f$, and $G_1 \coloneqq f^{-1}(G_0)$. Since any two spines of the
complement of $P_f$ are homotopy-equivalent, any two such choices will
give homotopy-equivalent virtual endomorphisms, and the same energies
$\oE^q$ and critical exponents $q^*$ and~$q_*$. But there are other
virtual endomorphisms that give the same dynamics on the limit space.
For instance, we could reindex, considering
$\pi^n_{n-1}, \phi^n_{n-1} \co G_n \rightrightarrows G_{n-1}$
as in \S\ref{subsecn:multivalued},
without
changing the dynamics of $f$ on~$\cJ$. (Note, by contrast, that
iterating to
$\pi^n_0, \phi^n_0 \co G_n \rightrightarrows G_0$ does not change $\cJ$, but
does change $f$ and~$\oE^q$ in a predictable way.) More generally, for
a hyperbolic pcf rational
map one could look at a forward-invariant set $P' \supsetneq P_f$ of
Fatou points, adding some preimages of points in~$P_f$, and construct
a virtual endomorphism from that; for $P' = f^{-n}(P_f)$, we get
$(\pi^{n+1}_n,\phi^{n+1}_n)$, but there are many other possibilities,
yielding virtual endomorphisms that are not homotopy equivalent in the
sense of \cite[Def.~2.2]{Thurston20:Characterize}
(since
the graphs $G_0$ have different rank) but give the same limiting
dynamics on $\cJ$.

For general dynamics $f \co \cJ \to \cJ$, it is unclear how to
parameterize all the different ways to see $f$ as the limit of a graph
virtual endomorphism, and thus Question~\ref{quest:invariance} is
open. In the special case of rational maps $f \co \CCa \to \CCa$ (or
expanding branched self-covers), all spines of $\CCa - P_f$
are homotopy equivalent, and so we can write $\oE(f)$ unambiguously.
Another question for future research it to consider cases where $\cJ$
is not topologically
$1$-dimensional and thus not a limit of graph virtual endomorphisms at
all: there should be some suitable replacement for the graph
energies~$E^p_q$.
\end{remark}


\section{Combinatorial modulus}
\label{sec:modulus}

In this section, we first recall the definition of the combinatorial
$q$-modulus (or, equivalently, its inverse,
the combinatorial
$q$-extremal-length)
associated to a family $\Gamma$ of curves $\gamma$ on a topological space $X$
equipped with a
finite cover~$\cU$. In fact there are several such notions, treating families of curve, of multi-curves, and of weighted multi-curves, as defined in \S \ref{subsecn:pcg}.  We next show their
coarse equivalence. Suppose $\Gamma$ is a family of curves on $X$. 
\begin{itemize}
\item There is a canonically associated family $\Gamma^w$ of weighted multi-curves $\sum w_i\gamma_i$, $\gamma_i \in \Gamma$, normalized so $\sum_i w_i=1$.  We extend the definition of combinatorial modulus to weighted families so that the moduli of $\Gamma$ and $\Gamma^w$ coincide; see Lemma \ref{lemma:weighted}.  
\item The family $\Gamma$ may be chopped up, or \emph{subdivided}, into a family of weighted multi-curves whose strands are shorter curves.  Under natural conditions, the moduli of $\Gamma$ and of the resulting subdivided family are comparable, with control; see Lemma~\ref{lem:subdivide}.  
\item If $\cU$ and $\cV$ are two covers with controlled overlap, then
  the moduli of the family $\Gamma$ with respect to $\cU$ and to $\cV$
  are comparable, with control; see Lemma~\ref{lemma:bounded_overlap}.
  As an application, two curves that are
  close, with control, have comparable modulus; see
  Lemma~\ref{lemma:regularity}.
\item If $X=G^q$ is a $q$-conformal graph, $C$ is a weighted $1$-manifold 
  as in \S \ref{subsecn:pcg}, and $\Gamma$ is the set of maps
  $\gamma\co C \to G^q$ from $C$ to $G^q$ in a fixed
  homotopy class $[\gamma]$, then the combinatorial $q$-extremal-length
  of~$\Gamma$ with respect to the covering by closed edges of~$G^q$ is
  comparable to a power of the $q$-conformal energy,
  $(E^1_q[\gamma])^q$, with control; see Proposition~\ref{prop:energies-comp}.
\end{itemize}
We conclude with a result, Theorem~\ref{thm:keith-kleiner},
which says the Ahlfors-regular
conformal dimension coincides with a critical exponent for
combinatorial modulus of weighted curve families whose elements have
diameters bounded from below. This follows from
 a result of Carrasco \cite[Corollary
1.4]{C13:Gauge} and Keith-Kleiner, which holds for unweighted curves, and Lemma~\ref{lemma:weighted}.   

\subsection{Combinatorial modulus}\label{subsec:comb-mod}
Let $\cS$ be a finite covering of a topological space~$X$, by which we
mean a multi-set of subsets $s$ of~$X$ whose union is~$X$. Note that we
allow repetitions,  and do not assume the $s$'s are open
or connected.
For $K\subset X$, denote by $\cS(K)$ the
set of elements of~$\cS$ which intersect~$K$, and let $\bigcup \cS(K)$
be the
union of these elements of~$\cS$; we may think of $\bigcup \cS(K)$ as
the ``$\cS$-neighborhood'' of~$K$.

Let $q\ge 1$.
A \emph{test metric} is a function $\rho\co \cS\to [0,\infty)$.   The
\emph{$q$-volume of~$X$ with respect to~$\rho$} is 
\[V_{q}(\rho,\cS)\coloneqq\sum_{s\in\cS} \rho(s)^q\,.\]
For $K \subset X$, the \emph{$\rho$-length} of $K$ is 
$$\ell_\rho(K, \cS)\coloneqq\sum_{s\in\cS( K)} \rho(s).$$
If $\gamma\colon C \to X$ is a curve, we define  the
$\rho$-length of its image in $X$ by
\[ \ell_\rho(\gamma, \cS) \coloneqq \sum_{s \in \cS(\gamma(C))}\rho(s),\]
i.e. identifying $\gamma$ with its image. Since $C$ is implicitly part of the data defining $\gamma$, we adopt the briefer notation $\cS(\gamma)$ for $\cS(\gamma(C))$.  

We pause to clarify some subtle points.  First, each sum is a sum over a multi-set.  So, if an element $s$ of $\cS$ is repeated, say twice, then the sum for $V_q(\rho, \cS)$ has two corresponding terms.  Next, recall that the domain of a curve is connected.  The $\rho$-length of a curve $\gamma$ depends only on its image $\gamma(C)$, and not on its parameterization.  So, for example, if $\gamma$ parameterizes a simple loop, then $\ell_\rho(\gamma, \cS)=\ell_\rho(\gamma^n, \cS)$ for all $n \geq 1$.  And finally, more generally, the computation of $\ell_\rho(\gamma, \cS)$ does not depend on the number of times $\gamma$ meets an element $s \in \cS$.  See \S\ref{sec:subdivision} for alternatives.

If $\Gamma$ is a family of curves in $X$, we make the following
definitions leading up to
the \emph{combinatorial modulus of $\Gamma$ with respect to~$\cS$},
denoted $\modulus_q(\Gamma,\cS)$.
\begin{equation}\label{eq:mod-q}
  \begin{aligned}
  L_\rho(\Gamma,\cS)&\coloneqq\inf_{\gamma\in\Gamma} \ell_\rho(\gamma,
                      \cS)\\
  \modulus_q(\Gamma, \rho,\cS)&\coloneqq\frac{V_{q}(\rho,\cS)}{L_\rho(\Gamma,\cS)^q}\\
 \modulus_q(\Gamma,\cS) &\coloneqq \inf_{\rho}\modulus_q(\Gamma,\rho,\cS).
  \end{aligned}
\end{equation}
In the last infimum, we restrict to test metrics~$\rho$ for which
$L_\rho(\Gamma,\cS) \ne 0$.
We often use an equivalent formulation of $\modulus_q(\Gamma,\cS)$
more familiar to analysts. A test metric is \emph{admissible}
for~$\Gamma$ if $\ell_\rho(\gamma, \cS) \geq 1$ for each
$\gamma \in \Gamma$. Then it is easy to see that
\[ \modulus_q(\Gamma, \cS)=\inf\,\{\,V_q(\rho, \cS) \mid \rho \; \textrm{admissible
    for $\Gamma$}\, \}.\]
The combinatorial modulus of a nonempty family of curves
is always finite and positive. 

The \emph{combinatorial extremal length} of a curve family $\Gamma$ is the reciprocal of the combinatorial modulus:
\[ \EL_q(\Gamma, \cS)=\sup_{\rho} \frac{L_\rho(\Gamma,
    \cS)^q}{V_q(\rho, \cS)}.\]
We will relate this to
$E^1_q[\gamma]$ in the sense of \S\ref{sec:energies}; see
\S\ref{sec:comb-mod-graph}.

\begin{remark}
\label{remark:singlecurve}
 In this finite combinatorial world, it makes sense to take
$\Gamma=\{\gamma\}$, a single curve. That is, for any $q\geq 1$, any
finite cover $\cS$, and any curve $\gamma$, we have a finite and nonzero
$\modulus_q({\gamma}, \cS)$. Essentially, discretizing via the
coverings thickens a single curve so that it
behaves like a family.
\end{remark} 

\subsection{Weighted multi-curves}\label{sec:weighted-curves}
Recall that a \emph{weighted multi-curve} is a finite formal sum
$\gamma=\sum w_i\gamma_i$ where $w_i>0$ and the $\gamma_i$'s are curves
on~$X$. It is \emph{normalized} if $\sum w_i=1$.

We will require two sorts of families of weighted multi-curves on a space $X$.

\begin{itemize}
\item Fix a multicurve $\gamma\colon C \to X$ with strands
  $\gamma_i\colon J_i \to X$, $J_i \in \{I, S^1\}$,  and fix
  corresponding positive weights $w_i$. Then we may consider the
  family whose elements are weighted multi-curves  $\gamma=\sum w_i
  \gamma_i$. The number of strands and the weights are fixed within
  the family.

\item Fix a set $\Gamma$ of unweighted curves on $X$. Then we may
  consider the family whose elements are weighted multi-curves
  $\gamma=\sum_i w_i \gamma_i$ where  $\gamma_i \in \Gamma$. Here,
  both the number of strands and the weights may vary within the
  family.
\end{itemize}

To be concrete about the second type of family, suppose
$\Gamma$ is a family of unweighted curves on $X$. Define 
\[ \Gamma^w\coloneqq\Bigl\{\,\sum
w_i\gamma_i\Bigm|\gamma_i \in \Gamma, \sum w_i=1\,\Bigr\}.\]
This is a family of normalized weighted curves canonically associated
to $\Gamma$. The assignment $\gamma \mapsto 1\cdot \gamma$ yields an
inclusion $\Gamma \hookrightarrow \Gamma^w$.  So, 
as families of weighted curves, we have $\Gamma \subset
\Gamma^w$.

We next extend the definition of combinatorial modulus to families of weighted multi-curves.  Let $\cS$ be a finite covering of $X$, and let $q\ge 1$. 
If $\rho\co \cS \to [0,\infty)$ is a test metric, the \emph{$\rho$-length} of a weighted multi-curve $\gamma=\sum w_i\gamma_i$ is given by 
\[ \ell_\rho(\gamma, \cS)\coloneqq\sum w_i \ell_\rho(\gamma_i).\]
We then define the combinatorial modulus $\modulus_q(\Gamma^w, \cS)$ using Eq.~\eqref{eq:mod-q}.

\begin{lemma}[Moduli of weighted curves]
\label{lemma:weighted} Suppose $\Gamma$ is a family of unweighted
curves, and $\Gamma^w$ is the corresponding family of normalized weighted curves.
For any $q \geq 1$, we have 
$\modulus_q(\Gamma^w, \cS) = \modulus_q(\Gamma, \cS)$. 
\end{lemma}

\begin{proof} The inequality $\modulus_q(\Gamma^w, \cS) \geq
  \modulus_q(\Gamma, \cS)$ holds since $\Gamma^w \supset \Gamma$.  Now
  suppose $\rho$ satisfies $\ell_\rho(\gamma, \cS) \geq 1$ for each
  $\gamma \in \Gamma$, and suppose $c=\sum_i w_i\gamma_i \in \Gamma^w$. Then
\[ \ell_\rho(c,\cS)=\sum_i w_i\left(\sum_{s \cap \gamma_i \neq\emptyset}\rho(s) \right)
  = \sum_i w_i\ell_\rho(\gamma_i, \cS)
  \geq \sum_i w_i =1.\]
Thus each admissible metric for $\Gamma$ is also admissible for
$\Gamma^w$, so
$\modulus_q(\Gamma^w, \cS) \leq \modulus_q(\Gamma, \cS)$.
\end{proof} 

\subsection{Subdivision}
\label{sec:subdivision}
We now turn to subdividing weighted multi-curves. We only need this construction for
the case $\Gamma = \{\gamma\}$, a single weighted multi-curve $\gamma=\sum_i w_i\gamma_i$, and not for curve families. 

Suppose $\gamma = \sum w_i \gamma_i$ is a weighted multi-curve corresponding to a
map $\gamma\co  C \to X$, where $C=\bigcup_i w_iJ_i$ and $\gamma_i\co  J_i
\to X$, where each
$J_i$ is either an interval or a circle. Suppose
that for each $i$ we are given a finite set of maps
$\alpha_{i,k} \co I_{i,k} \to J_i$, where each $I_{i,k}$ is a copy of the
unit interval and the images of the $\alpha_{i,k}$ cover $J_i$ without
gaps
or overlaps, except for singleton points on the boundary.
Then the
corresponding \emph{subdivision} of $\gamma$ is
the weighted multi-curve on $X$
\[
  \sum_i \sum_k w_i (\gamma_i \circ \alpha_{i,k}).
\]
corresponding to the map 
\[ \zeta\co  D=\bigcup_{i,k} w_i I_{i,k} \to X\]
with component maps $\zeta_{i,k} = \gamma_i \circ \alpha_{i,k}$.
The subdivision of a normalized weighted multi-curve is not
usually normalized.

\begin{lemma}[Subdivision]
\label{lem:subdivide} Suppose $\gamma=\sum_iw_i\gamma_i$ is a weighted
multi-curve on $X$, and $\cS$ is a finite cover of $X$. Let $q \geq 1$,
and let $\zeta$ be a subdivision of~$\gamma$. Then $\modulus_q(\zeta,\cS) \le
\modulus_q(\gamma,\cS)$.

Furthermore, if no $\zeta_{i,k}$ has image contained in a single element
of~$\cS$ and there is a constant~$K$ so that, for each
$s \in \cS$ and each strand~$i$, the number of connected components of
$\gamma_i^{-1}(s)$ is bounded by~$K$, then
$\modulus_q(\gamma,\cS) \lesssim_{K,q} \modulus_q(\zeta,\cS)$.
\end{lemma}
  
\begin{proof}
  Suppose $\rho$ is any test metric.  For the first assertion, we have
  $\ell_\rho(\zeta,\cS) \ge \ell_\rho(\gamma,\cS)$. The inequality for
  $q$-modulus then follows from the definitions.

  For the second assertion, the given conditions imply that if a given
  element $s \in \cS$ meets
  some $\gamma_i$, it then meets at most $2K$ of the restricted curves
  $\zeta_{i,k}$. Thus
  $\ell_\rho(\zeta,\cS) \le 2K \cdot \ell_\rho(\gamma,\cS)$. It
  follows that
  $\modulus_q(\zeta,\cS) \ge (2K)^{-q} \modulus_q(\gamma,\cS)$.
\end{proof}

\subsection{When moduli are comparable}
\label{sec:regularity}
Several different coverings naturally arise in our setting on graphs.
We show that given a curve family, the corresponding $q$-moduli are coarsely equivalent when certain natural regularity conditions hold.

Two coverings $\cU, \cV$ of a topological space $X$ have
\emph{$K$-bounded overlap} if, for each $u \in \cU$, we have
$1 \le \#\cV(u) \le K$, and vice-versa for each $v \in \cV$ we have
$1 \le \#\cU(v) \le K$.  Here, the notation follows that from \S \ref{subsec:comb-mod}, so that $\cV(u)=\{v \in \cV \mid v \cap u \neq \emptyset\}$ and $\cU(v)=\{u \in \cU \mid u \cap v \neq \emptyset\}$. 
A covering~$\cU$ is \emph{$K$-bounded}
if it has $K$-bounded overlap with itself.

\begin{lemma}[Bounded overlap]
\label{lemma:bounded_overlap}
Suppose $\cU, \cV$ are two finite coverings of a topological space $X$ with $K$-bounded overlap.  
Let $\Gamma$ be a family of weighted multi-curves on $X$.
Then, for any $q\geq 1$,
\begin{align*}
    \modulus_q(\Gamma, \cU) &\asymp_{K,q}\modulus_q(\Gamma,\cV)\\
  \intertext{or, more precisely,}
    \modulus_q(\Gamma, \cU) &\in [K^{-q-1},K^{q+1}]\cdot\modulus_q(\Gamma, \cV).
\end{align*}
\end{lemma}

\begin{proof}
We 
  imitate the argument in \cite[Theorem
  4.3.1]{cannon:swenson:characterization}. Suppose
  $\sigma\co  \cV \to \mathbb{R}$ is a test metric for $\cV$. Define
  a test metric $\rho\co  \cU \to \mathbb{R}$ and a function
  $f\co  \cU \to \cV$ by setting
\[ \rho(u) = \sigma(f(u)) \coloneqq\max\,\{\sigma(v)\mid u \cap v \neq
  \emptyset\},
\]
i.e., $f(u)\in \cV$ is any element that realizes the maximum. 

For any strand $\gamma$ of an element in $\Gamma$, we have
\[ \ell_\sigma(\gamma,\cV)
  = \sum_{\substack{v\in\cV\\v \cap \gamma \neq \emptyset}}\sigma(v)
  \leq \sum_{\substack{u\in\cU\\u \cap \gamma \neq \emptyset}}\Biggl( \sum_{\substack{v\in\cV\\v \cap u \neq \emptyset}} \sigma(v)\Biggr)
  \leq \sum_{\substack{u\in\cU\\u \cap \gamma \neq \emptyset}}K\rho(u)
  =K\ell_\rho(\gamma,\cU).\]
So for any weighted multi-curve $\gamma=\sum w_i\gamma_i$ in $\Gamma$, we therefore have by linearity that 
\[ \ell_\sigma(\gamma,\cV) \leq K\ell_\rho(\gamma, \cU).\]
Taking the infimum over the set $\Gamma$, we conclude 
\[ L_\rho(\Gamma, \cU) \geq \frac{1}{K}\cdot L_\sigma(\Gamma, \cV).\]
From the definition and $K$-bounded overlap, it is straightforward to see that 
\[ V_q(\rho,\cU) \leq K\cdot V_q(\sigma,\cV).\]
We conclude from the definition of modulus as the infimum  of the ratio $V_q(\cdot)/(L(\cdot))^q$ that 
\[\modulus_q(\Gamma,\cU) \leq K^{q+1}\modulus_q(\Gamma, \cV).\]
The other bound follows by symmetry. 
\end{proof}
We need additional notation to set up the statement of the next lemma.
For $\cS$ a cover of~$X$, let $\cS^2$ be the cover whose elements are
$\bigcup \cS(s)$ for $s\in \cS$. (That is, we take the union of all
elements of~$\cS$ that intersect~$s$.) Inductively define $\cS^N$ to
be the cover whose elements are $\bigcup \cS(t)$ for
$t \in \cS^{N-1}$.  The elements of $\cS^N$ are in natural bijection
with those of~$\cS$, but are larger ``combinatorial $\cS$-neighborhoods''.

\begin{lemma}[Fellow travellers]
  \label{lemma:regularity}
  Let $X$ be a topological space equipped with a $K$-bounded finite
  cover~$\cS$. Fix a compact $1$-manifold $C=\bigcup_i J_i$ and corresponding positive weights $w_i$,  and suppose $\gamma^j=\sum_iw_i\gamma_i^j, j=1, 2$ are two normalized weighted multi-curves given by maps $\gamma^j\co C \to X$. 
 Suppose there is an $N \in \mathbb{N}$ so that each strand
 $\gamma_i^1$ of $\gamma^1$  is contained in
 $\bigcup\cS^N(\gamma_i^2)$, and similarly $\gamma_i^2$ is contained
 in $\bigcup\cS^N(\gamma_i^1)$.
 Then for all $q \geq 1$, we have 
 \[\modulus_q(\gamma^1, \cS) \asymp_{K,N,q} \modulus_q(\gamma^2, \cS).\] 
\end{lemma}

Note that the statement does not concern \emph{families} of multi-curves---just single multi-curves. The content of the lemma is that the  bound depends only on the constants $K, N, q$.

\begin{proof} Fix $q \geq 1$. Fix temporarily a test metric $\rho\co \cS
  \to [0,\infty)$. For $s \in \cS$ we denote by $\hat{s}$ the
  corresponding element of $\cS^N$. We denote by $\widehat{\rho}\co
  \cS^N \to [0,\infty)$ the test metric obtained by setting
  $\widehat{\rho}(\widehat{s})\coloneqq\rho(s)$.

  Fix a strand $\gamma_i^2$ of $\gamma^2$. By assumption,
  $\gamma_i^2 \subset \cS^N(\gamma_i^1)$. This implies that if
  $s \in \cS(\gamma_i^2)$, then $\widehat{s} \in \cS^N(\gamma_i^1)$.
  Thus
  $\ell_{\widehat{\rho}}(\gamma_i^1, \cS^N) \geq \ell_\rho(\gamma^2_i,
  \cS)$ and so via linear combinations and the definition we have
  $\ell_{\widehat{\rho}}(\gamma^1, \cS^N) \geq \ell_\rho(\gamma^2, \cS)$. Since
  $\gamma^1, \gamma^2$ are weighted curves, as opposed to families, we have
  $L_{\widehat{\rho}}(\gamma^1, \cS^N) \geq L_\rho(\gamma^2, \cS)$. Since
  $V_q(\widehat{\rho}, \cS^N)=V_q(\rho, \cS)$, we conclude
\[ \frac{V_q(\widehat{\rho}, \cS^N)}{L_{\widehat{\rho}}(\gamma^1, \cS^N)^q}\leq \frac{V_q(\rho, \cS)}{L_\rho(\gamma^2, \cS)^q}.\]
Taking $\rho$ to realize the infimum of the right-hand ratio, we conclude 
\[ \modulus_q(\gamma^1, \cS^N) \leq \modulus_q(\gamma^2, \cS).\]
The covering $\cS^N$ has $K^N$-bounded-overlap with the covering $\cS$. By Lemma \ref{lemma:bounded_overlap}, we conclude 
\[ \modulus_q(\gamma^1, \cS) \lesssim_{K, N, q}\modulus_q(\gamma^1, S^N) \leq _{K, N, q} \modulus_q(\gamma^2, \cS).\]
The other bound follows by symmetry. 
\end{proof}

The next two lemmas are technically convenient. 

A graph~$G$ has a natural covering~$\cE$ whose elements are the
closed edges of~$G$. There is also a \emph{partial
  covering}~$\cE^\circ$ whose elements are the interiors of the edges
of~$G$. Since~$\cE^\circ$ does not cover all of~$G$, there are
certain curve families (constant curves at vertices) with length~$0$ and
undefined modulus. However, any non-constant curve intersects at least
one element of $\cE^\circ$, so if we restrict ourselves to families of
non-constant curves we do not run into trouble defining $q$-modulus with respect to $\cE^\circ$.

\begin{lemma}
\label{lemma:openedges}
  If $G$ is a graph of valence bounded by~$M$ and $\Gamma$ is any family of weighted multi-curves on~$G$, not necessarily normalized, each of whose strands is non-constant, then
  for any $1 \le q \le \infty$,
  \[
    \modulus_q(\Gamma, \cE) \asymp_M \modulus_q(\Gamma, \cE^\circ).
  \]
\end{lemma}

\begin{proof}
  This is similar to Lemma~\ref{lemma:bounded_overlap}, but that lemma
  does not apply directly, since $\cE^\circ$ is not a covering. But
  the same techniques work. Note that there is a canonical bijection $\cE \leftrightarrow \cE^\circ$ and so a test metric $\rho$ on one determines uniquely a test metric on the other that we denote with the same symbol. 

  For any test metric~$\rho$ on~$\cE^\circ$, we have
  $L_\rho(\Gamma, \cE^\circ) \le L_\rho(\Gamma, \cE)$. 
  Since $V_q(\rho, \cE)=V_q(\rho, \cE^\circ)$, it follows
  that
  \[\modulus_q(\Gamma, \cE) \le \modulus_q(\Gamma, \cE^\circ).\]

  Conversely, given a test metric~$\rho$ on~$\cE$, define a test metric~$\sigma$ on $\cE^\circ$ by setting $\sigma(e)$ to be the
  maximum of~$\rho(e')$ for $e'$ equal to $e$ or any of its neighbors.
  In a graph of valence at most $M$, a closed edge meets itself and at most $2(M-1)$ other closed edges. 
  Thus $L_\rho(\Gamma,\cE) \le L_\sigma(\Gamma, \cE^\circ)$ and
  $(2M-1)V_q(\rho,\cE) \ge V_q(\sigma, \cE^\circ)$, so
  \[\modulus_q(\Gamma,\cE) \ge \modulus_q(\Gamma,\cE^\circ)/(2M-1).\qedhere\]
\end{proof}

We use Lemma~\ref{lemma:openedges} to make cleaner estimates in the proof of Proposition \ref{prop:energies-comp} below, since
the elements of the collection $\cE^\circ$ do not overlap.

\begin{definition}[Stars]
\label{def:star}
Suppose $G$ is a length graph with each edge of length $1$, and $e$ is an edge of $G$. The \emph{open star} $\widehat{e}$ of $e$ is the open-$1/3$-neighborhood of $e$:
\[ \widehat{e}\coloneqq\bigcup_{x \in \overline{e}}B(x, 1/3).\]
We denote by $\widehat{\cE}$ the covering of $G$ by
(open) stars of edges. 
\end{definition}

\begin{lemma}
\label{lemma:stars}
  If $G$ is a graph of valence bounded by~$M$ and $\Gamma$ is any family of weighted multi-curves on~$G$, not necessarily normalized, then
  for any $1 \le q \le \infty$,
  \[
    \modulus_q(\Gamma, \cE) \asymp_M \modulus_q(\Gamma, \widehat{\cE}).
  \]
\end{lemma}

\begin{proof} $\cE$ and $\widehat{\cE}$ have $M$-bounded overlap. 
Apply Lemma \ref{lemma:bounded_overlap}.
\end{proof}

\subsection{Combinatorial modulus and graph energies}
\label{sec:comb-mod-graph}
Our next task is to relate the $q$-combinatorial modulus of the family comprised of a homotopy class $[\gamma]$ of weighted multi-curve $\gamma$ with the
graphical energy $E^1_q[\gamma]$ from
\S\ref{sec:energies}.

Suppose $\gamma\co C=\bigcup_iw_iJ_i \to G^q$ is  a weighted multi-curve
on a $q$-conformal graph~$G^q$, where each $J_i=S^1$. We have defined two
quantities. From
\S\ref{sec:energies}, we have the graphical energy
$E^1_q[\gamma]$. We have also the combinatorial modulus
$\modulus_q([\gamma], \cE)$.  The main result of this
section is that these two quantities are comparable.

\begin{proposition}
  \label{prop:energies-comp}
 Suppose $q \geq 1$ and $m \geq 1$. Let $G^q$ be a $q$-conformal graph with valence bounded by~$M$ and
  $q$-lengths $\alpha(e)\in[1/m,m]$ bounded above and below. Let $\cE$ denote the covering of $G^q$ by closed edges. Let
  $\gamma \co C \to G$ be a non-empty weighted multi-curve on~$G^q$ with no
  null-homotopic components.
  Suppose $N$ is an upper bound on the number of times a
  reduced representative in~$[\gamma]$ passes over each edge of~$G^q$. Then
  \[ (E^1_q[\gamma])^q =_{\mathrm{def}} \EL_q([\gamma],\alpha) \asymp_{m,M,N,q} \EL_q([\gamma], \cE) =_{\mathrm{def}} 1/\modulus_q([\gamma],\cE).\]
\end{proposition}

\begin{proof}
  Up to reparameterization, there is a unique reduced
  representative in~$[\gamma]$; we assume that $\gamma = \sum_i w_i \gamma_i$ is this
  representative. We will then use Lemma~\ref{lemma:openedges} to work
  with the partial covering $\cE^\circ$ rather than $\cE$.

  We then have a quantity $\modulus_q(\gamma, \cE^\circ)$.  Since $\{\gamma\} \subset [\gamma]$ we have $\modulus_q(\gamma, \cE^\circ)\leq \modulus_q([\gamma], \cE^\circ)$.  Since the underlying set of any
  strand of a curve in $[\gamma]$ contains that of the corresponding strand in the reduced representative $\gamma$, we
  have $\modulus_q(\gamma, \cE^\circ)\geq \modulus_q([\gamma], \cE^\circ)$.
  Thus $\modulus_q(\gamma, \cE^\circ)=\modulus_q([\gamma], \cE^\circ)$. Taking
  reciprocals, we conclude $\EL_q(\gamma, \cE^\circ)=\EL_q([\gamma],
  \cE^\circ)$.

  Let $\rho\co  \cE^\circ \to \mathbb{R}$ be a test metric; we will
  assume $\rho(e) > 0$ for each edge. (This does not change the
  relevant infima/suprema, although it does mean they will not always
  be realized.)
 By assumption, $n_\gamma(e) \leq N$.  We have then 
\[ L_\rho(\gamma, \cE^\circ)=\ell_\rho(\gamma, \cE^\circ)  =  \sum_i w_i\ell_\rho(\gamma_i, \cE^\circ) = \sum_i w_i \sum_{e \subset \gamma_i}\rho(e) \asymp_N \sum_e
  n_\gamma(e)\rho(e);\]
the first equality coming from the fact that $\gamma$ is a single multi-curve and not a family. 
For fixed non-zero~$\rho$, consider the length graph $K^\infty_\rho$ with
underlying graph~$G$ but lengths~$\rho$ on the edges, and let $\psi_\rho
\co G^q \to K^\infty_\rho$ be the identity map (which we introduce to keep
track of which lengths to use). Then the right-hand side above is
$E^1_\infty(\psi_\rho \circ \gamma)$.
We conclude 
\[ L_\rho(\gamma, \cE^\circ) \asymp_N E^1_\infty(\psi_\rho \circ \gamma).\]

 In addition,
\[
  V_q(\rho,\cE^\circ) = \sum_e \rho(e)^q \asymp_{m,q} \sum_e
  \alpha(e)^{1-q} \rho(e)^q = \bigl(E^q_\infty(\psi_\rho)\bigr)^q,
\]
so
\[ \EL_q(\gamma, \cE^\circ) \overset{{\mathrm{def}}}{=}\sup_\rho \frac{L_\rho(\gamma, \cE^\circ)^q}{V_q(\rho, \cE^\circ)}\asymp_{m,N,q}
  \sup_\rho \left(\frac{E^1_\infty(\psi_\rho \circ \gamma)}{E^q_\infty(\psi_\rho)}\right)^q=_{\text{Prop.~\ref{prop:duality}}}\EL_q([\gamma],\alpha),\]
since maximizing over~$\rho$
is equivalent to maximizing over metric graphs~$K^\infty$.
\end{proof}

\subsection{Conformal dimension and combinatorial modulus}
\label{subsecn:conf-dim-comb-mod}
In this subsection, we give the relationship between critical
exponents for combinatorial modulus and conformal dimension
used in our main result.

\begin{citethm}[{\cite[Corollary 1.4]{C13:Gauge}}]
\label{thm:keith-kleiner}
Let $X$ be a connected, locally connected, approximately
self-similar
metric space. For $\delta>0$ denote by $\Gamma_\delta$ the family of
curves of diameter bounded below by $\delta$, and by $\Gamma_\delta^w$
the corresponding family of normalized weighted multi-curves from \S \ref{sec:weighted-curves}. Suppose
$(\cS_n)_{n=0}^\infty$ is a sequence of snapshots at some scale parameter. 

Then the Ahlfors-regular conformal gauge of $X$ is nonempty, and for some $\delta_0>0$ and all $0<\delta<\delta_0$, 
\[ \ARCdim(X)=\inf\,\{\,q \mid \modulus_q(\Gamma^w_\delta, \cS_n) \to 0 \; \mbox{when}\; n \to +\infty\,\}.\]
\end{citethm}
This theorem in the case of unweighted curves was proved by Carrasco
and, in unpublished work, Keith and Kleiner.
The statement as given above
follows from their result and Lemma \ref{lemma:weighted}.

The results above let us characterize the conformal dimension in terms
of curve families.
Consider the limit
dynamical system $f\co  \cJ \to \cJ$ associated to a forward-expanding recurrent virtual
graph endomorphism $\pi, \phi\co  G_1 \rightrightarrows G_0$.
We have a covering $\cU_0$ of $\cJ$ by connected open sets, which by Theorem \ref{thm:lc} we
can choose to be arbitrarily fine. Via
iterated pullback we get a sequence $\cU_n$ of finite covers that form
a sequence of snapshots at some scale parameter $\theta$  with respect
to a visual metric~$d_{vis}$. We then
have the following.

\begin{corollary}
\label{cor:critexp_J}
For all sufficiently small $\delta>0$, 
\begin{align*}
  \ARCdim(\cJ,d_{vis})
    &= \inf\,\{q \mid \modulus_q(\Gamma^w_\delta, \cU_n) \to 0 \; \mbox{when}\; n \to +\infty\}\\
    &= \sup\,\{q \mid \modulus_q(\Gamma^w_\delta, \cU_n) \to \infty \; \mbox{when}\; n \to +\infty\}.
\end{align*}
At the critical exponent, the combinatorial modulus is bounded from below. 
\end{corollary}

\begin{proof}
The first equality holds by Theorem~\ref{thm:keith-kleiner}; the second equality and the last assertion follow from \cite[Corollary 3.7]{MR3019076}. 
\end{proof}


\section{Sandwiching \textalt{$\ARCdim$}{ARCdim}} 
\label{sec:sandwich} 

To prove Theorem~\ref{thm:crit-sandwich},
we begin by supposing
$\pi, \phi\co G_1 \rightrightarrows G_0$ is a
$\lambda$-backward-contracting ($\lambda<1$), recurrent, virtual endomorphism of
graphs. 

In fact, we will prove something slightly stronger, namely, that  the
conclusion holds under the weaker assumption that $\oE^\infty[\pi,
\phi]<\lambda < 1$. (The earlier assumption was that $E^\infty(\phi) < 1$.)
We begin by making some basic simplifications. 

First, we may assume $G_0$ has no vertices of valence $1$. 
Next, a map $\phi\colon G_1 \to G_0$ between connected graphs without leaves which induces a  surjection on the fundamental group is necessarily surjective.  
Finally, the pullback of a surjective map under a connected cover is again surjective. 
In summary, we may assume $G_0$, hence $G_1$ and indeed all the
$G_n$, have no leaves, and that for all $n$, each of $\phi^{n+1}_n$,
$\phi^n_0$, and~$\phi^\infty_n$ are surjective.  

There is a uniform upper bound, say $M$, of the valences of the graphs $G_n$, $n \in \mathbb{N}$.

For technical reasons, we subdivide any loops
of~$G_0$ by adding a vertex on the interior.  
Since asymptotic energies are independent of the chosen metric, for later convenience, 
we choose the metric so each
edge of $G_0$ has length $1$. Together, these conventions serve to make the star (recall Definition~\ref{def:star}) of any edge at
any level contractible.

The assumption $\oE^\infty[\pi, \phi]<\lambda < 1$ implies there exists
$N\geq 1$ such that $\phi^N_0\co G_N \to G_0$ is
contracting, say with some other constant $\lambda'<1$.  We iterate the pair $(\pi, \phi)$
so that the new $G_1$ is the old~$G_N$, and reset $\lambda\coloneqq\lambda'$. This modification does not alter the critical exponents $p_*, p^*$. 

We apply the construction in \S \ref{sec:ve_and_dynamics} to present
the limit space $\cJ$ as an inverse limit; we let
$G_n$, $\phi^\infty_n$, $\pi^\infty_n$, etc., be the corresponding spaces
and maps as in that section. Theorem~\ref{thm:lc} implies that the
limit dynamics $f\co \cJ \to \cJ$ associated to $(\pi, \phi)$ is
topologically cxc with respect to a finite open
cover $\cU_0$ by open connected sets such that the mesh of $\cU_n$
tends to zero as $n \to \infty$. We may also take the mesh of $\cU_0$
to be as small as we like.

We equip the limit space $\cJ$ with a visual metric $d_{vis}$ with parameter
$\theta<1$ as in Theorem~\ref{thm:visual}. The collection of coverings
$(\cU_n)_n$ is then a sequence of
snapshots with parameter~$\theta$ (\S \ref{subsecn:visual}).  Inconveniently, the elements of $\cU_n$ are not clearly related to structures in the graphs $G_n$. We therefore 
define another collection of coverings of $\cJ$ as follows. Given
$n \in \mathbb{N}$ and $e \in E(G_n)$, we denote by
\[ V(e)\coloneqq (\phi^\infty_n)^{-1}(\widehat{e})\subset\cJ\]
the open subset of~$\cJ$ lying over the star of~$e$ in~$G_n$;
$\cV_n$ is the covering of $\cJ$ by these sets. Unlike $\cU_n$, the sets in
$\cV_n$ are not necessarily connected.

In \S \ref{subsec:snap-stars}, we show that the coverings $\cU_n$ and $\cV_n$ are quantitatively comparable. As a consequence, the sequence $\cV_n$ is also a sequence of snapshots. In what follows, we work exclusively with the sequence of coverings $(\cV_n)_n$. 

In \S \ref{subsec:jg}, we relate weighted multi-curves on $G_n$ to weighted multi-curves on $\cJ$, and show comparability of moduli with respect to natural coverings. We do this by applying the constructions in \S \ref{subsecn:sections}. 

\S \ref{subsec:leq} and \S \ref{subsec:geq} give the proofs of the upper and lower bounds on $\ARCdim(\cJ)$, respectively.

\subsection{Snapshots from stars}
\label{subsec:snap-stars}

\begin{lemma} We have the following properties of the $\cV_n$.
\label{lem:V}
\begin{enumerate}
\item For $0 \leq k < n$, $e$ an edge of $G_k$, and $\wte$ an edge of $G_n$ with
  $\pi_n^k(\wt e) = e$, the map $f^{\circ (n-k)}\co V(\wte) \to V(e)$ is a homeomorphism.
\item With respect to the visual metric, $\mathrm{mesh}(\cV_n) \to 0$.
\end{enumerate}
\end{lemma}

\begin{proof} The first assertion follows from the fact that
  stars are contractible, $\phi^\infty_n$ is the pullback of $\phi^\infty_k$ under
  $\pi^n_k$, the presentation of $\cJ$ as the inverse limit of the
  sequence $(\phi^{n+1}_n)_n$, and the definition of the dynamics on the limit
  space $f\co \cJ \to \cJ$ as induced by~$\pi$.

For the second, we argue directly, using the definition of the product
topology.
For $x \in \cJ$,  write
$x=(x_0, x_1, x_2, \ldots) \in \prod_{i=0}^\infty G_i$. 
Then there is a neighborhood
basis of $x$ in the product topology on~$\cJ$ consisting of sets of the form
\[ U(x; \epsilon, M)\coloneqq \left(\prod_{i=1}^M B_{G_i}(x_i,
    \epsilon) \times \prod_{i=M+1}^\infty G_i \right) \cap \cJ
\]
for $\epsilon>0$ and $M \in \mathbb{N}$.
Fix such a neighborhood $U(x; \epsilon, M)$. 
Choose $N>M$ large enough that $2\lambda^{N-M} < \epsilon$, and
suppose $n>N$, $e \in E(G_n)$, and $x \in V(e)$. Since $\diam_{G_n}(\widehat{e})<2$, we have for
$i \leq M$ that
\[ \diam_{G_i}(\phi^n_i(\widehat{e})) < 2 \lambda^{n-i}\leq 2\lambda^{N-M}<\epsilon.\]
Hence for each $0 \leq i \leq M$ we have
$\phi^\infty_i(V(e)) \subset B_{G_i}(x_i,\epsilon)$ and so
$V(e) \subset U(x; \epsilon, M)$. The conclusion follows since $\cJ$
is compact, and the topology determined by the visual metric is the
same as that induced from the product topology.
\end{proof}

\begin{lemma}[The $\cU_n$ are comparable to the $\cV_n$]
\label{lem:UV}
There exists $N_0, N_1 \in \mathbb{N}$ with the following property.

For each $n\geq N_0$ and each $V \in \cV_n$, there exist $U \in \cU_{n-N_0}$ and $U' \in \cU_{N_1+n-N_0}$ such that 
\[ U' \subset V \subset U.\]
\end{lemma}

Recall that \emph{Lebesgue number} of a finite open cover~$\cU$ is the largest
$\delta>0$ such that any subset of diameter less than $\delta$ is
contained in an element of~$\cU$.

\begin{proof} Lemma \ref{lem:V} implies there exists $N_0$
  such that the mesh of $\cV_{N_0}$ is smaller than the Lebesgue
  number of $\cU_0$. 
  
  We first show the conclusion holds for $n=N_0$.  Fix $V \in \cV_{N_0}$, so that $V=V(e)$,
  where $e \in E(G_{N_0})$. The choice of $N_0$ implies that
  $V \subset U$ for some $U \in \cU_0$. To find $U'$, let $y$ be the
  midpoint of~$e$. By the uniform continuity of $\phi^\infty_{N_0}$,
  there exists $\delta>0$ such that the image of any visual
  $\delta$-ball under $\phi^\infty_{N_0}$ has diameter at most $1/4$.
  Theorem \ref{thm:visual} implies there exists $N_1>N_0$ such that
  $\diam U' < \delta$ for each $U' \in \cU_{N_1}$. Pick any
  $x \in \cJ$ with $\phi^\infty_{N_0}(x)=y$, and pick any
  $U' \in \cU_{N_1}$ containing~$y$. Then
  $\phi^\infty_{N_0}(U') \subset B(x_{N_0}(e), 1/4) \subset e \subset
  \widehat{e}$ and so $U' \in V(e)$.

We now prove the conclusion holds for general $n \geq N_0$.  Given such $V\in \cV_n$, rename it $\wtV$, put $V\coloneqq f^{n-N_0}(\wtV)$, and
apply the argument above to obtain $U'$ and $U$. Now let $\wtU'$ and $\wtU$ be connected components of the 
preimages of $U'$ and $U$, respectively, under $f^{n-N_0}$ meeting $\wtV$. Recalling the construction in \S \ref{subsecn:convention}, each maps homeomorphically to its image. Renaming
$\wtU'$ to $U'$ and $\wtU$ to $U$ yields the result.
\end{proof}

\begin{proposition}
\label{prop:qp_from_stars}
There exists $m \in \mathbb{N}$ such that the family
$(\cV_n)_{n \geq m}$ is a family of snapshots with
parameter~$\theta$.
\end{proposition}

\begin{proof}
  Let $r_0$ be the constant from Theorem
  \ref{thm:visual}(\ref{item:scaling}). Lemma \ref{lem:V}(2) implies
  there is $m \in \mathbb{N}$ such that for $n \geq m$, each
  $V \in \cV_n$ has diameter at most~$r_0$, and hence the restriction
  $f\co V \to f(V)$ is a homeomorphism which scales distances by the
  factor~$\theta^{-1}$. 
  
  Recall the graphs $G_n$ are equipped with length metrics in which
  edges are isometric to the unit interval. For each
  $V=V(e) \in \cV_{m}$, let $x_V \in \cJ$ be a point which projects
  under $\phi^\infty_{m}$ to the midpoint of the edge $e$. 
  Since there are finitely many edges at level
  $m$, we can find $s>0$ such that for each $V \in \cV_{m}$,
  $\diam(\phi^\infty_{m}(B(x_V, s)))<1/2$. This choice of $s$ guarantees
  that the visual balls $B(x_V, s)$ for $V \in \cV_{m}$ are pairwise disjoint. Let
  $D=\max\{\diam V \mid V \in \cV_{m}\}$. For each $V \in \cV_m$, we
  have
 \[ B(x_V,s) \subset V \subset B(x_V, D)\]
 so that if we put $C\coloneqq\max\{D/\theta^m, \theta^m/s \}$, then $C>1$, and 
 \[ B(x_V, C^{-1}\theta^m) \subset V \subset B(x_V, C\theta^m).\]
 The finite collection of pointed sets
 $\{(V, x_V) \mid V \in \cV_{m}\}$ then satisfies condition~(1) in
 Definition~\ref{defn:snapshots} on sequences of snapshots. By construction
 $C^{-1}\theta^m<s$ and it follows that condition~(2) holds as well.

Now suppose $n>m$ and $\widetilde{V} \in \cV_n$.  The restriction $f^{n-m}|\widetilde{V}$ is
  an expanding homothety with factor $\theta^{m-n}$ onto say $V \in
  \cV_m$. Put $x_{\widetilde{V}}\coloneqq(f^{n-m}|\widetilde{V})^{-1}(x_V)$.
  Using the same constant $C$ constructed in the previous paragraph,
  it follows that the sequence $(\cV_n)_{n \geq m}$ is a sequence of
  snapshots of $\cJ$ with parameter~$\theta$.
\end{proof}

For convenience, we reindex our virtual endomorphism and sequence of
covers $(\cV_n)_n$ so
that $m=0$, as in Eq.~\eqref{eq:reindex}

\subsection{Comparing modulus of curves on \textalt{$G_n$}{Gn} and on \textalt{$\cJ$}{J}}
\label{subsec:jg}

Suppose $\sigma^n_\infty\colon G_n \to \cJ$ is the homotopy section of $\phi^\infty_n$ given by Corollary \ref{cor:recurrrent_graph_ve}, and suppose $\gamma\co C \to G_n$ is a weighted multi-curve on $G_n$. 
Define the curve
$\gamma'\coloneqq\sigma^n_\infty \circ \gamma$ on~$\cJ$, and the curve
$\gamma''\coloneqq\phi^\infty_n\circ \gamma'$ on~$G_n$.
By Corollary \ref{cor:recurrrent_graph_ve}, there exists $K>0$ so that
\[ \gamma \sim_K \gamma''.\]
It is important that $K$ is independent of $n$. 
Though $\gamma$, $\gamma'$, and $\gamma''$ are single curves, they nevertheless have a
meaningful combinatorial modulus when regarded as a family; see
Remark~\ref{remark:singlecurve}.

Recall from \S\ref{sec:regularity} that,
on $G_n$, there are coverings $\cE_n$ by closed edges, and $\widehat\cE_n$ by
open stars.  Also, recall that $M$ is a uniform upper bound on the degrees of the graphs $G_n$, $n \in \mathbb{N}$.

\begin{lemma}
\label{lemma:eg_vs_ej}
For all $q \geq 1$ and any weighted multi-curve $\gamma$ on $G_n$, with $\gamma'$ and
$\gamma''$ as above we have 
\[ \modulus_q(\gamma', \cV_n) =\modulus_q(\gamma'', \widehat{\cE}_n)\asymp_M \modulus_q(\gamma'', \cE_n) \asymp_{K,M,q} \modulus_q(\gamma, \cE_n). \]
In particular, the implicit constants are independent of $n$ and~$\gamma$. 
\end{lemma}

Tracing through the dependencies, note that the constant $K$ depends on the Lipschitz constant $\lambda<1$ of $\phi$. 
 
\begin{proof}
  By definition the map $\phi^\infty_n$ sends $\gamma' $ to $\gamma''$ and
  induces an isomorphism between the nerve of~$\cV_n$ and
  the nerve of~$\widehat{\cE}_n$. The elements $V(e)$ are given by
  $(\phi^\infty_n)^{-1}(\whe)$; in particular they are saturated with
  respect to the fibers of $\phi_\infty^n$. Thus
  $\modulus_q(\gamma', \cV_n) = \modulus_q(\gamma'', \widehat{\cE}_n)$. The
  middle estimate is the content of Lemma \ref{lemma:stars}.
  For the
  last estimate, $\gamma$
  and $\gamma''$ are in uniform combinatorial
  $\cE_n$-neighborhoods of each other since $\gamma \sim_K \gamma''$.
  Then Lemma~\ref{lemma:regularity} (with its $K\coloneqq M$ and its $N\coloneqq K$)
  implies that
  $\modulus_q(\gamma'', \cE_n) \asymp_{K, M,q} \modulus_q(\gamma, \cE_n)$. 
\end{proof}

\subsection{Upper bound on \textalt{$\ARCdim$}{ARCdim}}
\label{subsec:leq}

Let $\delta_0$ be as in Theorem~\ref{thm:keith-kleiner}, and fix
$0<\delta<\delta_0$.

\begin{proposition}
\label{proposition:to_zero}
Suppose that $q>1$ and $\oE^q[\pi, \phi]<1$. Let $\Gamma_\delta$ be the family of curves in $\cJ$ whose diameters are bounded below by $\delta$.  Then $\modulus_q(\Gamma_\delta, \cV_n) \to 0$ as $n \to \infty$.
\end{proposition}

\begin{proof}
  We
  begin with the following claim. There is an integer $N$ such that
  for each $\gamma \in \Gamma_\delta$ there is $e \in E(G_N)$ so that
  $e \subset \phi^\infty_N(\gamma)$. To see this, we argue as follows.
  Using Lemma \ref{lem:V}, choose~$N$ so that the mesh of $\cV_N$ is
  at most $\delta/10$. Pick $x^1, x^2 \in \gamma$ with
  $d_{vis}(x^1, x^2)>\delta/2$. Choose closed edges $e_1, e_2 \in E(G_N)$
  such that $\phi^\infty_N(x_i) \in e_i$. The triangle inequality
  shows that $e_1 \cap e_2 = \emptyset$. The image
  $\phi^\infty_N(\gamma)$ is a connected set in $G_N$ which meets two
  disjoint closed edges, so must contain an entire edge~$e$.

  Let the level $N$ be as in the previous paragraph. 
  For $e \in E(G_N)$, let $\Gamma_e$ be the family of
  curves~$\gamma$ in $\cJ$ so that $e \subset \phi^\infty_N(\gamma)$. Then
  \begin{align*}
    \Gamma_\delta &\subset\! \bigcup_{e \in E(G_{N})}\!\Gamma_e\\
    \intertext{and so, for $n \ge N$,}
    \modulus_q(\Gamma_\delta, \cV_n) &\leq\! \sum_{e \in E(G_{N})}\!
                                       \modulus_q(\Gamma_e, \cV_n).
  \end{align*}
  It therefore suffices to show that
  $\modulus_q(\Gamma_{e}, \cV_n) \to 0$ for each $e \in E(G_{N})$.

  The proposition deals with the asymptotic energy $\oE^q[\pi,\phi]$
  and the limit set~$\cJ$.
  Neither the property that $\oE^q[\pi,\phi]<1$ nor the conformal gauge of the limit space is changed by reindexing,
  changing the weights, or iterating.
  We may therefore  choose
  $q$-lengths $\alpha\equiv 1$ on $G_0$ (and thus on~$G_N$), and then
  reindex to  assume $N=0$.  We now have a $q$-conformal structure $G_0^q$ on $G_0$. By lifting under the coverings $\pi^n_0$, we obtain $q$-conformal structures $G_n^q$ on $G_n$ for each $n$.  Since
  $\oE^q[\pi, \phi]<1$, by iterating we may assume $\phi\colon G^q_1 \to G^q_0$ satisfies 
  $E^q_q(\phi) = \lambda < 1$ and so $\phi^n_0\colon G^q_n \to G^q_0$ satisfies $E_q^q(\phi^n) \leq \lambda^n$ for
  all $n > 0$.

Fix some $e_0 \in E(G_0)$. We must show $\modulus_q(\Gamma_{e_0}, \cV_n) \to 0$ as $n \to \infty$. 
Fix $n \in \mathbb{N}$.   Recall elements of $\cV_n$ are preimages of
stars $(\phi^\infty_n)^{-1}(\widehat{s})$.  For brevity, we will
write $\rho(s)$ instead of $\rho((\phi^{\infty}_n)^{-1}(\widehat{s}))$. Let $x$ be a local length coordinate on the edge $s$. Define a test metric $\rho\co \cV_n
\to [0,\infty)$ by 
\begin{equation}
\label{eqn:rho}
\rho(s)=\int_{x \in s}\lvert(\phi^n)'(x)\rvert |dx|.
\end{equation}
In other words, $\rho(s)$ is the length of $\phi^n(s)$
regarded as a curve on the length-graph underlying~$G_0$, which we
denote $\ell_{G_0}(\phi^n(s))$. (Note that the image 
$\phi^n(s)$ may backtrack.)

We now show that this $\rho$ is admissible for $\Gamma_{e_0}$.
Pick $\gamma \in \Gamma_{e_0}$, and set $\beta\coloneqq\phi^\infty_n(\gamma)$; it is a curve in $G_n$. We have  
\[ e_0 \subset \phi^\infty_0(\gamma) = \phi^n_0(\phi^\infty_n(\gamma))=\phi_0^n(\beta).\]

Let $s_1, \ldots, s_m$ be the edges in $E(G_n)$ met by~$\beta$. Then 
\[ \ell_\rho(\gamma, \cV_n) = \sum_{j=1}^m\rho(s_j) =
  \sum_{j=1}^m \ell_{G_0}(\phi^n(s_j))\geq \ell_{G_0}(e_0)=1.
\]
The last inequality works even in the presence of backtracking: the
total length
in~$G_0$ covered by the $\phi^n(s_j)$ is at least as long as $e_0$,
even if $\beta$ crosses over a given edge $s_j$ multiple times.

Next we estimate the $q$-volume $V_q(\rho, \cV_n)$; for Jensen's Inequality, see, e.g., \cite[Theorem 3.3]{MR924157}.  We have
\begin{align*}
  V_q(\rho, \cV_n)
  &= \sum_{s\in\Edges(G_n)} \rho(s)^q
    &&\text{Definition of $V_q$}\\
  &=\sum_s \left( \int_{x\in s} \lvert (\phi^n)'(x)\rvert\; dx \right)^q
    &&\text{Definition of $\rho$}\\
 &\leq \sum_s \int_{x\in s} \lvert(\phi^n)'(x)\rvert^q \, dx 
    &&\text{Jensen's inequality}\\
&=\int_{x\in G_n}\lvert(\phi^n)'(x)\rvert^q \, dx
    &&\\
  &=\int_{x \in G_0}\Fill^q(\phi^n)\, dx
    &&\text{Change of variables from $G_n$ to $G_0$}\\
&\leq \norm{\Fill^q(\phi^n)}_{\infty, G_0} \cdot \left(\int_{x\in G_0}1\, dx\right)\\
&\asymp \left(E_q^q(\phi^n)\right)^q && \text{Definition of $E_q^q$}\\
& < \lambda^{nq}.
\end{align*}
We conclude $\modulus_q(\Gamma_{e_0}, \cV_n) \to 0$, as required.  
\end{proof}

Proposition~\ref{proposition:to_zero} and
Theorem~\ref{thm:keith-kleiner} then imply
$\ARCdim(\cJ) \leq p^*[\pi, \phi]$.

\begin{remark}
  The calculations in the proof of
  Proposition~\ref{proposition:to_zero} can be interpreted as follows.
  Let $K^\infty_0$ be $G_0$, considered as a length graph with all edge
  lengths~$1$. Let $\psi^n \co G^q_n \to K^\infty_0$ be $\phi^n$, homotoped to
  be constant-derivative on each edge of~$G^q_n$. Then, if you trace
  through the definitions, $V_q(\rho,\cV_n)=\bigl(E^q_\infty(\psi^n)\bigr)^q$,
  and the relevant inequalities are
  \[
    E^q_\infty(\psi^n) \le E^q_\infty(\phi^n)
    \le E^q_q(\phi^n) \cdot E^q_\infty(\textrm{id}\co G^q_0 \to K^\infty_0) <
    \lambda^n \cdot E^q_\infty(\textrm{id}\co G^q_0 \to K^\infty_0) \asymp \lambda^n.
  \]
\end{remark}

\subsection{Lower bound on \textalt{$\ARCdim$}{ARCdim}}
\label{subsec:geq}

For this inequality, we will use more of the theory from
\S\ref{sec:modulus}.  Again let $\delta_0$ be as in Theorem
\ref{thm:keith-kleiner}, and fix $\delta<\delta_0$ that is
sufficiently small (to be specified).

\begin{proposition}
\label{proposition:to_infinity}
Suppose $\oE^q[\pi, \phi] > 1$. Then $\modulus_q(\Gamma^w_\delta, \cV_n) \to \infty$.
\end{proposition}

\begin{proof} 
  Set $\mu = \oE^q[\pi,\phi] > 1$; then
  $E^q_q[\phi^n] \ge \mu^n$ \cite[Proposition 5.6]{Thurston20:Characterize}.
  We are going to find a normalized weighted multi-curve $\zeta'$ on $\cJ$
  in $\Gamma_\delta^w$ such that
  $\modulus_q(\zeta', \cV_n) \to \infty$ as $n \to \infty$. This will
  imply $\modulus_q(\Gamma^w_\delta, \cV_n) \to \infty$ as required.
  
  Fix $n \in \mathbb{N}$. Recall from \S \ref{sec:regularity} that the graph $G_n$ is naturally covered by its set $\cE_n$ of closed edges.  By Theorem~\ref{thm:sf}, we can find a curve
  exhibiting $E^q_q[\phi^n]$:
  there is a reduced weighted multi-curve $\gamma\co C^1 \to G^q_n$ and a map
  $\psi\co G^q_n \to G^q_0$ that minimizes $E^q_q$
  in the homotopy class $[\phi^n]$ and fits into
  a tight sequence
\begin{equation}
\label{equation:tight}
C^1\overset{\gamma}{\longrightarrow} G^q_n \overset{\psi}{\longrightarrow}G^q_0.
\end{equation}
(Since the exponent $q$ in $G_n^q$ is fixed, we will suppress it
from the notation.) That is,
\begin{equation}
\label{equation:blowup}
\mu^n \le 
  E^q_q(\psi) = \frac{E^1_q(\psi \circ \gamma)}{E^1_q(\gamma)}= \frac{E^1_q[\phi^n \circ \gamma]}{E^1_q[\gamma]}.
\end{equation}
This goes to infinity
as $n \to \infty$. Furthermore, Theorem \ref{thm:sf} guarantees that
for each strand~$J_i$ of~$C$, the restriction $\gamma|J_i$ has image
covering a given edge
of $G_n$ at most twice, so that Proposition~\ref{prop:energies-comp}
applies.

The weighted multi-curve $\gamma\co C^1 \to G_n$ is not unique. In particular,
we can and will scale the weights on~$C^1$ so that $E^1_q[\psi \circ \gamma]=1$.
Then $1/E_q^1[\gamma] \ge \mu^n$, and 
Proposition~\ref{prop:energies-comp}
implies that $\modulus_q([\gamma], \cE_n) \gtrsim \mu^{nq}$. 

We next apply the constructions in \S \ref{subsec:jg} to obtain curves $\gamma'\colon C \to \cJ$ and $\gamma''\colon C \to G_n$ with $\gamma''\coloneqq\sigma^n_\infty \circ \gamma'$ and $\gamma\sim_K\gamma''$ for some constant $K>0$ independent of $n$.
Lemma~\ref{lemma:eg_vs_ej} implies
\[ \modulus_q(\gamma', \cV_n) =\modulus_q(\gamma'', \cE_n) \asymp \modulus_q(\gamma, \cE_n)\gtrsim \mu^{nq}.\]
Though the diameters of the strands of $\gamma'$ are bounded from below and
the modulus is blowing up, we do not yet have control on the weights
of $\gamma'$; usually, $\gamma'$ will not be in~$\Gamma_\delta^w$. (Correspondingly,
the strands in $\gamma'$ are much longer than any constant~$\delta$.)
We will remedy
this by subdividing $\gamma'$ to obtain a more suitable curve, as in \S \ref{sec:subdivision}.

The curve $\psi \circ \gamma\co C \to G_0$ is reduced. The strands
$\psi \circ (\gamma|J_i)\co J_i \to G_0$ may run many times over a given
edge of~$G_0$. We decompose $\psi \circ \gamma$ into separate pieces, one for each
time such a strand runs over an edge of $G_0$. Formally: for a fixed
pair $(i,e)$ consisting of a strand $\psi \circ (\gamma|J_i)$ and an edge
$e$ of $G_0$, suppose the restriction $\psi \circ (\gamma|J_i)$ runs
$N(i,e)$ times over $e$. We obtain a collection of sub-intervals
$I_{i,e,k}$ of $J_i$ for which $\psi \circ \gamma\co I_{i,e,k} \to e$ is a
homeomorphism. (Note the $I_{i,e,k}$ are disjoint except at their
endpoints, and since each strand $J_i$ of~$c$ is a circle there are no end
issues to worry about.)
We identify each $I_{i,e,k}$ with the unit interval
and obtain a weighted multi-curve $\zeta\co D \to G_n$ as follows. With
$w_i$ the weight of $J_i$ in~$C$, set
\[ D \coloneqq \bigsqcup_i\, \bigsqcup_e \bigsqcup_{k=1}^{N(i,e)} w_i I_{i,e,k}.\]
Let $\iota \co D \to C$ be the natural inclusion of
intervals, and define $\zeta \coloneqq \gamma \circ \iota \co D^1 \to G_n$. By
construction, the
function $n_{\psi \circ \zeta}\co G_0 \to \mathbb{R}^+$ is constant on edges
of $G_0$ and coincides with $n_{\psi \circ \gamma}$.

Recall we have normalized so $E_q^1[\psi \circ \gamma]=1$, so with
$\qdual$ the H\"older conjugate of~$q$, 
\[1 = E_q^1[\psi \circ \gamma]=\norm{n_{\psi \circ \gamma}}_{\qdual} \asymp_{q, \#E(G_0)} \norm{n_{\psi \circ \gamma}}_1=\norm{n_{\psi \circ \zeta}}_1=\sum_{i,e,k} w_{i,e,k}.\]
At the third step, we use the fact that in $\mathbb{R}^{E(G_0)}$, any two norms are comparable.
We conclude: the sum of the weights of $D$ is comparable to $1$. 

We now decompose the curve $\gamma' \co C \to \cJ$ with the same decomposition. Let
$\zeta' = \gamma' \circ \iota \co D \to \cJ$. We must show that the size of each strand of $\zeta'$ has diameter bounded below, independent of~$n$. We focus
attention on one component $I_{i,e,k}$ of~$D$ and identify that
interval with $[0,1]$. Let
$\zeta'' = \phi^\infty_n \circ \zeta' \co D \to G_n$. Since $\gamma \sim_K \gamma''$ and the decompositions $\zeta$ and $\zeta''$ correspond, the endpoints $\zeta''(0)$ and $\zeta''(1)$ are
within a uniformly bounded $G_n$-distance $K$ of $\zeta(0)$ and $\zeta(1)$
(independent of~$n$). Also recall that we assumed that the system is 
forward-expanding, so $\phi$ is Lipschitz with some constant $\lambda<1$ and
$\phi^n$ is Lipschitz with constant $\lambda^n$.
Then
\begin{align*}
  \abs{\phi^\infty(\zeta'(0))- \phi^\infty(\zeta'(1))}
    &= \abs{\phi^n(\zeta''(0))- \phi^n(\zeta''(1))} \\
    &\ge  \abs{\phi^n(\zeta(0))-\phi^n(\zeta(1))}-\abs{\phi^n(\zeta''(0))- \phi^n(\zeta(0))}\\
           &\qquad- \abs{\phi^n(\zeta(1))- \phi^n(\zeta''(1))} \\
    &\ge 1-2\lambda^{-n}K.
\end{align*}
We suppose that $n$ is large enough so that $1 - 2\lambda^{-n}K$ is
bigger than $1/2$. We have shown that each strand of $\zeta'$, when projected to $G_0$, has diameter at least $1/2$. Since $\phi^\infty$ is uniformly continuous (as a
function from a compact metric space), it follows that each strand of
$d'$ has definite diameter; we choose~$\delta$ smaller than this
diameter. Combining this with
the observation that the sum of the weights of~$D$ is comparable
to~$1$, we conclude that, after an innocuous rescaling of weights, $\zeta'$ lies in
$\Gamma_\delta^w$.

We then have
\begin{align*}
\mu^{nq} &\le \frac{1}{(E^1_q[\gamma])^q}\\
  &\lesssim \modulus_q([\gamma], \cE_n) &&\text{by Proposition~\ref{prop:energies-comp}}\\
  &= \modulus_q(\gamma,\cE_n)&&\text{since $\gamma$ is reduced} \\
  &\asymp \modulus_q(\zeta,\cE_n) &&\text{by Lemma~\ref{lem:subdivide}, subdivision}\\
  &\asymp \modulus_q(\zeta'',\cE_n) &&\text{by Lemma~\ref{lemma:eg_vs_ej}, fellow travellers}\\
  &= \modulus_q(\zeta',\cV_n) &&\text{by definition of the cover $\cV$.}
\end{align*}
For the conditions of Lemma~\ref{lem:subdivide} at the fourth step,
note that the $G_0$-length of each strand of
$\phi^n\circ \zeta$ is $1$, so by forward-expansion the $G_n$-length of
each strand is at least $\lambda^n$, and in particular for large $n$
the image of each strand is not contained in a single edge of~$G_n$.

This completes the proof of Proposition \ref{proposition:to_infinity}.
\end{proof} 

Proposition~\ref{proposition:to_infinity} and Proposition \ref{proposition:to_zero} complete the proof of
Theorem~\ref{thm:crit-sandwich}.


\section{Applications}
\label{sec:applications}

We turn now to applications. \S \ref{sec:Nbar_gtr_1} gives the proof
of Theorem~\ref{thm:Nbar_ge_1}, on Sierpiński carpets. \S
\ref{sec:barycentric_carpet} proves the estimates for the barycentric
subdivision example mentioned in the introduction. \S\S
\ref{sec:fat-devaney} and \ref{sec:skinny-devaney} give the estimates
for the fat and skinny Devaney examples. The brief \S
\ref{subsecn:uwscp} introduces Carrasco's \emph{uniformly well-spread cut point}
(UWSCP) condition. \S \ref{sec:mating-examp} applies our methods to
examples obtained by the operation of ``mating'' and concludes with a
question about the relationship between the UWSCP condition and other
properties.

For some estimates, we rely on explicit bounds on how fast
$\oE^q$ can decrease as a function of~$q$
\cite[Proposition 6.11]{Thurston20:Characterize}:
\begin{proposition}\label{prop:Eq-lower-bound}
  For $\pi, \phi \co G_1 \rightrightarrows G_0$ a virtual
  endomorphism of graphs of degree~$d\coloneqq\deg(\pi)$, if
  $1 \le p \le q \le \infty$, then
  \[
  \oE^q[\pi,\phi] \ge d^{-\frac{1}{p}+\frac{1}{q}} \cdot\oE^p[\pi,\phi].
  \]
\end{proposition}

As an easy consequence, we have the following theorem announced in the
introduction.  Recall that the quantity $\oN\coloneqq\oE^1_1$ (Definition \ref{defn:ae}) counts the asymptotic growth rate, as $n \to \infty$, of the essential number of preimages of $\phi^n_0$, minimized among maps within its homotopy class.
\begin{taggedthm}{\ref{thm:Nbar}}
For any recurrent expanding virtual graph endomorphism $[\pi,\phi]$ where
$\deg(\pi) = d$, we have
\[ \ARCdim[\pi, \phi] \geq \frac{1}{1-\log_d \oN[\pi, \phi]}.\]
\end{taggedthm}

\begin{proof}
  By Proposition~\ref{prop:Eq-lower-bound} with $p = 1$, for any~$q$ we have
  \[
    \oE^q[\pi,\phi] \ge d^{\frac{1}{q} - 1}\, \oN[\pi,\phi].
  \]
  If $q$ is less than the quantity in the theorem statement, the right-hand side
  is greater than~$1$. The result follows from
  Theorem~\ref{thm:crit-sandwich}.
\end{proof}

\subsection{Cases when \textalt{$\overline{N} > 1$}{Nbar > 1}}
\label{sec:Nbar_gtr_1}

Our proof of Theorem~\ref{thm:Nbar_ge_1} is a corollary of a more general result about certain expanding dynamical systems on the sphere:

\begin{taggedthm}{C\/$'$}
\label{taggedthm:Cprime}
Suppose $g\colon S^2 \to S^2$ is an expanding Thurston map such that each cycle in its post-critical set $P_g$ contains a critical point. If $\pi, \phi\colon G_1 \to G_0$ is any virtual graph endomorphism induced by a choice of a spine $G_0$ for $S^2-P_g$, then $\oN[\pi, \phi]>1$. 
\end{taggedthm}

A \emph{Thurston map} $g\colon S^2 \to S^2$ is an orientation-preserving branched self-cover of degree at least two such that the post-critical set $P_g\coloneqq\cup_{n>0}\, g^n(\{\text{branch points}(g)\})$ is finite. The definition of \emph{expanding} Thurston map  appearing in the hypothesis of Theorem \ref{taggedthm:Cprime} here is that of Bonk and Meyer  \cite{BM17:Expanding}.

Here is an example of an expanding Thurston map.  First, recall from \S \ref{subsec:motivation} and Figure \ref{fig:barycentric_ve} that the rational map $f$ there sends the small triangles conformally onto the large triangles.  Next, consider the Thurston map $g$ obtained from Figure \ref{fig:barycentric_ve} where now the small triangles are sent not conformally, but Euclidean-planar-affinely, to the large triangles. While the maps $f$ and $g$ are conjugate-up-to-isotopy relative to their post-critical sets, they are not topologically conjugate, since the fixed branch-point of $g$ is locally topologically repelling, whereas that of $f$ is attracting. Recall that the Julia set $J_f$ is a Sierpiński carpet. It turns out (see below) that collapsing the Fatou components of $f$ to points yields a map which is topologically conjugate to $g$.  Our proof of Theorem \ref{thm:Nbar_ge_1} proceeds by passing to such a quotient and invoking Theorem \ref{taggedthm:Cprime}.

The next few paragraphs summarize several results from \cite{BM17:Expanding} that we use in the proof.

Fix arbitrarily a metric on $S^2$ compatible with its topology. Suppose $g$ is an arbitrary Thurston map, and $\mathcal{C}\subset S^2$ is a Jordan curve containing $P_g$. This data induces a cell structure $\mathcal{T}_0$ on $S^2$, with two open $2$-cells, called \emph{tiles}, given by the components of $S^2-\mathcal{C}$. Lifting by iterates of $g$ yields a sequence of cell structures $\mathcal{T}_n$ on $S^2$ for $n=0, 1, 2, \ldots$; let $\text{mesh}(g, n, \mathcal{C})$ be the maximum diameter of an open $2$-cell, or \emph{$n$-tile}, at level $n$.  The Thurston map $g$ is said to be \emph{expanding} if $\text{mesh}(g, n, \mathcal{C}) \to 0$ as $n \to \infty$; this property is independent of the choice of metric and of $\mathcal{C}$; see  \cite[Ch. 6]{ BM17:Expanding}. There exists an iterate $k$ and a Jordan curve $\mathcal{C}'$ isotopic to $\mathcal{C}$ relative to $P_g$ such that $g^k(\mathcal{C}')\subset \mathcal{C}'$; see \cite[Thm. 15.1]{ BM17:Expanding}.  

In our applications, passing to such an iterate will be innocuous, so in this paragraph we now suppose that $g$ is an expanding Thurston map, and that $g(\mathcal{C})\subset \mathcal{C}$ for some Jordan curve $\mathcal{C}\supset P_g$. Then for each $n$, the tiling $\mathcal{T}_{n+1}$ refines the tiling $\mathcal{T}_n$ according to a subdivision rule \cite[Ch. 12]{BM17:Expanding}.  A numerical invariant is then the \emph{combinatorial expansion factor} 
\[ \Lambda_0(g)\coloneqq\lim_{n \to \infty} D_n(g,\mathcal{C})^{1/n}\]
where $D_n(g, \mathcal{C})$ is, roughly speaking, the minimum number $m$ of closed $n$-tiles in a
chain $t_0, t_1, \ldots, t_m$ with $t_j \cap t_{j+1} \neq \emptyset$, $j=0, \ldots, m-1$ such that $t_0, t_m$ each contain $1$-cells of $\mathcal{T}_0$ whose closures are disjoint, i.e. the chain  ``joins disjoint closed edges'' of $\mathcal{C}$. (A slight modification is needed in the case when
$\#P_g=3$
\cite[Sec.\ 5.7]{BM17:Expanding}.) We then have that $\Lambda_0(g)$ is independent
of~$\mathcal{C}$ and is greater than~$1$ \cite[Prop.\
16.1]{BM17:Expanding}. For each $1<\theta^{-1}<\Lambda_0(g)$, there
exists an analogously defined visual metric on $S^2$ with expansion factor $\theta^{-1}$
\cite[Thm.\ 16.3(ii)]{BM17:Expanding}. Fixing such a choice of metric with its expansion factor $\theta^{-1}$, there is a constant $K>1$ such that
the diameter of each tile~$t$ at level $n$ satisfies
$\mathrm{diam}(t) \in [\theta^n/K, K\theta^n]$
\cite[Prop.\ 8.4(ii)]{BM17:Expanding};
compare our Theorem \ref{thm:visual}(1).

\begin{proof}[Proof of Theorem~\ref{taggedthm:Cprime}]
Let $g$ be an expanding Thurston map.  The property $\oN>1$ is unchanged under passing to iterates, so we assume there is a Jordan curve $\mathcal{C}\supset P_g$ for which $g(\mathcal{C}) \subset \mathcal{C}$.  Let $\Lambda_0(g)$ be the resulting combinatorial expansion factor.  We choose a visual metric $d_{vis}$, and let $\theta^{-1}$ be its expansion factor. 

In the remainder of the proof, we show $\theta^{-1} \leq \oN$. Since
$\theta^{-1}$ can be chosen arbitrarily subject to the constraint $\theta^{-1}\leq \Lambda_0(g)$ can be arbitrary, this is enough to conclude that $\oN \geq \Lambda_0(g)>1$. (See Remark \ref{remark:nbar_versus_lambda}.)

Let $G_0 \subset S^2$ be a realization of the dual of~$\mathcal{T}_0$;
it is a spine for $S^2-P_g$ with two vertices. As usual let $G_n=g^{-n}(G_0)$,
so that $G_n$ is the dual of~$\mathcal{T}_n$. Note that the mesh of
the faces of~$G_n$ tends to zero as well with respect to~$d_{vis}$. Fix
$p \in P_g$. Let $U_0$ be the component of the complement of~$G_0$
containing~$p$ and let $C_0=\partial U_0$. Since $G_0$ is a spine for
$S^2-P_g$, any loop in~$G_0$ that is freely homotopic to a peripheral loop
about $p$ contains~$C_0$.

\begin{claim}\label{claim:disjoint-loops}
  There exists $c>0$ such that for all $n \in \mathbb{N}$ there exist
  at least $c\theta^{-n}$ pairwise disjoint loops $C_{n, i}$ in $G_n$
  freely homotopic to a peripheral loop about $p$.
\end{claim}

Assuming the claim, we show $\theta^{-1} \leq \oN$ as follows. For any $\psi \in
[\phi^n]$, the curve
$\psi(C_{n,i})$ is freely homotopic to $C_0$ and so
contains~$C_0$. Thus any $y \in C_0$ has, for each $n$ and $i$, a
preimage in~$C_{n,i}$ and so $N[\phi^n] \gtrsim \theta^{-n}$, proving
$\oN(f) \geq \theta^{-1}$.

We now turn to the proof of Claim~\ref{claim:disjoint-loops}, using the fact that $g$ is
expanding on the whole sphere. Let $D=\min\{\,d_{vis}(p, q) \mid q \in
C_0)\,\}$, let $N_0 \in \mathbb{N}$ be chosen so $D\theta^{-n}/K>1$ for
$n \geq N_0$, and for $n \geq N_0$ let $m_n$ be greatest positive
integer less than or equal to $D\theta^{-n}/K$. Then for all such $n$,
any chain of tiles $t_1, \ldots, t_{m_n}$ in $\mathcal{T}_n$ with
$p \in t_1$ and $t_i \cap t_{i+1} \neq \emptyset$ for
$i=1, \ldots, m_n-1$ avoids $C_0$.

For $n \geq N_0$, we will define curves $C_{n,1}, \ldots, C_{n, m_n}$.
Given $E \subset S^2$ and $n \in \mathbb{N}$,
recall that
$\mathcal{T}_n(E)$ is the union of the closed tiles at level $n$
meeting~$E$.
Let $E_{n,1}=\{p\}$ and for $i=1, \ldots, m_n-1$ inductively set
$E_{n, i+1}\coloneqq \mathcal{T}_n(E_{n, i})$. Fix one such~$i$.
Then $E_{n,i}$ is contained in the interior of $E_{n, i+1}$. The
complement $E_{n,i+1} - \mathop{\mathrm{interior}}(E_{n,i})$ is tiled by
elements of $\mathcal{T}_n$. Consider the corresponding
subgraph of the dual graph $G_n$. We take $C_{n,i}$ to be a simple
cycle in this dual graph that separates $C_0$ from~$p$. The
$C_{n,i}$ are pairwise disjoint, freely homotopic to~$C_0$, and
disjoint from~$C_0$, by construction. This proves
Claim~\ref{claim:disjoint-loops} with
$c$ approximately $D/K$.  Theorem \ref{taggedthm:Cprime} is proved. 
\end{proof}

\begin{proof}[Proof of Theorem~\ref{thm:Nbar_ge_1}]
  Suppose $f$ is a hyperbolic, critically-finite map with
  carpet Julia set and post-critical set $P_f$.  

  By Moore's Theorem, since the Julia set is a Sierpiński carpet,
  the quotient space obtained by collapsing the closures of Fatou
  components of $f$ to points is a sphere; we denote the resulting
  projection by $\rho\co
  \widehat{\mathbb{C}} \to S^2$.  Then
  $\rho$ gives a semiconjugacy to an induced map
  $g\co S^2 \to S^2$.  It is shown in  \cite[Thm.\ 5.1]{GHMZ18:InvJordanCurves} that $g$ is an expanding Thurston map.

The quotient map~$\rho$ is uniformly approximable by a continuous
family of homeomorphisms. Taking any member of this family yields a
homeomorphism $h\co (\widehat{\mathbb{C}}, P_f) \to (S^2, P_g)$,
well-defined up
to isotopy relative to $P_f$.  The map $h$ induces a
conjugacy-up-to-isotopy from $f$ to~$g$, so, as in
Remark~\ref{rem:invariance}, $\oN(g) = \oN(f)$. By Theorem~\ref{taggedthm:Cprime}, we have $\oN(f)=\oN(g)>1$. 
\end{proof}

\begin{remark}
\label{remark:nbar_versus_lambda}
  The major difference between the quantities $\oN(f)$ and
  $\Lambda_0(f)$ introduced in the proof of Theorem
  \ref{thm:Nbar_ge_1} is that the former represents a
  maximum growth rate while the latter represents a minimum growth rate. 
\end{remark}

The converse to Theorem~\ref{thm:Nbar_ge_1} need not hold; there are
many examples.
We sketch two constructions.
Begin with the
quadratic carpet example of Milnor and Tan \cite[Appendix~F]{Milnor93:GeomDynQuadratic},
\[
  f(z) \approx -0.138115091\left(z + \frac{1}{z}\right) -0.303108805.
\]
The two critical
points have periods 3 and 4.
The thesis of the first author
\cite[Theorem~7.1]{Pilgrim94:Cylinders} shows that there exists a
rational map $g$ which
combinatorially is the ``tuning'' of $f$ and the basilica polynomial along
the period 3 critical orbit.  The result of tuning replaces each component of the basin of the superattracting 3-cycle with a copy of the basilica Julia set. The Julia set of $g$ is easily seen to
have two Fatou components whose closures meet, corresponding to the
immediate attracting basins of the basilica, so it is not a carpet.
Insung Park \cite[Theorem 2]{Park21:Obstructions} has proved that the
energies $\oE^p$ do not decrease under tuning, so in particular
$\oN(g) \ge \oN(f) > 1$.

One may easily construct other examples that have not just local cut points, as in the preceding case, but global cut points. The quartic rational map 
\begin{equation}\label{eq:quartic-non-carpet}
  f(z)\approx -\frac{z^3(z+1)}{(z+0.3309124475)^3(z+0.0072626575)},
\end{equation}
with Julia set shown in Figure~\ref{fig:noncarpet},
has critical points at $p, c, 0, \infty$ with orbits
\begin{gather*}
  p \overset{3}{\longrightarrow}\infty
  \overset{2}{\longrightarrow}-1\overset{1}{\longrightarrow}0 \overset{3}{\righttoleftarrow} \\
c \overset{2}{\righttoleftarrow}
\end{gather*}
where $-1 < p < c < 0$; the weights on the arrows show the local
degree. This example is obtained from taking the torus automorphism
$z \mapsto (1+i)z$ on the torus $\CC / (\ZZ \oplus i \ZZ)$ by
\begin{itemize}
\item taking its $\ZZ/4\ZZ$-quotient to get the Lattès map $g(z) =
-z(z+1)/(z+1/2)^2$, with critical orbits $-1/2\overset{2}{\longrightarrow}\infty \overset{2}{\longrightarrow}-1 \longrightarrow 0 \righttoleftarrow$;
\item blowing up an arc from the unique fixed post-critical point at $0$ to
  the critical point at~$-1/2$ to get a degree-three map~$h$ with
  carpet Julia set; and then 
\item blowing up an arc from the unique fixed super-attracting post-critical point of $h$ to
  the repelling
  fixed point of $h$ on the boundary of its basin to get the degree-four
  map $f$ above.
\end{itemize}

Direct calculation shows that the
following condition is satisfied, as shown in
Figure~\ref{fig:noncarpet}: there is a curve
$\alpha\co [0,1] \to \widehat{\mathbb{C}}$, with
$\alpha(0)=\alpha(1)=0$, and $\alpha(t) \not\in P_f$ for $0<t<1$; the
image of~$\alpha$ is an embedded loop symmetric with
respect to the real axis; the bounded component of the complement of
the image of~$\alpha$ contains $c$ and no other points of~$P_f$; some lift
$\widetilde{\alpha}$ of~$\alpha$ under~$f$ is homotopic to $\alpha$
through curves with the same properties.  
From \cite[Theorem~5.14]{Pilgrim94:Cylinders} it
follows that
the boundary of the immediate basin of the origin is not a Jordan curve,
and hence that $J_f$ is not a Sierpiński carpet.
\begin{figure}
  \[
    \centerline{\includegraphics[width=2.75in]{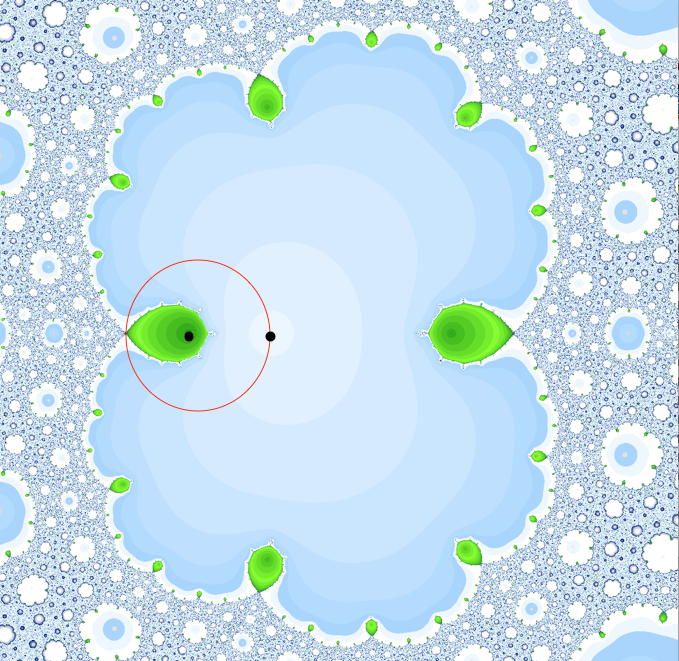}}
  \]
  \caption{Detail of Julia set for the non-carpet example $f$ in
    Eq.~\eqref{eq:quartic-non-carpet}. The post-critical points
    $c$ and $0$ are marked in black, and the curve~$\alpha$ is shown in red.}
  \label{fig:noncarpet}
\end{figure}

With a bit of work, one can show that $\oN(f) = \sqrt{2} >1$. 

\subsection{Barycentric subdivision Julia set}
\label{sec:barycentric_carpet}
We give details for the estimates for the conformal dimension of the Julia set of the map  $f(z)=\frac{4}{27}\frac{(z^2-z+1)^3}{(z(z-1))^2}$ from the introduction.

The dynamics on the set of critical points is shown below, with  $\omega = \exp(2\pi i /6)$:
\[
    \begin{tikzpicture}[y=1.6cm,x=1.5cm]
      \node (infty) at (2,0) {$\infty$};
      \node (0) at (1,0.7) {$0$};
      \node (1) at (1,-0.7) {$1$};
      \node (omega) at (0,0.85) {$\omega$};
      \node (omegainv) at (0,0.45) {$\omega^{-1}$};
      \node (-1) at (0,-0.2) {$-1$};
      \node (half) at (0,-0.7) {$\frac{1}{2}$};
      \node (2) at (0,-1.2) {$2$};
      \draw[->,loop right] (infty) to node[right,cdlabel]{2} (infty);
      \draw[->] (0) to node[above right=-2pt,cdlabel]{2} (infty);
      \draw[->] (1) to node[below right=-2pt,cdlabel]{2} (infty);
      \draw[->] (omega) to node[above=-1pt,cdlabel]{3} (0);
      \draw[->] (omegainv) to node[below=-1pt,cdlabel]{3} (0);
      \draw[->] (-1) to node[above right=-2pt,cdlabel]{2} (1);
      \draw[->] (half) to node[above=-2pt,pos=0.3,cdlabel]{2} (1);
      \draw[->] (2) to node[below right=-2pt,cdlabel]{2} (1);
    \end{tikzpicture}
  \]
\begin{figure}
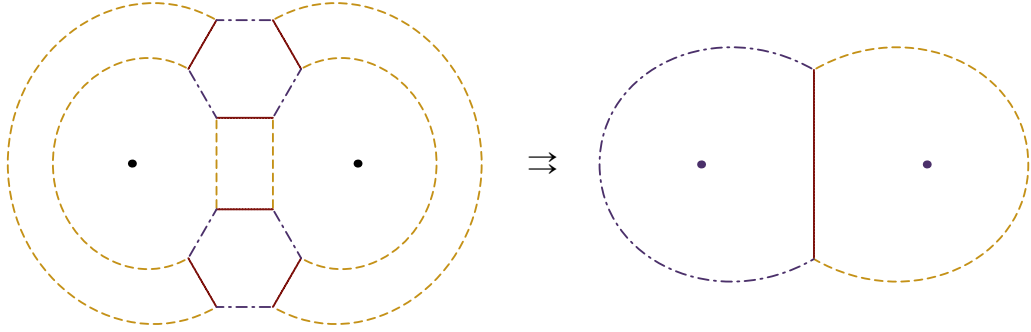

  \[
    \mfig{barycenter-1}\quad\rightrightarrows\quad\mfig{barycenter-0}
  \]
  \caption{Spines for the barycentric subdivision rational map}
  \label{fig:barycenter-G1}
\end{figure}
The map $f$ may be described as follows. Consider the
spherical ``triangle'' $T_0$ (the upper half-plane) with vertices
$(\infty,0,1)$ and angles $(\pi,\pi,\pi)$. If we give $\CCa$ the
spherical metric to make this  triangle equilateral, then
the spherical triangle $T_1$ with vertices $(\infty, \omega, 2)$
has corresponding angles
$(\pi/2, \pi/3, \pi/2)$. There is a unique
conformal map sending $T_1$ to~$T_0$ and mapping the vertices
$(\infty, \omega, 2) \mapsto (\infty, 0, 1)$. Twelve
reflected images of $T_1$ tile the sphere, and the map $f$ is
the unique extension given by Schwarz reflection. The extended real
axis is forward-invariant, and its preimage under $f$ divides the
sphere into twelve small triangles, 6 in each of the upper- and
lower-half planes. This induces a \emph{finite subdivision rule} on
the sphere, in the sense of \cite{CFP01:FiniteSubdiv}. We equip the
codomain with a cell structure with $0$-cells at $0, 1, \infty$ and
$1$-cells the corresponding segments of the extended real axis. Taking
inverse images then refines each $2$-cell in a pattern that,
combinatorially, effects \emph{barycentric subdivision}; see Figure
\ref{fig:barycentric_ve}. The function $f$ is a Galois branched
covering map; with deck group isomorphic to $S_3$ and acting by
spherical isometries. It gives the $j$-invariant of an elliptic curve
as a function of its $\lambda$-invariant; equivalently, it gives the
shape invariant of a set of 4 points on $\widehat{\mathbb{C}}$ as a
function of the cross-ratio of a corresponding list.

This map was studied by Cannon-Floyd-Parry
  \cite[Example 1.3.1]{CFP01:FiniteSubdiv}, where they showed that the
  sequence of tilings $\cT_n$ generated by iterated preimages is not
  conformal in Cannon's sense: there is no metric on the sphere
  quasi-conformally equivalent to the standard metric in which
  combinatorial moduli of curve families are comparable to analytic
  moduli. This is related to the fact that there are fixed critical
  points, which means the valence of $\cT_n$ blows up as $n \to
  \infty$. Cannon-Floyd-Kenyon-Parry investigated it further
  \cite[Figure 25]{CFKP03:Subdivision}, and proved that its Julia set
  is a Sierpiński carpet.  Haïssinsky and the second author studied
  it as well \cite[\S4.6]{kmp:ph:cxci}.

  To give an upper bound for the conformal dimension for the Julia set
  of this map, we
  estimate $\oE^2(f)$, which we know is less than one
\cite[Theorem 1]{Thurston20:Characterize}. To get a concrete estimate, we look at a finite
  stage and compute $E^2_2[\phi^n]$ for some~$n$. It turns out that
  $n=1$ is not enough to get an estimate less than~$1$, so we consider
  $\phi^2 \co G_2 \to G_0$, as shown in
  Figure~\ref{fig:barycenter-G2}. 
  \begin{figure}
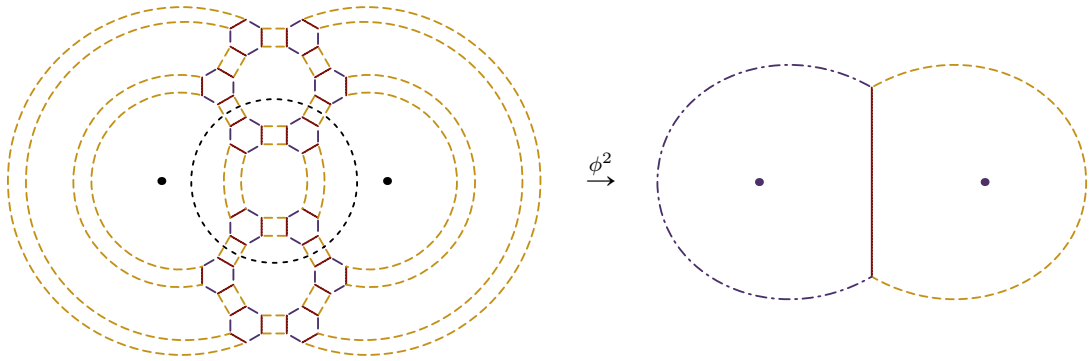

    \[
      \mfig{barycenter-3}
      \quad\overset{\phi^2}{\rightarrow}\quad
      \mfig{barycenter-0}
    \]
    \caption{The map $\phi^2 \co G_2 \to G_0$ for the barycentric
      subdivision map}
    \label{fig:barycenter-G2}
  \end{figure}

  Fix $q>1$.  We can take any $q$-conformal structure $G_0^q$ we like on~$G_0$.  It is most convenient to have $\alpha(e) = 1$
  for all three edges of~$G_0$.  We get a pulled-back
  structure $G_2^q$ on~$G_2$.  In this case, by symmetry, an optimal
  map will map the complete round dashed central circle on
  the left of Figure~\ref{fig:barycenter-G2} to the central edge on
  the right. (This circle passes through the midpoint of
  eight edges; the other halves of these edges of~$G_2$ map to
  different edges of~$G_0$.)

  Finding the map that minimizes $E^2_2$ from a graph to an interval
  (with specified boundary behavior) is equivalent to finding the
  resistance of a resistor network (or alternatively a harmonic
  function on the graph), and can be solved with standard linear
  algebra techniques. This is easily computed to be
  $E^2_2[\phi^2] = 10/13$, as shown in
  Figure~\ref{fig:barycenter-electric}.
  (The computation can be further simplified by using the
    symmetries evident in the figure.)
  Since the energy of any iterate gives an upper bound for the asymptotic
  energy \cite[Proposition 5.6]{Thurston20:Characterize}, we therefore have
  $\oE^2(f) \le \sqrt{10/13}$.

  \begin{figure}
    \[
      \mfig{barycenter-10}
    \]
    \caption{The central portion of the concrete map realizing
      $E^2_2[\phi^2]$ for the barycentric subdivision map. The edges
      running off the top and bottom are
      half-edges in the domain of the map. The numbers
      are the relative lengths of the images of each edge (or
      half-edge); to get the
      actual length, divide by $26$. The map itself is projection onto
      the vertical axis. }
    \label{fig:barycenter-electric}
  \end{figure}
  

\begin{proposition}\label{prop:barycenter-upper}
  For the barycentric carpet map $f$, we have
  \[
    \ARCdim(J_f) \le p^* \le \frac{2}{1 - \log_6(10/13)}<1.745.
  \]
\end{proposition}
\begin{proof}
  By definition, we have $\oE^{p^*}(f) = 1$. Apply
  Proposition~\ref{prop:Eq-lower-bound} with $q = 2$ to find
  \[
    \sqrt{10/13} \ge \oE^2(f) \ge 6^{1/2-1/p^*}
  \]
  which simplifies to the desired inequality after taking logs of both
  sides.
\end{proof}

Iterating further to better estimate $\oE^2(f)$ will improve the bound in
Proposition~\ref{prop:barycenter-upper}, although it probably will not
reach the optimal result, since Proposition~\ref{prop:Eq-lower-bound}
is not, in general, sharp.

To get a lower bound on conformal dimension, we compute $\oN(f)$.

\begin{proposition}
  For the barycentric subdivision rational map, $\oN(f) = 2$.
\end{proposition}

An application of Theorem \ref{thm:Nbar} then yields $\ARCdim(J_f) \geq 1/(1-\log_6(2))$, as claimed at the end of \S \ref{subsec:motivation}.

\begin{proof}
  By examining Figure~\ref{fig:barycenter-G1}, it is easy to find a
  map in $[\phi]$ for which the inverse image of every generic point
  is two points. Thus $\oN(f) \le N[\phi] \le 2$.

  To get the opposite inequality, we will find $2^n$
  edge-disjoint curves $\gamma_i^n$ on~$G_n$ so that the curves
  $\phi^n(\gamma_i^n)$ are all homotopic to a simple loop~$\gamma_0$
  on~$G_0$; cf.\ the proof of Theorem~\ref{taggedthm:Cprime}. This will
  immediately imply that $N[\phi^n] \ge 2^n$, as
  desired.

  We find the $\gamma_i^n$ by inductively finding $2^n$ edge-disjoint
  paths connecting any two edges of the dual of the $n\th$ barycentric
  subdivision of a triangle. This is trivial for $n=0$, and $n=1$ is
  shown in Figure~\ref{fig:barycenter-sides}. This also serves as the
  inductive step: at level~$n$, in each triangle replace the concrete
  paths in Figure~\ref{fig:barycenter-sides} with the family of $2^{n-1}$
  parallel paths constructed by induction.
  \begin{figure}
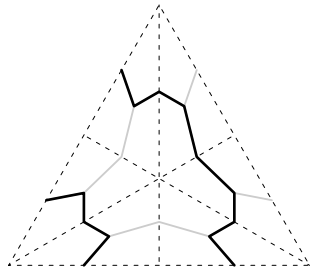

    \[
      \mfig{barycenter-23}
    \]
    \caption{Connecting sides of a triangle in the barycentric
      subdivision with edge-disjoint paths. This is the case $n=1$ and
      the inductive step.}
    \label{fig:barycenter-sides}
  \end{figure}

  The closed curves $\gamma^n_i$ are obtained by doubling the triangle
  as usual.
\end{proof}

\subsection{Fat real Devaney examples}
\label{sec:fat-devaney}
We begin by recalling some specifics concerning the Devaney family $f_\lambda(z)=z^2+\lambda/z^2$. 
The points $0$ and~$\infty$ are
always critical points satisfying 
$0 \overset{2}{\longrightarrow} \infty \overset{2}{\righttoleftarrow}$.
The other critical points
have orbits that start
\begin{equation}\label{eq:devaney-orbits}
  \lambda^{1/4} \overset{2}{\longrightarrow} \pm 2\sqrt{\lambda} \longrightarrow
  4\lambda+1/4\eqqcolon x_\lambda \longrightarrow \ldots.
\end{equation}
A sufficient condition for
$f_\lambda$ to be hyperbolic and have a Sierpiński
carpet Julia set is that $x_\lambda$ eventually iterates into the
Fatou component containing the origin. Such parameter values form a
countable collection of open disks called \emph{Sierpiński holes}.

As shown by the first author and R. Devaney
\cite{kmp:devaney:sierpinski}, given any finite word
$w=\epsilon_0\epsilon_1\ldots \epsilon_{n}$ in the alphabet $\{L,
R\}$, there exists a unique $\lambda=\lambda_w$, necessarily real and
negative, with the following property. For each $i=0, 1, \ldots, n$,
the image $f_\lambda^i(x_\lambda)$ lies to the left ($L$) or right
($R$) of the origin, according to the symbol in the $i\th$ position
of~$w$, and $f_\lambda^{n+1}(x_\lambda)=0$.
Let $\lambda_n^{\mathrm{skinny}}$ and $\lambda_n^{\mathrm{fat}}$
denote the parameter values corresponding to the words $LR^n$ and
$R^n$, respectively.

In this section we prove the second half of Theorem~\ref{thm:1_and_2}, dealing
with the fat Devaney family.
A virtual endomorphism spine for this rational map is shown in
Figure~\ref{fig:fat-spine} for $n=3$. As usual, the covering
map~$\pi$ preserves the decorations on the edges, and
the map~$\phi$ is the deformation retract onto~$G_0$ as a spine for
$S^2 - P$, where $P$ consists of the indicated points in the
diagram plus an extra point at~$\infty$.

The critical points at the fourth roots of~$\lambda$ are shown
schematically on the diagram of~$G_1$ with crosses; they map to the upper and
lower critical values on the diagram of~$G_0$, which in turn maps to the
sequence of points on the right of~$G_0$, eventually ending at the
central
critical point at~$0$, which maps to~$\infty$.
\begin{figure}
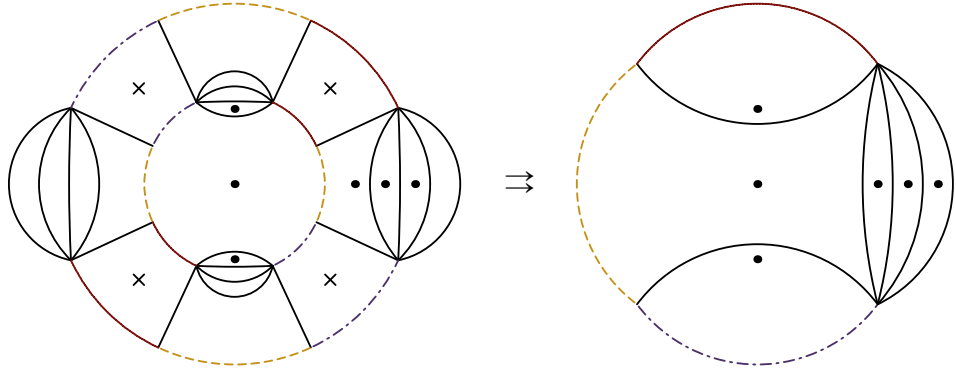

  \[
    \mfig{fat-1} \quad\rightrightarrows\quad\mfig{fat-0}
  \]
  \caption{Virtual endomorphism $G_1 \rightrightarrows G_0$ for the fat Devaney family, $R^n$, shown here with
    $n=3$. The post-critical set~$P$ is shown with bullets. The
    crosses on the left are the pre-periodic critical points.}
  \label{fig:fat-spine}
\end{figure}

To bound the conformal dimension from below, we again find $\oN(f_\lambda)$
and use bounds on how quickly $\oE^q$ decreases as a function of~$q$.
To find $\oN$, consider the $n$-component multi-curve~$C$ shown on the
right of Figure~\ref{fig:fat-curves}. The curve $f^{-1}(C)$ is shown
on the left of Figure~\ref{fig:fat-curves}. (As indicated, it is easy
to find $f^{-1}(C)$ by using the fact that $G_1$ is a cover of $G_0$.)
Each component of $f^{-1}(C)$ covers one of the components of $C$ with
degree two. One of the components of $f^{-1}(C)$ is peripheral (the
outer one), and the others are all components of~$C$.
Note that we
have the following.
\begin{itemize}
\item $C$ is \emph{completely invariant} (in the sense of Selinger
  \cite{Selinger12:AugmentedTeich}): $f^{-1}(C) = C$, up to homotopy
  in $S^2 - P_f$ and dropping inessential or peripheral components.
\item $C$ is \emph{Cantor-type} (in the sense of Cui-Peng-Tan
  \cite{CPT16:Renorm}): for some iterate,
each component of $C$ has at least two preimages homotopic to itself.
\end{itemize}

\begin{figure}
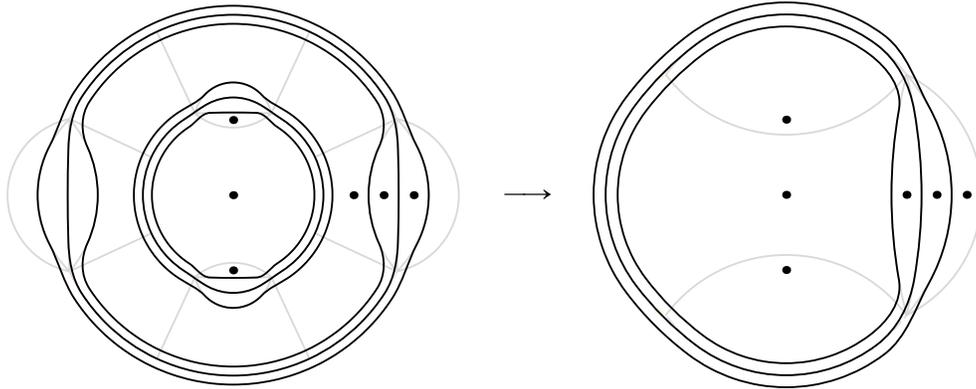

  \[
    \mfig{fat-11} \quad \longrightarrow\quad\mfig{fat-10}
  \]
  \caption{Curves in $S^2 - P$ for the fat Devaney family, and
    their inverse images. Each inverse image covers the original by
    degree~$2$.}
  \label{fig:fat-curves}
\end{figure}

\begin{lemma}\label{lem:fat-Nbar}
  For the Devaney family at parameter $\lambda_n^{\mathrm{fat}}$, let
  $r_n$ be  the largest root of
  $\lambda^{n+1}-2\lambda^n + 1$. Then
  $\oN(f_{\lambda_n^{\textrm{fat}}}) \ge r_n$.
\end{lemma}

\begin{proof} Let $\{\gamma_j\}$ enumerate the components of~$C$. 
  Consider the $n$-by-$n$ non-negative matrix~$A$ whose entry~$A_{ij}$ records how
  many components of $f^{-1}(\gamma_j)$ are homotopic to $\gamma_i$
  (without accounting for the degree of the cover). Concretely, we
  have
  \[
    A =
    \begin{pmatrix}
      1 & 1 & \cdots & 1 & 1\\
      1 & 0 & \cdots & 0 & 0\\
      0 & 1 & \cdots & 0 & 0\\
      \vdots & \vdots& \ddots & \vdots & \vdots \\
      0 & 0 & \cdots & 1 & 0
    \end{pmatrix}.
  \]
  Let $w$ be Perron-Frobenius eigenvector of~$A$, with
  eigenvalue~$\lambda$; concretely, $w_i = \lambda^{n-i-1}$ and
  $\lambda$ is the positive root of
  \[
    \lambda^n = \lambda^{n-1} + \lambda^{n-2} + \dots + \lambda + 1.
  \]
  Multiplying by $\lambda-1$ shows that $\lambda$ is the root $r_n$ given in the
  lemma statement. Informally, $\lambda$ is the rate of growth of the
  number of preimages of any of the curves $[\gamma_i]$.

  To see that $\lambda$ is a lower bound for $\oE^1(f)$, we specialize
  the discussion in \cite[\S7.5]{Thurston20:Characterize} to the
  present setting. Let $W_0^1 = \sum w_i [\gamma_i]$ be the
  weighted multi-curve on $G_0$ corresponding to~$w$.
  (Recall that the superscript ``$1$'' means we are thinking of this
  as a 1-conformal graph.) Set $W_1^1 = f^{-1}(W_0^1) = \pi^{-1}(W_0^1)$ be the
  pullback weighted curve, as a curve on $G_1$. (Here the pullback is
  as weighted curves, i.e., not taking into account the degree of
  the cover as in W.~Thurston's obstruction theorem.) By the
  eigenvalue property, $\phi(W_1^1)= \lambda W_0^1$. We similarly have
  pullback embeddings $W^1_n = (\pi^n)^{-1}(W^1_0)$ on $G_n$, so
  $\phi^n(W^1_n) = \lambda^n W^1_0$.
  
  Let $G_0^1$ be $G_0$ as a weighted graph, with all weights equal
  to~$1$. Using the interpretation of $E^1_1$ via stretch
  factors---that is, a supremum of ratios of
  extremal lengths---as in \S\ref{subsec:el}, we have
  \[
    E^1_1[\phi^n] \ge \frac{E^1_1[\phi^n(W^1_n) \to G^1_0]}{E^1_1[W^1_n\to G^1_n]}
    = \frac{E^1_1[\lambda^n W^1_0 \to G^1_0]}{E^1_1[W^1_n\to G^1_n]} = \lambda^n,
  \]

  as desired.
  \end{proof}

\begin{remark}
  It is straightforward to find weights on the edges of $G_0$
  to show that the inequality in Lemma~\ref{lem:fat-Nbar} is an
  equality, but the lower bound is all we need for our
  result.
\end{remark}

\begin{lemma}\label{lem:fat-root}
There exists a unique root $r_n$ of $g_n(\lambda) = \lambda^{n+1} - 2\lambda^n + 1$ in the interval 
$(1,2)$ and, for $n \ge 1$,
  \[
    r_n \ge 2 - 2^{-n+1}.
  \]
\end{lemma}
\begin{proof}
  We have $g_n(2) = 1 > 0$.   For $n=1,2,3,4$, we can directly check
  that $g_n(2-2^{-n+1}) \le 0$. For $n \ge 5$,
  \begin{align*}
    g_n(2 - 2^{-n+1}) &= 2^{n+1}\biggl(1-\frac{1}{2^n}\biggr)^{n+1}\!
                        - 2^{n+1}\biggl(1-\frac{1}{2^n}\biggr)^n+1\\
     &\le 2^{n+1}\biggl(1-\frac{n+1}{2^n}+\frac{n(n+1)}{2^{2n+1}}\biggr)
       - 2^{n+1}\biggl(1-\frac{n}{2^n}\biggr)+1\\
     &= -1 + \frac{n(n+1)}{2^n} < 0.
  \end{align*}
  It follows that there is a root of $g_n$ in $[2-2^{-n+1},2)$. Descartes' rule of signs
 shows there is a unique root in $(1,2)$.
\end{proof}

\begin{proof}[Proof of Theorem~\ref{thm:1_and_2}, fat family]
  By Theorem~\ref{thm:Nbar}, Lemma~\ref{lem:fat-Nbar},
  Lemma~\ref{lem:fat-root}, and elementary estimates,
  \[
    \ARCdim(f_{\lambda_n^{\textrm{fat}}})
    \ge \frac{1}{1-\log_4(\oN(f_{\lambda_n^{\textrm{fat}}}))}
    \ge \frac{2}{2-\log_2 r_n}
    \ge \frac{2}{1+2^{-n}}. \qedhere
  \]
\end{proof}

We sketch two alternate proofs of the same lower bound on
$\ARCdim(J(\lambda_n^{\mathrm{fat}}))$, using the same multi-curve. Fix
$n$; we write $f=f_{\lambda_n^{\mathrm{fat}}}$, and set $J=J(f)$. We
equip $J$ with a visual metric, as in \S \ref{subsecn:visual}; it
belongs to the quasi-symmetric gauge of $f$, by Proposition
\ref{prop:elevator}.

Our first method is a variant of that employed in \cite{HP08:Obstruction}. It associates to $C$ a critical exponent for the combinatorial modulus of the family of curves homotopic to~$C$. This exponent is then a lower bound for $\ARCdim$. 
Abusing notation, let $[C]$ denote the family of curves in $J$ which
are homotopic in $\CCa-P_f$ to  some component $\gamma_j$ of the
indicated multi-curve $C=\{\gamma_j\}$.
Since each curve in the family is essential and non-peripheral, the diameters of elements of this family are bounded below, say by 
$\delta>0$, and so $[C] \subset \Gamma_\delta$.
Let $\mathcal{U}_m$, $m \in \mathbb{N}$, denote the sequence of coverings as in \S
\ref{sec:sandwich}; it is a collection of
snapshots, by Proposition \ref{prop:qp_from_stars}.
For $Q>1$ let $A_Q$  denote the matrix whose $(i,j)$ entry is the sum of
terms of the form
$d_{ijk}^{1-Q}$ where $d_{ijk}=\deg(\delta_k \to
\gamma_j)$ and the curves $\delta_k$ range over preimages of $\gamma_j$ homotopic to $\gamma_i$ in the complement of the post-critical set.  In our case these covering degrees are all equal to~$2$, 
so $A_Q=2^{1-Q}A$. Thus the
Perron-Frobenius eigenvalue $\lambda_Q$ of $A_Q$ is $2^{1-Q}r_n$ where $r_n$ is the root from Lemma \ref{lem:fat-root}; note
this is strictly decreasing in $Q$.
Setting this equal to $1$ and solving for $Q$ yields $Q_*=1+\log_2
r_n$ for the critical exponent.
Fix now $1<Q < Q_*$. Then  as $m\to \infty$
\[ \modulus_Q(\Gamma_\delta, \mathcal{U}_m) \geq \modulus_Q([C],\mathcal{U}_m) \gtrsim 1, \]
where the last inequality is the statement of \cite[Proposition
5.1]{HP08:Obstruction}.  By Corollary \ref{cor:critexp_J}, we have $ \ARCdim(J)>Q$ and hence $\ARCdim(J) \geq Q_*$. 

Our second alternate proof we present here as a sketch; the motivation
comes from \cite{kmp:ph:ex}. Associated to the multi-curve $C$ is a
holomorphic virtual endomorphism of spaces
$\pi_Y,\phi_Y \co Y_1 \rightrightarrows Y_0$ where $Y_0$ is a
collection of open Euclidean right annuli of circumference~$1$ (and geodesic
boundary) indexed by the components of~$C$ and $Y_1$ is a
collection of pairwise disjoint right Euclidean sub-annuli of~$Y_0$
indexed by the components of $f^{-1}(C)$ homotopic to~$C$. We require
that $\phi_Y$ induces an inclusion $\overline{Y_1} \hookrightarrow
Y_0$ conformal on the interior and that $\pi_Y\co Y_1 \to Y_0$ is conformal, a local expanding homothety in
the Euclidean coordinates with constant factor $2$, and with each
component mapping by degree $2$.

Associated to this conformal expanding
dynamical system is a non-escaping set $X\subset Y_0$ and a self-map $g\co X \to X$.
The set $X$ is isometric to
a product $S^1 \times \mathcal{C}$, where $S^1$ is the Euclidean
circle of circumference $1$, and $\mathcal{C}$ is a Cantor set
associated to a graph-directed iterated function system on a disjoint
union of $\# C$ copies of $S^1$, with contraction maps having factor $2$,
and where the copies map according to the combinatorics of the map
$f^{-1}(C) \to C$. The Hausdorff dimension of~$\mathcal{C}$
is equal to $\log_2 r_n$, by a variant of the well-known ``pressure
formula''.  It follows from \cite[Proposition 4.1.11]{MT10:ConfDim} that the conformal dimension of $X$ is then equal to $1+\log_2 r_n=Q_*$.
There is a natural, non-surjective semiconjugacy $X \to J$ from $g$ to $f$.  

If 
this semiconjugacy were a homeomorphism, monotonicity of conformal dimension would
imply the desired lower bound on $\ARCdim(J)$; however, it is not. 
Cui, Peng, and Tan \cite{CPT16:Renorm}
show that there is a ``thick'' subset of $X$---the components living
over ``buried'' points in the Cantor set $\mathcal{C}$---on which the
semiconjugacy is injective; this should imply the
desired lower bound. For example, passing to some high iterate, and
deleting the extreme inner and outermost branches of the interval contraction mappings 
defining the corresponding Cantor set, one obtains a
sub-system whose repellor maps injectively to~$J$ and whose
dimension is close to that of the original system~$X$.

\begin{remark}
  It is challenging, using our techniques, to give a concrete upper bound
  estimate on $\ARCdim$ that is less than~$2$. Although we know
  that $\oE^2(f) < 2$, at some iterate the actual energy
  must be less than 2, and at that iterate we could apply
  Proposition~\ref{prop:Eq-lower-bound} to get an upper bound on
  $\ARCdim$. But
  it appears that we have to iterate quite a lot
  to get to these values and get a good upper bound on $\ARCdim$.
  In some sense, since the Julia set in these examples is a Sierpiński
  carpet of Hausdorff dimension close to~$2$, it is not
  well-approximated by graphs.
\end{remark}

\subsection{Skinny Devaney examples}
\label{sec:skinny-devaney}

We now turn to the other half of Theorem~\ref{thm:1_and_2}, dealing
with the skinny Devaney family. Suitable spines in this case are shown in
Figure~\ref{fig:skinny-spine} for $n=3$ (kneading sequence $LR^3$);
again, the generalization is evident.

We will find an explicit~$q \in (1,2)$, a $q$-conformal structure $G_0^q$ given by a metric~$\alpha$ on~$G_0$, and
map $\phi \co G^q_1 \to G^q_0$ so that $E^q_q(\phi) = 1$. The
metric~$\alpha$ is shown on the right of
Figure~\ref{fig:skinny-spine}, except that we have not yet determined
the value~$x$. The map $\phi$ is indicated schematically on the left
of Figure~\ref{fig:skinny-spine}: each region surrounded by a green
loop is contracted to a point, and $\phi$ is optimized in the
remaining regions.

\begin{figure}
  \[
    \mfig{skinny-2}\quad\rightrightarrows\quad
    \mfig{skinny-10}
  \]
  \caption{Virtual endomorphism $G_1 \rightrightarrows G_0$ for the skinny Devaney family, $LR^n$, shown here
    with $n=3$. The map $\pi$ preserves color and orientation of the
    plane, as usual. The map $\phi$ collapses the regions within green
  dotted circles to points.}
  \label{fig:skinny-spine}
\end{figure}
Inspection reveals that for most points $y \in G^q_0$, there is a unique
point $x \in G_1$ with $\phi(x) = y$ and $\abs{\phi'(x)} = 1$,
compatible with $E^q_q(\phi) = 1$. The exceptions are:
\begin{enumerate}
\item\label{item:chain1} On the long left edge, there is a chain in~$G_1$ of
  $n$ ``double edges'' (two edges connecting the same vertices), with
  each edge of length 1:
  \[
    \includegraphics{skinny-20}
  \]
   (Note the right
   vertical edge gets mapped to a point.)
   This chain maps to an edge in~$G^q_0$ of length~$x$.
\item\label{item:chain2} On each the four edges of the central square,
  there a
  chain in $G^q_1$ of an edge of length~$x$, two parallel edges of
  length~$2$, another edge of length~$x$, and two parallel edges of
  length~$1$:
  \[
    \includegraphics{skinny-21}
  \]
  (Note the right-hand endpoints map to the same point.)
  These chains each map to an edge in~$G_0$ of length~$1$.
\end{enumerate}
We now solve for $x$ and~$q$ to make the supremand in the definition
of $E^q_q(\phi)$ equal to~$1$ on these edges as well. We use the
principle that two parallel edges of length $a$ and~$b$ in
a $q$-conformal graph can be replaced by a single edge of length
\[
  a \oplus_q b \coloneqq \bigl(a^{1-q} + b^{1-q}\bigr)^{1/(1-q)}
\]
without changing the optimal $E^q_q$ energy of any map. (See
\cite[Proposition 7.7]{Thurston20:Characterize}. The case $q=2$ is the
standard parallel law for resistors.)

Thus, the $n$ double edges in \eqref{item:chain1} have an effective
$q$-length of
\[
  x \coloneqq n \cdot (1 \oplus_q 1) = n\cdot 2^{1/(1-q)}.
\]
We set $x$ to this value to make $E^q_q(\phi) = 1$.

The chain of edges in \eqref{item:chain2} have an effective $q$-length
of
\[
  x + (2 \oplus_q 2) + x + (1 \oplus_q 1) = 2x + 3\cdot 2^{1/(1-q)} =
  (2n+3)\cdot 2^{1/(1-q)}.
\]
For $q = 1 + 1/\log_2(2n+3)$, this quantity is equal to~$1$, which
makes $E^q_q(\phi) = 1$, with the supremum in the definition of
$E^q_q$ achieved everywhere on~$G_0$.

\begin{proposition}
  For the skinny Devaney example $\lambda_n^{\mathrm{skinny}}$, for
  any $q > 1 + 1/\log_2(2n+3)$, we can find a metric on $G_0$ in
  Figure~\ref{fig:skinny-spine} so that $E^q_q(\phi) < 1$. Thus
  $\oE^q[\pi,\phi] < 1$ and
  \[
    \ARCdim(J_{\lambda_n^{\mathrm{skinny}}}) \le 1 + \frac{1}{\log_2(2n+3)}.
  \]
\end{proposition}

\begin{proof}
  For $q$ bigger than $1+1/\log_2(2n+3)$, it is straightforward to
  modify the metric on~$G_0$ in Figure~\ref{fig:skinny-spine} by
  adjusting the lengths slightly to make $E^q_q[\phi] < 1$. (For
  instance, on the chain of loops on the right side, multiply the
  length around the $k\th$ dot in from the end by $(1+\epsilon)^k$ for
  small $\epsilon$.)
  Then for these $q$, we have $\oE^q[\pi,\phi] \le E^q_q[\phi] < 1$
  and $\ARCdim < q$, as desired.
\end{proof}

\begin{remark}
  One can also give a lower bound on
  $\ARCdim(J_{\lambda_n^{\mathrm{skinny}}})$ by estimating $\oN[\pi,\phi]$.
\end{remark}

\subsection{Uniformly well-spread cut points}
\label{subsecn:uwscp}

\begin{definition}
\label{def:linearly_connected}
A metric space $X$ is \emph{linearly connected} if there exists a constant $L\geq 1$ such that for each pair of points $x, y \in X$, there is a continuum $E$ containing $\{x,y\}$ such that $\diam E \leq Ld(x,y)$.
\end{definition}

If $X$ is connected and $f\co X \to X$ is metrically cxc, then $X$ is linearly connected \cite{kmp:ph:cxci}. 

\begin{definition}
  \label{def:uwscp}
A compact, connected metric space $X$ is said to satisfy the \emph{uniformly well-spread cut points condition} (UWSCP) if there exists a constant $C \geq 1$ such that for each $x \in X$ and each $r>0$, there exists a set $A \subset X$ with $\#A \leq C$ such that no component of $X - A$ meets both $B(x,r/2)$ and $X - \overline{B}(x,r)$. 
\end{definition}

The following result is \cite[Theorem 1.2]{C14:ConfDim}.
\begin{citethm}
\label{thm:confdim1}
If $X$ is doubling, compact, connected, linearly connected, and satisfies the UWSCP condition, then $\ARCdim(X)=1$. 
\end{citethm}

\subsection{Matings}
\label{sec:mating-examp}

In this section, we give an example showing that among Julia
sets of hyperbolic rational functions, the UWSCP condition is sufficient but
not necessary for $\ARCdim(J)=1$. To our knowledge,
this is the first result of its kind. We start by recalling
generalities on matings.

\subsubsection*{Formal mating} Here, we denote by $\mathbb{S}^1:=\mathbb{R}/\mathbb{Z}$.  For an integer $d \geq 2$, let $\tau_d: \mathbb{S}^1 \to \mathbb{S}^1$ be given by $\tau_d(\theta)=d\cdot \theta$ modulo $1$. The action of a degree $d$ monic polynomial on the complex plane may be compactified by extending it to the circle at infinity $\mathbb{S}^1_{f,\infty}$ to give an action on a closed topological 2-disk whose boundary values are given by $\tau$. 

The operation of \emph{formal mating}
takes as input two monic critically finite polynomials $f_1, f_2$ of
the same degree $d \geq 2$, and returns as output a
topological Thurston map $f_1 \amalg f_2\co S^2 \to S^2$.  The sphere $S^2$ on which $f_1 \amalg f_2\co S^2 \to S^2$ acts is 
obtained by gluing together $\mathbb{S}^1_{f_1,\infty}$ to $\mathbb{S}^1_{f_2,\infty}$
via the map $\theta \mapsto -\theta$.  Thus $f_1 \amalg f_2$ preserves a natural 'equator circle' on which it acts via $\tau_d$. 
 We refer the reader to \cite{BEKMPRL12:PolyMatings} for a survey containing facts
mentioned below. Our focus here is exclusively on the case when the
$f_i$ are hyperbolic; this assumption simplifies the discussion.
Abusing terminology, given a hyperbolic critically finite
polynomial~$f$, we call a bounded Fatou component a \emph{basin}
of~$f$.

\subsubsection*{Ray equivalence relation and geometric mating} The
sphere $S^2$ on which $f \amalg g$ acts comes equipped with a natural
invariant ``ray-equivalence''  relation, $\sim_{\mathrm{ray}}$. If $R\co
\widehat{\mathbb{C}} \to \widehat{\mathbb{C}}$ is a rational map such
that quotienting by the ray equivalence relation yields a
continuous semiconjugacy $\rho\co S^2 \to \widehat{\mathbb{C}}$ from $f \amalg g$ to $R$ that is conformal on the basins of the $f_i$, we say
$f_1, f_2$ are \emph{geometrically mateable} and that $R$ is the \emph{geometric}
mating of $f_1$ and $f_2$. See Figure
\ref{fig:mate-rabbit-basilica}.\footnote{According to J. Milnor, geometric mating is ``interesting, since it is neither well-defined, injective, surjective nor continuous.''}
The restriction $\rho\co \mathbb{S}^1_\infty \to J(R)$ is surjective; its fibers are the intersections of the ray-equivalence classes with $\mathbb{S}^1_\infty$. It is interesting to ask
how big the fibers can be.  

The following is a special case of a
general principle, to our knowledge first articulated by D. Fried
\cite{Fried87:FPDynamical}: that expanding dynamical systems on reasonable
spaces are ``finitely presented'', in the following precise sense:
they are quotients of an expanding subshift of finite type, and there
is another subshift which encodes when two points lie in the same
fiber.  Restricted to the case of matings, we were surprised not to find the following  result in the literature;
c.f.\ the survey \cite[Remark 4.13]{PM12:Mating} where the possibility of
ray equivalence classes of unbounded size is entertained. Explicit
bounds for the closely related question of the length of ray
connections are established in certain cases by W. Jung
\cite{Jung17:QuadMating}.  

\begin{proposition}
\label{prop:reqreln}
Suppose $f_1, f_2$ are monic hyperbolic critically finite polynomials of degree $d \geq 2$ that
are geometrically mateable, with geometric mating $R$. 
Then there exists a constant $C=C(f_1,f_2)$ such that fibers of $\rho\co \mathbb{S}^1_\infty \to J(R)$ have cardinality at most $C$.
\end{proposition}

For the proof, we apply a general result of Nekrashevych.  To motivate the technique, and because we need it anyway, we begin by presenting a construction of a semiconjugacy $\pi\co \Sigma_d \to \mathbb{S}^1_\infty$ from the full 1-sided shift on $d$ symbols to the map $\tau_d$. We then analyze the composition 
\[ \Sigma_d \overset{\pi}{\longrightarrow}\mathbb{S}^1_\infty \overset{\rho}{\longrightarrow}J(R). \]

\begin{proof}
We take $x_0\coloneqq\infty e^{2\pi i 0}$ as a basepoint, and denote by $x_k$
for $k=0, \ldots, d-1$ its $d$ preimages under the map $\tau_d$. There is a canonical positively
oriented (counterclockwise) arc $\alpha_k \subset \mathbb{S}^1_\infty$  joining the
basepoint $x_0$ to each $x_k$.  Fix an iterate $n \geq 1$.
Lifting the $\alpha_k$ under $\tau_d, \tau_d^{\circ 2}, \ldots, \tau_d^{\circ n}$ and concatenating the lifts gives an identification of $\tau_d^{-n}(x_0)$ with $\{0, 1, \ldots, d-1\}^n$ given by taking endpoints of iteratively concatenated lifts. The lengths of the concatenations are uniformly bounded, since the $n$th stage has length bounded by a convergent geometric series with ratio $1/d$, and passing to the limit we obtain the desired semiconjugacy $\pi\colon \Sigma_d \to \mathbb{S}^1_\infty$ from the shift to $\tau_d$.
(More simply, we can also see this semiconjugacy a writing an element in
$\RR/\ZZ$ in base~$d$.)

Our strategy is to now ``push'' this construction down to the
dynamical plane of $R$ via the semiconjugacy $\rho$. Denote by $P$ the
post-critical set of $R$. For each $k=0, 1, \ldots, d-1$ let $\beta_k$
be a smooth path in $\rs - P$ homotopic relative to endpoints to the
arc $\rho(\alpha_k)$. Note that since $R$ is hyperbolic, it is
expanding outside of a neighborhood of $P$. Applying the same
iteratively-lifting-and-concatenating construction from the previous
paragraph using the
$\beta_k$ and $R$ in place of the $\alpha_k$ and $\tau_d$, we get a
well-defined  composition
$\rho\circ\pi\co \Sigma_d \to J(R)$ induced by taking endpoints of
infinitely iterated concatenated lifts of the $\beta_k$.

Nekrashevych \cite[Proposition
3.6.2]{nekrashevych:book:selfsimilar} shows that in this setting, there is a finite automaton,
called the \emph{nucleus}, whose one-sided infinite paths encode the
equivalence relation on $\Sigma_d$ identifying points in the fiber of the composition $\rho \circ \pi$. It follows that the size of these equivalence classes are uniformly bounded by the size of the nucleus. The fibers of $\rho$ are no larger than those of $\rho \circ \pi$, yielding the result. 
 \end{proof}

See also \cite[\S6.13]{nekrashevych:book:selfsimilar}, in which details for a specific example of mating are presented.

\subsubsection*{Quadratic matings}

The following theorem is due to Tan Lei \cite{Tan92:Matings} and M.
Shishikura \cite{Shishikura00:ReesMatings}, and was proven using ideas
of M. Rees.
\begin{citethm}
\label{thm:r-s-tl}
Two critically finite quadratic polynomials $f_1, f_2$ are mateable if and only if they do not lie in conjugate limbs of the Mandelbrot set.
\end{citethm}

To a hyperbolic critically finite quadratic polynomial $f$ is
associated a unique nontrivial interval
$[a,b] \subset \mathbb{R}/\mathbb{Z}$: the external rays of angles $a$
and $b$ land at a common periodic point on the boundary of the
immediate basin $U$ containing the critical value of $f$, and separate
$U$ from all other periodic attracting basins. For the basilica,
$[a,b]=[1/3, 2/3]$, while for the airplane, $[a,b]=[3/7, 4/7]$. The
denominators of $a$ and~$b$ take the form $2^n-1$ where $n$ is the period of the
finite critical point.

\subsubsection*{Generalized rabbits}

If the hyperbolic component containing $f$ has closure meeting the
main cardioid component containing $z^2$, we call $f$ a
\emph{generalized rabbit}. In this case, one may encode $f$ by
rational numbers in a different way.

For each $p/q \in \mathbb{Q}/\mathbb{Z} - \{0\}$, there is a unique hyperbolic critically finite quadratic polynomial $f=f_{p/q}$ such that there are $q$ periodic basins meeting at a common repelling fixed-point $\alpha$, and such that the dynamics on the set of these $q$ periodic basins, when equipped with the natural local cyclic ordering near $\alpha$, is given by a rotation with angle $p/q$. This fact can be deduced from the classification of critically finite hyperbolic quadratic polynomials via their so-called invariant laminations; see \cite{MR2508255}.  The basilica polynomial is $f_{1/2}$, while the rabbit polynomial is $f_{1/3}$. 

A special feature of generalized rabbit polynomials is the following.
If $f$ is a generalized rabbit polynomial and $\theta_1 \sim \theta_2$
is any nontrivial ray-equivalent pair of angles, so that the
corresponding rays land on a common point $z$ in the Julia set of $f$,
then the point~$z$ is on the boundary of a
basin of~$f$. This fact need not hold for other pcf polynomials like
the airplane.

\begin{proposition}
\label{prop:not_touch}
  For $i=1,2$ suppose $p_i/q_i \in \mathbb{Q}/\mathbb{Z}-\{0\}$ and
  suppose $q_1, q_2$ are coprime. Let $f_i$ be the corresponding
  generalized rabbit quadratic polynomials, and $R$ the geometric
  mating of $f_1$ and~$f_2$.
  Then the basins of $f_1$ and $f_2$  do not touch in the Julia set of $R$.
\end{proposition}

Note that the mating exists by Theorem~\ref{thm:r-s-tl}.
In the proof below applied to the mating of the basilica and rabbit,
the key observation is that the intervals $[1/3, 2/3]$ and $[1-2/7,
1-1/7]$ are disjoint.

\begin{proof}
We argue by contradiction. Let $U_i$ be the immediate basin containing
the finite critical value of $f_i$. Suppose $z \in J(R)$ is a point
that is simultaneously on the boundary of a basin for~$f_1$ and a
basin for $f_2$. Since, by Proposition~\ref{prop:reqreln}, the ray
equivalence classes are finite and the
$f_i$ are generalized rabbits, this implies there exists an angle
$\theta$ such that the ray of angle $\theta$ for $f_1$ lands on a
point $z_1$ on the boundary of a basin of $f_1$, and the ray of angle
$-\theta$ for $f_2$ lands on a point $z_2$ on the boundary of a basin
of~$f_2$.
By iterating this ray-pair forward under the formal mating, we may assume that $z \in \partial U_1$. Since $q_1$ and $q_2$ are co-prime, by passing to a further iterate, we may assume that $z \in \partial U_1 \cap \partial U_2$. 

In the dynamical plane of $f_i$, let $(a_i,b_i)$ be the ray-pair of angles
landing at the point $\alpha_i$ separating $U_i$ from the remaining
immediate attracting basins of~$f_i$, as recalled above.
In giving coordinates for the circle at infinity $\RR/\ZZ$, we
parameterize to agree
with the usual counterclockwise orientation on~$f_1$ and disagree
for~$f_2$, so that the coordinates match when we mate.
For example, if $f_1$ is the
basilica and $f_2$ the rabbit, then $(a_1, b_1)=(1/3, 2/3)$ and
$(a_2, b_2)=(1-2/7, 1-1/7)$.
Then the set of angles of rays
landing on $\partial U_i$ is a subset of $[a_i, b_i]$. The condition
that $q_1$ and~$q_2$ are coprime implies that $f_1$ and $f_2$ are not
in conjugate limbs
of the Mandelbrot set, and hence that $[a_1, b_1]$ and
$[a_2,b_2]$ are disjoint. But the previous paragraph shows
$\theta$ belongs to both of these intervals. This is impossible.
\end{proof}

Proposition~\ref{prop:rmb} is then an immediate consequence of the
following two Propositions.

\begin{proposition}
\label{prop:not_uwscp}
Suppose $f_1$, $f_2$, and~$R$ are as in Proposition \ref{prop:not_touch}.
Then the Julia set $J_R$ of $R$ does not satisfy UWSCP. In particular, the Julia set of the basilica mated with the rabbit does not satisfy UWSCP. 
\end{proposition}

\begin{proof}
  Suppose $J_R$ satisfies UWSCP, and $C$, $x$, $r$, and~$A$ be any
  data as in Definition~\ref{def:uwscp}. Then there exists a
Jordan curve $\gamma$ surrounding $x$ meeting $J_R$ only in the finite
set~$A$. Since the basins of the $f_1$ and $f_2$ do not touch, the
loop
$\gamma$ lies in the closure of finitely many immediate basins of one
of the $f_i$, say $f_1$.

Since $J_R$ is the common boundary of the  basins of $f_1$
and of $f_2$, each component of
$\overline{\mathbb{C}} - \gamma$ contains a basin of $f_2$, say
$U$ and $V$. On the one hand, $U$ and $V$ are joined by a path $\beta$
lying in the closure of finitely many basins of $f_2$ and meeting
$J(R)$ in finitely many points, since this is true in the dynamical
plane of $f_2$. On the other hand, any such path intersects~$\gamma$
and thus passes through a point of~$A$. It
follows that the basins
of $f_1$ and $f_2$ touch in this point of~$A$. By Proposition
\ref{prop:not_touch}, this is impossible.
\end{proof}

\begin{proposition}\label{prop:rabbit-basilica}
  The virtual endomorphism~$F_{RB}$ of the mating of the rabbit and
  the basilica satisfies $\overline{N}[F_{RB}] = 1$ and, for $q > 1$,
  $\oE^q[F_{RB}] < 1$, and so $\ARCdim(\cJ_{RB}) = 1$.
\end{proposition}

\begin{figure}
  \centerline{%
    \begin{tikzpicture}[x=2.5in]
      \node at (0,0) (G1) {$\mfig{matings-21}$};
      \node at (1,0) (G0) {$\mfig{matings-20}$};
      \draw[bend left,->] (G1.5) to node[above,cdlabel]{\pi} (G0.175);
      \draw[bend right,->] (G1.-5) to node[below,cdlabel]{\phi} (G0.-175);
    \end{tikzpicture}}
  \caption{Virtual endomorphism $G_1 \rightrightarrows G_0$ for the mating of the rabbit and the basilica}
  \label{fig:spinse-rabbit-basilica}
\end{figure}

\begin{proof}
  The spines for this mating are shown in
  Figure~\ref{fig:spinse-rabbit-basilica}, with a black equator and
  colored rays. We may take the map~$\phi$ in its homotopy class to be
  the map that ``pushes'' the extra colored edges towards the equator.
  If we do this suitably, each colored point in $G_0$ has one
  (colored) preimage, and each point on the equator has at most three
  preimages, one on the equator and two colored. It follows by
  iteration that $N(\phi^n) = 2n + 1$ and so
  $\overline{N}[F_{RB}] = 1$. For the statement about $\oE^q$, we
  proceed as in \cite[Example 2.4]{Thurston16:RubberBands}. Fix
  $q > 1$, and consider a metric on $G_0$ where the colored edges have
  equal length~$L$, and the equator has length $1$. Lifting this
  metric under the covering to a metric on $G_1$, and lifting the
  colors as well, each colored edge of $G_1$ has length $L$, and the
  equator has length $2$. For the map $\phi$ described above,
  $\Fill^q(\phi) = 1$ on the colored edges, and, for $L$ sufficiently
  large, $\Fill^q(\phi) \approx 2^{1-q} < 1$ by Eq.~\eqref{eq:Fillq}.
  We can homotop~$\phi$ to
  make $\Fill^q$ less than one everywhere by pulling the images of the
  vertices on the equator very slightly in along the colored edges;
  this decreases the derivative on the colored edge, while increasing
  the derivative on the equator (but keeping it less than $1$). Then
  $\oE^q[F_{RB}] \le E[\phi] < 1$, as desired.
\end{proof}

\begin{remark}
  The proof in Proposition~\ref{prop:rabbit-basilica} applies to the mating of any pair of rabbit-type polynomials as in the hypothesis of Proposition
  \ref{prop:not_touch}: all such matings have
  $\overline{N}[\pi,\phi]=1$.
  However, if we mate two polynomials that
  are far out in the limbs, we may end up with a Sierpiński carpet.
  Figure~\ref{fig:amk} gives empirical evidence for this assertion. It
  shows the Julia set of the result of mating of the airplane
  polynomial~$f_1$ and another polynomial~$f_2$, with
  corresponding angles $(a_1, b_1)=(3/7, 4/7)$ and
  $(a_2, b_2)=(3/31, 4/31)$. The mating appears to be a Sierpiński carpet.
  (This example was found by Insung Park and Caroline
  Davis.)
\end{remark}

\begin{figure}
\centerline{\includegraphics[width=2.75in]{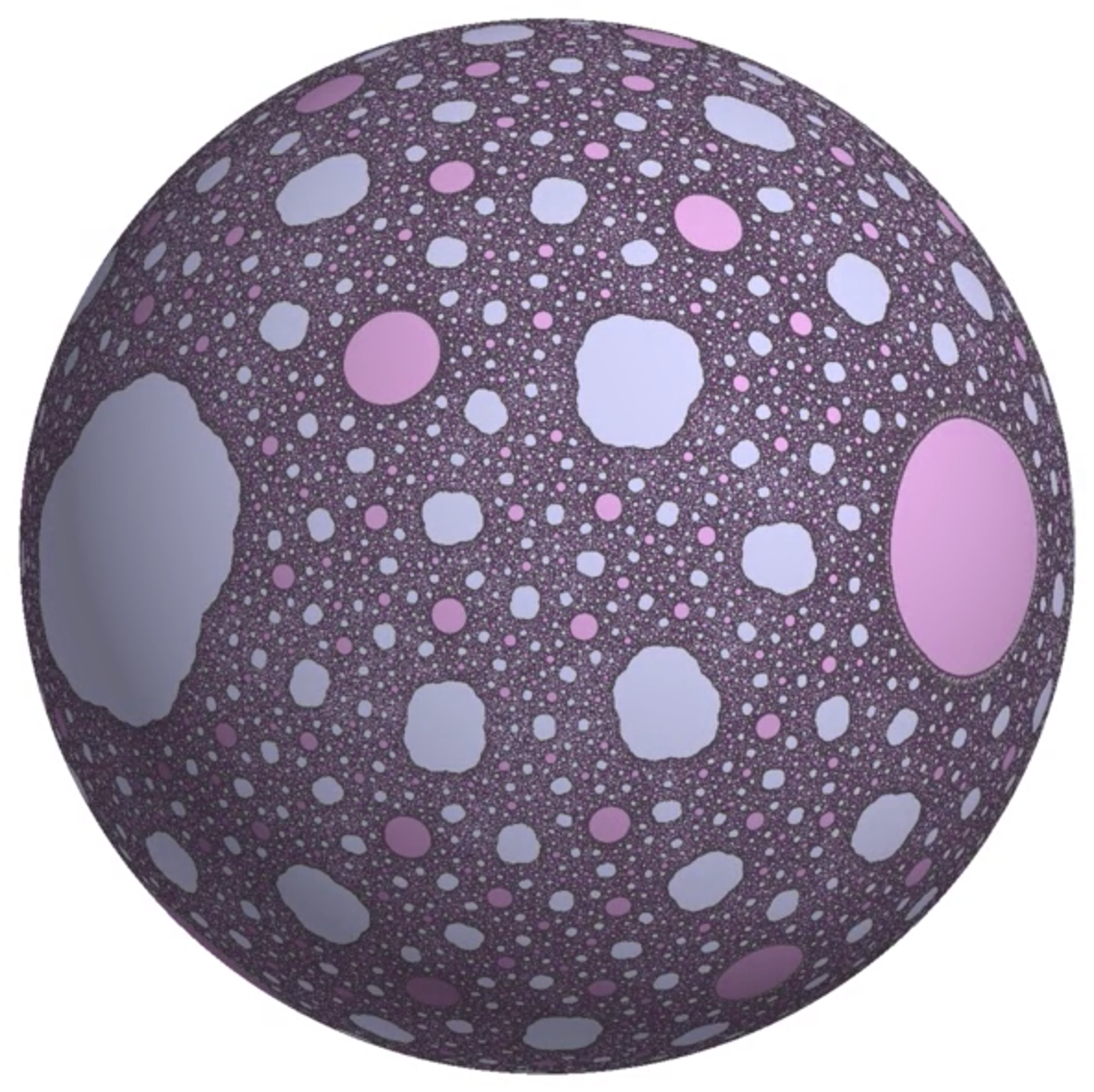}}
\caption{The Julia set of airplane mate a kokopelli-like polynomial.
  Picture by Arnaud Chéritat.}
\label{fig:amk}
\end{figure}

\begin{conjecture}
  The Julia set of a hyperbolic rational map satisfies
  UWSCP iff its virtual endomorphism has uniformly bounded $N[\phi^n]$
  (independent of~$n$).
\end{conjecture}


\bibliographystyle{hamsalpha}
\bibliography{ConfDim}

\end{document}